\documentclass[12pt]{amsart}

\usepackage[dvipsnames]{xcolor}
\usepackage{graphicx}
\usepackage[toc,page]{appendix}
\usepackage{chngcntr}
\usepackage{apptools}
\usepackage{bbm}
\usepackage{dsfont}
\usepackage{amsmath}
\usepackage{color}
\usepackage{caption}
\usepackage{subcaption}
\usepackage[T1]{fontenc}
\usepackage{mathtools}
\usepackage{amsfonts,amssymb,amscd}
\calclayout

\newtheorem{thm}{Theorem}[section]
\newtheorem{prop}[thm]{Proposition}
\newtheorem{lem}[thm]{Lemma}
\newtheorem{cor}[thm]{Corollary}

\theoremstyle{definition}
\newtheorem{definition}[thm]{Definition}

\newtheorem{notation}[thm]{Notation}

\newcommand{\cC}{\mathcal{C}} % right half-plane
\theoremstyle{remark}

\newtheorem{rem}[thm]{Remark}

\theoremstyle{question}

\numberwithin{equation}{section}

\newcommand{\C}{\mathbb{C}}  % The complex numbers.
\newcommand{\Z}{\mathbb{Z}}  % The integers numbers.
\newcommand{\N}{\mathbb{N}}  % The natural numbers.
\newcommand{\Julia}{\mathcal{J}}

\newcommand{\ra}{\rightarrow}

\DeclareMathOperator{\diam}{diameter}

\DeclareMathOperator{\length}{length}

\AtAppendix{\counterwithin{thm}{section}}

\DeclareMathOperator{\dist}{dist} % The distance.

\begin{document}

	%%
	%% The title of the paper goes here.  Edit to your title.
	%%
	
	\title[Dimension 1]{Transcendental Julia sets of Minimal Hausdorff Dimension}
	
	%%
	%% Now edit the following to give your name and address:
	%% 
	
	\author{Jack Burkart and Kirill Lazebnik}
	
	%%
	%% If there is another author uncomment and edit the following.
	%%
	
	%\author{Second Author}
	%\address{Department of Mathematics, University of South Carolina,
	%Columbia, SC 29208}
	%\email{second@math.sc.edu}
	%\urladdr{www.math.sc.edu/$\sim$second}
	
	%%
	%% If there are three of more authors they are added in the obvious
	%% way. 
	%%
	
	%%%
	%%% The following is for the abstract.  The abstract is optional and
	%%% if not used just delete, or comment out, the following.
	%%%
	
	\maketitle
	
	\begin{abstract} We show the existence of transcendental entire functions $f: \mathbb{C} \rightarrow \mathbb{C}$ with Hausdorff-dimension $1$ Julia sets, such that every Fatou component of $f$ has infinite inner connectivity. We also show that there exist singleton complementary components of any Fatou component of $f$, answering a question of \cite{RSEremenkoPoints}. Our proof relies on a  quasiconformal-surgery approach developed in \cite{BurLaz}.
	\end{abstract}
	
	%%
	%% LaTeX can automatically make a table of contents.  This is done by
	%% uncommenting the following:
	%%
	
	\tableofcontents
	
	%%
	%%  To enter text is easy.  Just type it.  A blank line starts a new
	%%  paragraph. 
	%%
	\section{Introduction}
	
	The \emph{Julia set} of an entire function $f: \mathbb{C} \rightarrow \mathbb{C}$, denoted by $\mathcal{J}(f)$, is the set of points at which the dynamical system $(f, \mathbb{C})$ behaves chaotically. The behavior of $f$ near $\infty$ plays an important role, and one has the following dichotomy. Either $\infty$ is a removable singularity, in which case $f$ is a polynomial, or $\infty$ is an essential singularity, in which case $f$ is a transcendental entire function. 
	
	When $f$ is a polynomial, $\mathcal{J}(f)$ is usually small in the sense of \emph{Hausdorff dimension}. Hausdorff dimension is the most well-studied measure of size for Julia sets, and this is the measure we will study in this manuscript.  For instance, one has that the quadratic polynomial $p_c(z):=z^2+c$ satisfies $\textrm{dim}(\mathcal{J}(p_c))<2$ for generic $c\in\mathbb{C}$ (see for instance \cite{MR1279476}), although there exist parameters $c$ satisfying $\textrm{dim}(\mathcal{J}(p_c))=2$ \cite{MR1626737}, and even $\textrm{area}(\mathcal{J}(p_c))>0$ \cite{MR2950763}, \cite{avila2015lebesgue}.

	%	The most well-studied notion of ``small'' and ``large'' Julia sets is that of \emph{Hausdorff dimension}, and this is the notion we will use in this manuscript. For instance, one has that for a.e. $c\in\mathbb{C}$, the quadratic polynomial $p_c(z):=z^2+c$ satisfies $\textrm{dim}(\mathcal{J}(p_c))<2$ {\color{red}(not actually sure if this is true)}, but there exist parameters $c$ along the boundary of the Mandelbrot set satisfying $\textrm{dim}(\mathcal{J}(p_c))=2$. 
	
	On the other hand, when $f$ is a transcendental entire function, the generic situation is that $\textrm{dim}(\mathcal{J}(f))=2$. For instance, in \cite{Mis} it was shown that $\mathcal{J}(e^z)=\mathbb{C}$, and in \cite{Mc} it was shown that functions in certain standard exponential and sine families have Julia sets of dimension $2$. Thus, in contrast with the polynomial setting, the difficulty in the transcendental setting is to find Julia sets of small dimension, a problem whose history we overview briefly now. 
	
	In \cite{Baker} it was proven the Julia set of any transcendental $f$ must contain a non-trivial continuum, and hence we always have $\textrm{dim}(\mathcal{J}(f))\geq1$. In the class of transcendental $f$ with bounded singular set, denoted $\mathcal{B}$, it was shown in \cite{S1}, \cite{S2}, \cite{S4} that  \[ \{ \textrm{dim}(\mathcal{J}(f)) : f \in \mathcal{B} \} = (1,2], \] (see also \cite{AlbBis}). Finally, in \cite{Bis18}, it was proven that outside of the class $\mathcal{B}$, the lower bound of $1$ in the inequality $\textrm{dim}(\mathcal{J}(f))\geq1$ is actually attained (see also \cite{BurPack}, \cite{XuZhang}). Our main result (see Theorem \ref{main_theorem} below) also achieves this lower bound, but by different methods, and with different resulting dynamical properties which we now discuss.
	
	The Fatou set of the function in \cite{Bis18} is in fact completely described: it consists of a collection of \emph{multiply connected wandering domains}, abbreviated m.c.w.d.. This class of Fatou components has been well-studied \cite{MR2400399}, \cite{KisShishAnnulus}, \cite{BergZheng}, \cite{BergRipStalWD}, \cite{MR2900166}, \cite{MR3422682}, \cite{RSEremenkoPoints}, \cite{ferreira2021multiply} and appears in several different contexts in transcendental dynamics. A m.c.w.d. $U$ of \cite{Bis18} consists of a topological annulus minus countably many discs which accumulate only on the outer boundary of $U$ (see Figure \ref{inner_outer_connectivity}(A)). This topological structure is aptly termed \emph{infinite outer connectivity} (defined precisely in \cite{BergRipStalWD}). The Fatou components of the function in our Theorem \ref{main_theorem} are also all m.c.w.d.'s, however they have \emph{infinite inner connectivity} (see Figure  \ref{inner_outer_connectivity}(B)), and we prove they have the following more intricate topological structure:

	%	have so-called \emph{infinite outer connectivity} (defined precisely in \cite{BergRipStalWD}). This roughly means that the complementary components of $U$ accumulate on the ``outermost'' boundary component of $U$ (see Figure \ref{inner_outer_connectivity}(A)). The Fatou components of the function in our Theorem \ref{main_theorem} are also all m.c.w.d.'s, however they have \emph{infinite inner connectivity} (see Figure  \ref{inner_outer_connectivity}(B)), and we prove they have the following more intricate topological structure:
	
	%	 \emph{infinite inner connectivity}, and a more intricate 
	
	%	 a more intricate topological structure: they have \emph{infinite inner connectivity}, and the complementary components consist both of 
	
	%	. This roughly means the complementary components of $U$ accumulate on the ``innermost'' boundary component of $U$ (see Figure \ref{inner_outer_connectivity}(B)).
	
	\begin{figure}[!h]
		\centering
		\scalebox{.4}{%% Creator: Inkscape 1.0.1 (c497b03c, 2020-09-10), www.inkscape.org
%% Accompanies image file 'inner_outer_connectivity.pdf' (pdf, eps, ps)
\begingroup%
  \makeatletter%
  \providecommand\color[2][]{%
    \errmessage{(Inkscape) Color is used for the text in Inkscape, but the package 'color.sty' is not loaded}%
    \renewcommand\color[2][]{}%
  }%
  \providecommand\transparent[1]{%
    \errmessage{(Inkscape) Transparency is used (non-zero) for the text in Inkscape, but the package 'transparent.sty' is not loaded}%
    \renewcommand\transparent[1]{}%
  }%
  \providecommand\rotatebox[2]{#2}%
  \newcommand*\fsize{\dimexpr\f@size pt\relax}%
  \newcommand*\lineheight[1]{\fontsize{\fsize}{#1\fsize}\selectfont}%
  \ifx\svgwidth\undefined%
    \setlength{\unitlength}{1289.19845941bp}%
    \ifx\svgscale\undefined%
      \relax%
    \else%
      \setlength{\unitlength}{\unitlength * \real{\svgscale}}%
    \fi%
  \else%
    \setlength{\unitlength}{\svgwidth}%
  \fi%
  \global\let\svgwidth\undefined%
  \global\let\svgscale\undefined%
  \makeatother%
  \begin{picture}(1,0.3117816)%
    \lineheight{1}%
    \setlength\tabcolsep{0pt}%
    \put(0,0){\includegraphics[width=\unitlength,page=1]{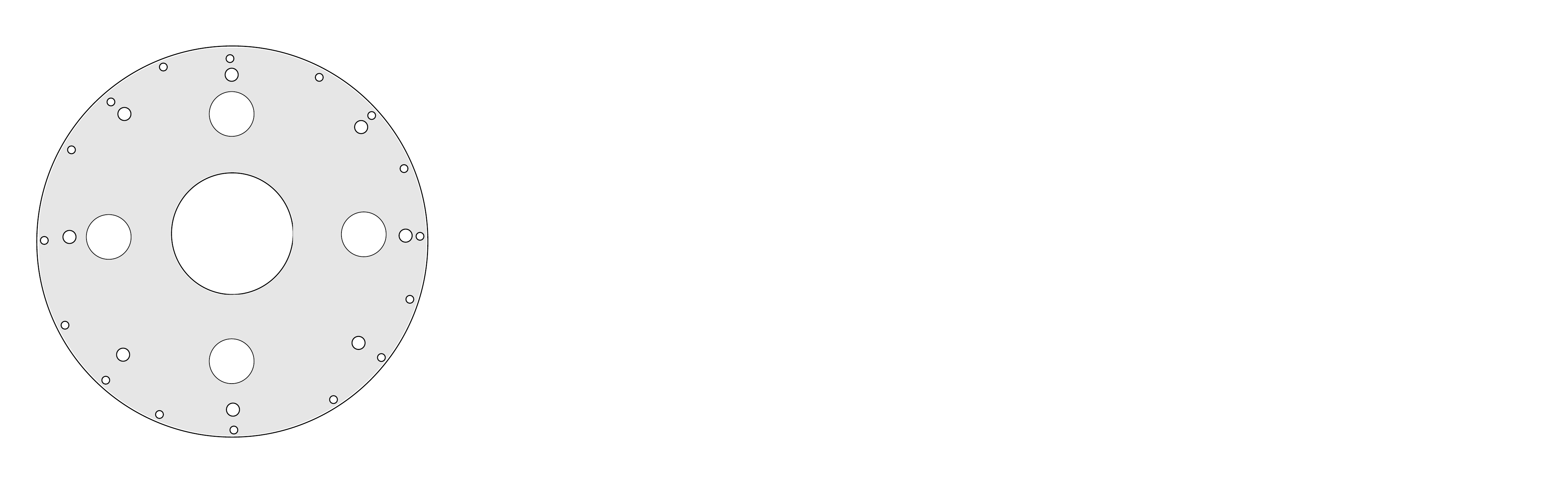}}%
    \put(0.13801292,0.0020157){\color[rgb]{0,0,0}\makebox(0,0)[lt]{\lineheight{1.25}\smash{\begin{tabular}[t]{l}\scalebox{2}{(A)}\end{tabular}}}}%
    \put(0.46477713,0.00172465){\color[rgb]{0,0,0}\makebox(0,0)[lt]{\lineheight{1.25}\smash{\begin{tabular}[t]{l}\scalebox{2}{(B)}\end{tabular}}}}%
    \put(0,0){\includegraphics[width=\unitlength,page=2]{inner_outer_connectivity.pdf}}%
  \end{picture}%
\endgroup%
}
		\caption{(A) and (B) illustrate the concept of infinite outer-connectivity, and infinite inner-connectivity, respectively. (A) and (B) also accurately describe the topology of the m.c.w.d.'s in \cite{Bis18} and Theorem \ref{main_theorem}, respectively. As seen, the structure in (B) is more intricate.}
		\label{inner_outer_connectivity}
	\end{figure}

	\begin{thm}\label{main_theorem} There exists a transcendental entire function $f: \mathbb{C} \rightarrow \mathbb{C}$ satisfying: \begin{enumerate} \item $\textrm{dim}(\mathcal{J}(f))=1$, \item Each Fatou component of $f$ is a m.c.w.d. of infinite inner-connectivity, and \item Each m.c.w.d. of $f$ has uncountably many singleton complementary components. \end{enumerate}
	\end{thm}	
	
	Item (3) of Theorem \ref{main_theorem} answers a question of \cite{RSEremenkoPoints} (see Question 9.5 of \cite{RSEremenkoPoints}) on the structure of m.c.w.d.'s. It is left open whether (3) in fact must \emph{always} occur for a m.c.w.d. of infinite inner-connectivity. Another intriguing question suggested by Theorem \ref{main_theorem} is whether there exist transcendental $f$ with $\textrm{dim}(\mathcal{J}(f))=1$ and \emph{doubly}-connected m.c.w.d.: this is closely related to Question (7) of \cite{Bis18}.
	
	Much of the contribution of the present manuscript is in providing an alternative approach to the breakthrough result of \cite{Bis18} (Item (1) in Theorem \ref{main_theorem}), an approach which the authors find conceptual and readily adaptable to other settings. The function $f$ of \cite{Bis18} is defined by an infinite product which is roughly designed to behave as a monomial on large portions of $\mathbb{C}$. The technical work in describing the dynamics of $f$ relies on formula-heavy estimates of the behavior of $f$ by certain terms in the infinite product. 
	
	The approach in the present manuscript is similar in that it constructs $f$ which is designed to behave as a monomial on large portions of $\mathbb{C}$, however this is done by quasiconformal methods. Namely, a quasiregular $h: \mathbb{C} \rightarrow \mathbb{C}$ is constructed, so that by the Measurable Riemann Mapping Theorem there exists a quasiconformal $\phi: \mathbb{C} \rightarrow \mathbb{C}$ such that $f:=h\circ\phi^{-1}$ is the entire function of Theorem \ref{main_theorem}. One has freedom in prescribing the dynamics of $h$, and so the difficulty of describing the dynamics of $f$ thus becomes a matter of estimating the ``correction'' map $\phi$, rather than on formula estimates as in \cite{Bis18}. The details of this quasiconformal approach were detailed in \cite{BurLaz}, and has other applications besides the one described in the present manuscript. 
	
	The advantages of the quasiconformal approach are usually technical in nature. For instance, a key aspect of the proof of Theorem \ref{main_theorem} is in understanding the location of critical values of $f$. In \cite{Bis18}, this requires a delicate estimate involving the infinite product formula. In the quasiconformal approach, this is almost trivial since the critical values of $h$ can be prescribed freely, and $f:=h\circ\phi^{-1}$ and $h$ share the same critical values. Another central difficulty in the proof of Theorem \ref{main_theorem} is showing that the outer boundary of a m.c.w.d. of $f$ is a $C^1$ curve (hence $1$-dimensional). In both approaches this involves studying pullbacks of circles. However, only in the quasiconformal approach is there an explicit parametrization (in terms of $\phi$) for the pullback, and this provides a different approach to the question of the precise degree of regularity for these curves. We will discuss further technical advantages of quasiconformal methods throughout the paper.
	
	We will outline the main arguments in the proof of Theorem \ref{main_theorem} and the structure of the paper in Section \ref{outlineproof}, before filling in the details in Sections \ref{The Construction}-\ref{singleton_components}. Appendix \ref{appendix} contains many classical theorems and definitions that we will make use of throughout the paper, along with a proof of an important Lemma we need in Section \ref{Expanding Zeros}. We would like to thank Chris Bishop for useful discussions, and Gwyneth Stallard for pointing out to us Question 9.5 of \cite{RSEremenkoPoints}. We would also like to thank the California Institute of Technology for hosting a visit of the second author which led to this work.

\section{Outline of the Proof}\label{outlineproof}	

We appeal to the main theorem of \cite{BurLaz} (described in Appendix \ref{appendix}) to produce the quasiregular function $h: \mathbb{C} \rightarrow \mathbb{C}$ as described in the Introduction. The map $h$ roughly behaves as $z\mapsto z^n$ for increasing $n$ as $z\rightarrow\infty$. In order to be able to prove dynamical properties about $f:=h\circ\phi^{-1}$, we need estimates on $|\phi(z)-z|$: these are proven in Section \ref{The Construction}.

In Section \ref{Global Mapping} we define a sequence of annuli $A_k$, $B_k$ for $k\geq1$ (see Figure \ref{Akillustration}), and we prove the following mapping behavior. First we show that \[ f(B_k)\subset B_{k+1}. \] Thus each $B_k$ is contained in a m.c.w.d. of $f$. We define subannuli $V_k\subset A_k$, and prove that \[ A_{k+1} \subset f(V_k). \] We also prove in Section \ref{Expanding Zeros} that there are balls $P_j \subset A_k$ such that $P_j\cap V_k=\emptyset$ which satisfy \[ A_{k+1} \subset f(P_j), \] and $f|_{P_j}$ is conformal.

\begin{figure}[!h]
	\centering
	\scalebox{.25}{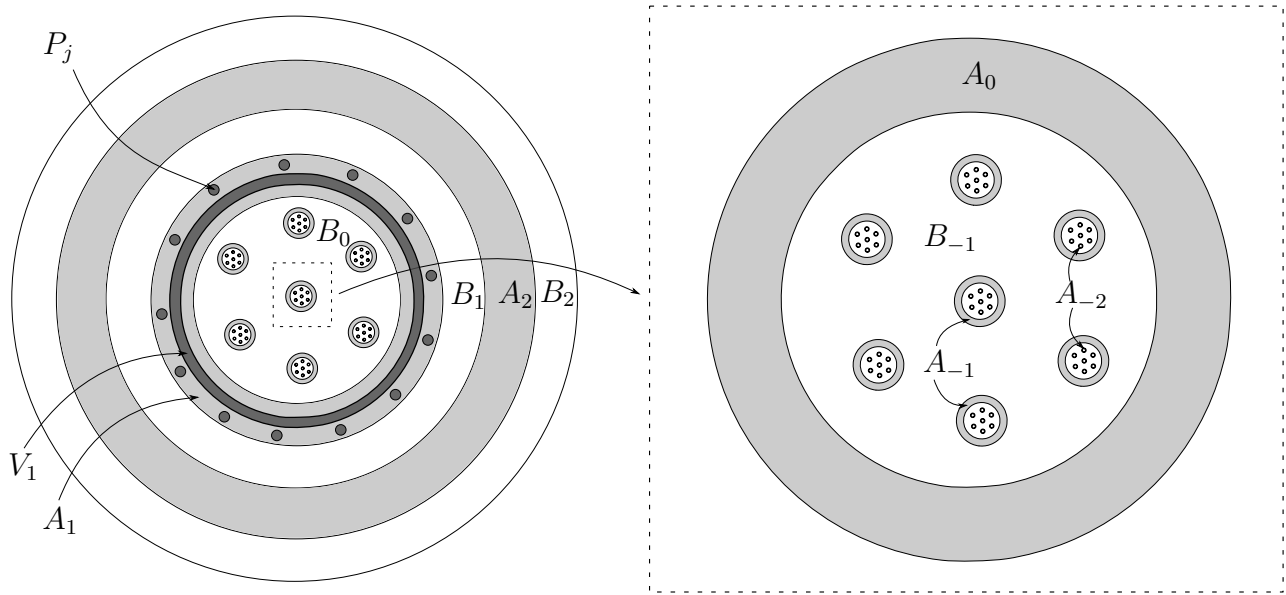}
	\caption{This Figure illustrates the definition of $A_k$, $B_k$ for all $k\in\mathbb{Z}$. The annuli $A_k$ are shaded light grey, and the $B_k$ are white. Also shown is $V_1\subset A_1$ and the ``petals'' $P_j\subset A_1$ (in dark grey).}
	\label{Akillustration}
\end{figure}

The definition of the annuli $A_k$, $B_k$ are extended to negative indices $k$ by pulling back under $f$ (see Figure \ref{Akillustration}).  Together, the annuli $A_k$ and $B_k$ cover $\mathbb{C}$ except for a Cantor set, which we denote by $E$. This Cantor set $E$ is the Julia set of the polynomial-like mapping obtained by restricting the definition of $f$ to a subdomain of $\mathbb{C}$, and we prove in Section \ref{Origin} that $\textrm{dim}(E)\ll1$. Similarly, denoting by $E'$ the set of points which map to $E$, it is readily deduced that $\textrm{dim}(E')=\textrm{dim}(E)$.

As the $B_k$ are contained in wandering components, we have that \begin{equation}\label{firstdefnofX} \mathcal{J}(f)\setminus E'  \subset \{ z \in\mathbb{C} : f^n(z) \in \cup_{k\in\mathbb{Z}} A_k \textrm{ for all } n \}. \end{equation} We denote the set defined on the right-hand side of (\ref{firstdefnofX}) by $X$, so that estimating $\textrm{dim}(\mathcal{J}(f))$ reduces to estimating $\textrm{dim}(X)$. 

We partition $X$ into two sets. For $z\in X$, we say that $z$ \emph{moves forwards} if $z\in A_k$ and $f(z)\in A_{k+1}$. If $z\in A_k$ and $f(z)\in A_j$ for $j\leq k$, we say that $z$ \emph{moves backwards}. We denote by $Y$ the set of $z\in X$ which move backwards infinitely often, and $Z:=X\setminus Y$ so that \[ X = Y \sqcup Z. \]

In Section \ref{Dimension Y}, we construct a sequence of covers $\mathcal{C}_m$ of $Y$, such that $\mathcal{C}_m$ covers all those points which move backwards $m$ times. This is done by simply pulling back the annuli $A_k$ under iterates of $f$ in regions where $f$ is conformal. Standard distortion estimates apply when estimating the diameters of elements of $\mathcal{C}_m$, and we deduce that $\textrm{dim}(Y)\ll 1$.

The set $Z$ is further partitioned into those points which eventually always stay in $\cup_k V_k$, denoted by $Z_1$, and $Z_2:=Z\setminus Z_1$. The dimension of $\mathcal{J}(f)$ is supported on $Z_1$. We prove in Section \ref{Dimension Z} that $Z_1$ consists of Jordan curves, and we prove in Section \ref{C1proofsection} that these curves are in fact $C^1$ (hence have dimension $1$). We prove in Section \ref{singleton_components} that $\textrm{dim}(Z_2)=0$, and $Z_2$ is precisely the set of singleton complementary components in (3) of Theorem \ref{main_theorem}.

\section{Quasiconformal Mapping Estimates}
\label{The Construction}
%	For any integer $N \geq 1$, we first show how to construct the function $f_N$ in Theorem \ref{main_dim_one}. We need to modify the construction of Theorem \ref{mainthm} in a neighborhood of the origin so that $f_N \circ \phi_N$ is a polynomial with a Cantor Julia set.

%\begin{rem} Note the relation $(2x_j)^{M_{j+1}}=x_{j+1}$. The map $f_n$ differs from the map obtained by applying Theorem \ref{mainthm} directly to (\ref{parameters_definition}) only in the disc $|z|\leq r_n$.
%\end{rem}

%In Theorem \ref{mainthm}, the entire function is a perturbation of $z\mapsto z^n$ in a large neighborhood of the origin. In order to prove Theorem \ref{application}, we will need to replace $z\mapsto z^n$ with a polynomial possessing a Julia set of dimension $<1$. In order to do this, we replace $z\mapsto z^n$ with the map $z\mapsto z^n+4z$, and interpolate between the two far from the origin. We will define an interpolation in Proposition \ref{zndeltazinterpolation} below, following Lemma 3.1 of \cite{MR4041106}. First we introduce a smooth bump function.

In this Section, we begin the proof of Theorem \ref{main_theorem} by first applying Theorem \ref{mainthm} (see Appendix \ref{appendix}) to a specific sequence  $(M_j)_{j=1}^\infty$, $(r_j)_{j=1}^\infty$ we now define. This yields an entire function $f$, so that $f\circ\phi=h$ is the quasiregular function described in Section \ref{outlineproof}. As discussed, we have freedom in describing the mapping behavior and dynamics of the quasiregular map $h$, but transferring this behavior to the entire function $f:=h\circ\phi^{-1}$ requires estimates on $|\phi(z)-z|$, and this is the main focus of this Section.

\begin{definition}\label{parameter_defn2} We define an entire function $f$ and a quasiconformal map $\phi: \mathbb{C}\rightarrow\mathbb{C}$ by applying Theorem \ref{mainthm} to the parameters: \begin{align}\label{parameters_definition} M_j:=2^{j}\textrm{, } r_1:=16\textrm{, } c_1:=1  \textrm{ and } r_{j+1}:=c_{j}\cdot\left(\frac{1}{2}r_j\right)^{M_{j}} \textrm{ for } j\geq2.     \end{align} %where $h$ is quasiregular and $\phi$ is quasiconformal.%Let $h_n: \mathbb{C}\rightarrow\mathbb{C}$ be the function defined by $h_n(z):=g_{M_n, r_n-1}$ for $|z|\leq r_n$. For $|z|\geq r_n$, define $h_n$ by (\ref{h_formula}). Lastly, we let $\phi_{n}$ denote any quasiconformal map obtained by applying the Measurable Riemann Mapping theorem to the Beltrami coefficient of $h_n$, so that $f_n:=h_n\circ\phi_n^{-1}$ is entire. We will fix a normalization for $\phi_n$ in Remark \ref{fir_normalization}.
	%\log x_j := \left({\sum_{k=0}^{j} 2^{k+(k+1)+...+j } }\right) \cdot \log 2 \textrm{ for } j\geq 0. 
\end{definition}

\begin{table}[!h]
	\centering
	\begin{tabular}{ |c|c|c|c|c|c| } 
		\hline
		$k$ & 0 & 1 & 2 & 3 & 4\\ 
		\hline
		$M_k$ & $1$ & $2$ & $4$ & $8$ & 16 \\ 
		\hline
		$c_k$ & undefined & $1$ &$ 2^{-8}$ & $2^{-32}$ & $2^{-128}$ \\ 
		\hline
		$r_k$ & 0 & $16$ & $64$ & $2^{12}$ & $2^{56}$ \\ 	
		\hline
	\end{tabular}
	\caption{The values of $M_k$, $c_k$, and $r_k$ for small values of $k$. The sequence $M_k$ increases exponentially, and $r_k$ increases super-exponentially, while $c_k$ decays super-exponentially.}
	\label{ValueTable}
\end{table}

We will need some rough estimates on how fast $(r_j)_{j=1}^\infty$ grows and $(c_j)_{j=1}^\infty$ decays. We first show that by assuming $(r_j)_{j=1}^{\infty}$ satisfies some mild growth conditions, we can show that $(r_j^{M_j})_{j=1}^{\infty}$ grows much faster than  $(c_j)_{j=1}^{\infty}$ decays.
\begin{lem}
	\label{ckrk}	
	Assume for all $k \geq 3$ that $\sqrt{r_{k}} \geq r_j$ for all $j < k$. Then we have
	\begin{equation}
	\label{ckrkeq}
	r_k^{M_k} \cdot c_k = r_k^{M_k} \cdot \prod_{j=1}^{k-1} r_j^{-M_j} \geq r_k^{M_{k-1} + 1},
	\end{equation}
	for all $k \geq 3$.
\end{lem}
\noindent

\begin{rem} We will prove in Lemma \ref{rk_prep} that the assumption of Lemma \ref{ckrk} does indeed hold.
	%		For our choice of $(r_k)_{k=1}^{\infty}$, we will prove in Lemma \ref{rk_prep} that in fact we do indeed have for all $k \geq 3$ that $\sqrt{r_k} \geq r_j$ for all $j <k$. 
\end{rem}

\begin{proof}
	This is just a calculation, making use of the fact that $M_j - M_{j+1} = - M_j$ and $2M_{j} = M_{j+1}$ for all $j \geq 0$, along with the definition of $c_k$ given by (\ref{mainthm_c}). When $k =1$, we verify (\ref{ckrkeq}) by checking that $r_1^{M_1} = c_1 \cdot r_1^{M_0 + 1}$. For the case of $k \geq 2$, we verify (\ref{ckrkeq}) by computing 
	\begin{align*}
	r_k^{M_k} \cdot c_k &= r_k^{M_k} \cdot \prod_{j=2}^{k} r_{j-1}^{M_{j-1} -M_j} \\
	&= r_k^{M_k} \cdot \prod_{j=2}^{k} r_{j-1}^{-M_{j-1}} \\
	(\sqrt{r_k} \geq r_j)	&\geq r_k^{M_k} \cdot \prod_{j=2}^{k} r_{k}^{-M_{j-2}} \\
	&= r_k^{M_k} \cdot r_k^{-\sum_{j=0}^{k-2} M_j} \\
	&= r_k^{M_k - M_{k-1} + 1} = r_k^{M_{k-1} +1}.
	\end{align*}
	In the last line, we used the fact that $\sum_{j=0}^{k-2} M_j = M_{k-1} -1$.
\end{proof}
\noindent

\begin{lem}
	\label{rk_prep}
	The sequence $(r_k)_{k=1}^{\infty}$ defined in Definition \ref{parameter_defn2} satisfies
	\begin{enumerate}
		\item $r_{2} > r_1$.
		\item For all $k \geq 2$, $\sqrt{r_{k+1}} \geq r_k$.
	\end{enumerate}
	In particular, $(r_k)_{k=1}^{\infty}$ is an increasing sequence, and if $k \geq 2$ we have $\sqrt{r_{k+1}} \geq r_j$ for all $j =1,\dots,k$. 
\end{lem}	
\begin{proof}
	The claim $(1)$ is just a calculation:
	\begin{equation}
	\label{rkprep_base}
	r_2 = c_1\left(\frac{r_1}{2}\right)^{M_1} = 8^2 = 64> 16 = r_1.
	\end{equation}
	We'll prove the second claim by induction. First, we have
	\begin{equation}
	\label{rkprep_base2}
	r_3 = c_2\left(\frac{r_2}{2}\right)^{M_2} = 2^{-8}(2^5)^4= 2^{12} = 64^2.
	\end{equation}
	Therefore $\sqrt{r_3} \geq r_2 > r_1$. 
	
	Suppose that for some $k \geq 3$ we have $\sqrt{r_k} \geq r_{j}$ for all $j = 1,\dots,k-1$. Then by Lemma \ref{ckrk}, 
	\begin{equation}
	\label{rkprep_eqn1}
	r_{k+1} = c_k r_k^{M_k} 2^{-M_k} 	\geq r_k^{M_{k-1} + 1} 2^{-M_k} = r_k^{M_{k-2} +1} r_k^{M_{k-2}} 16^{-M_{k-2}} \geq r_k^{M_{k-2}+1}. 
	\end{equation}
	Since $k \geq 3$, we have $M_{k-2} \geq M_1 = 2$, so that 
	\begin{align}
	\label{rkprep_eqn2}
	r_{k+1} \geq r_k \cdot r_k^2.
	\end{align}
	Therefore, by the inductive hypothesis we must have $\sqrt{r_{k+1}} \geq r_k \geq r_j$ for all $j  = 1,\dots, k-1$. This proves the claim.
\end{proof}

\noindent	We record the following important inequalities that follow from the proof of Lemma \ref{rk_prep}.
\begin{cor}
	\label{rkinequalities}
	We have the following inequalities. For all $k \geq 3$	
	\begin{equation}
	\label{lemma3.2applies}
	c_k r_k^{M_k} \geq r_k^{M_{k-1}+1}, \textrm{ and }
	\end{equation}
	\begin{equation}
	\label{xkesteq}
	r_{k+1} \geq 2^{-M_{k}} r_k^{M_{k-1}+1}.
	\end{equation}
	For all $k \geq 5$, we have
	\begin{equation}
	\label{xkest2_eqn1}
	r_{k+1} \geq 2^{2^{k}}= 2^{M_k}, \textrm{ and }
	\end{equation}
	\begin{equation}
	\label{xkest2_eqn2}
	r_{k+1} \geq 4 r_{k}^2.
	\end{equation}
\end{cor}
\begin{proof}
	Most of the work has already been done in the proof of Lemma \ref{rk_prep}. We first prove equation (\ref{lemma3.2applies}). When $k \geq 2$, By Lemma \ref{rk_prep} we have $\sqrt{r_{k+1}} \geq r_j$ for all $j =1,\dots k$. Therefore, by Lemma \ref{ckrk}, we obtain (\ref{lemma3.2applies}).
	
	Equation (\ref{xkesteq}) follows immediately. Indeed, the first two lines of (\ref{rkprep_eqn1}) yields
	$$r_{k+1} \geq r_k^{M_{k-1}+1} 2^{-M_k},$$
	when $k \geq 3$.
	
	When $k \geq 5$, we can refine the estimate $r_{k+1} \geq r_k^{M_{k-2}+1}$ from (\ref{rkprep_eqn1}). We note that by Lemma \ref{rk_prep}, we have $r_k > 16$ for all $k \geq 5$. Therefore,
	$$r_{k+1} \geq r_k^{M_{k-2}+1} > 16^{M_{k-2}} = (2^4)^{M_{k-2}} = 2^{M_k}.$$
	Finally, since $r_k > 16$ for all $k \geq 1$, we certainly obtain (\ref{xkest2_eqn2}) from (\ref{rkprep_eqn2}).
\end{proof}
\noindent

Corollary \ref{rkinequalities} concludes our discussion of some technical relations and inequalities we will need for the sequences $(r_j)_{j=1}^\infty$, $(c_j)_{j=1}^\infty$. As discussed in the Introduction, one of the key advantages of the quasiconformal approach (over the infinite-product approach) is the relative simplicity of deducing the singular value structure of the constructed function. This is summarized in the following Proposition. Although (\ref{final_crit_pts_listing_2}) and (\ref{listing_of_zeros}) are slightly technical, they simply say the critical points are radially equidistributed on each circle $|z|=r_j$, and the zeros are equidistributed on a slightly larger circle.

\begin{prop}
	\label{critical_point_listing}
	Let $(M_j)_{j=1}^\infty$, $(r_j)_{j=1}^\infty$ and $(c_j)_{j=1}^{\infty}$ be as in Definition \ref{parameter_defn2}. Then the only critical points of $f$ are $0$ and the simple critical points given by 
	\begin{equation}
	\label{final_crit_pts_listing_2}
	\phi\left( r_j\cdot \exp\left(i\frac{(2k_j-1)\pi}{M_{j}} \right) \right), 
	\end{equation} 
	where $j\in\mathbb{N}$, and $1 \leq k_j \leq M_{j}$. The only singular values of $f$ are $0$ and the critical values $(\pm c_jr_j^{M_j})_{j=0}^\infty$. The zeros of $f$ are given by: 
	\begin{equation}
	\label{listing_of_zeros} 
	0\textrm{ and }\phi\left( r_j \cdot \exp \left(  \frac{1}{4}\frac{\pi}{M_{j}} + i \cdot \frac{(2k_j-1)\pi}{M_{j}} \right) \right), 
	\end{equation}
	where $j \in \N$ and $1 \leq k_j \leq M_{j}$.	All of the zeros of $f$ are simple except for $0$ which is of multiplicity $2$.
\end{prop}
\begin{proof}
	This follows immediately from Proposition 3.21 \cite{BurLaz}.
\end{proof}
\noindent

%\begin{rem} Note the relation $(2x_j)^{M_{j+1}}=x_{j+1}$. The map $f_n$ differs from the map obtained by applying Theorem \ref{mainthm} directly to (\ref{parameters_definition}) only in the disc $|z|\leq r_n$.
%\end{rem}

%In Theorem \ref{mainthm}, the entire function is a perturbation of $z\mapsto z^n$ in a large neighborhood of the origin. In order to prove Theorem \ref{application}, we will need to replace $z\mapsto z^n$ with a polynomial possessing a Julia set of dimension $<1$. In order to do this, we replace $z\mapsto z^n$ with the map $z\mapsto z^n+4z$, and interpolate between the two far from the origin. We will define an interpolation in Proposition \ref{zndeltazinterpolation} below, following Lemma 3.1 of \cite{MR4041106}. First we introduce a smooth bump function.

We now move on to show how to modify $f$ near the origin so that instead of being modeled by a function of the form $z^n$, it is modeled by a polynomial with a Cantor repeller Julia set. This will be advantageous because $z^n$ has a Julia set of dimension $1$, whereas the constructed Cantor repeller will have dimension $\ll1$.

The main idea is that a monic, degree $M_k$ polynomial $p(z)$ behaves like $z\mapsto z^{M_k}$ near $\infty$. We will show how to interpolate between $p(z)$ and $z\mapsto z^{M_k}$ in a way that is quasiconformal with dilatation bounded independent of $k$.  Our strategy closely follows Section 3 of \cite{MR4041106}.
\noindent

\begin{definition} We define \[ b(x)=\begin{cases} 
	\exp(1+\frac{1}{x^2-1}) & \textrm{ if } 0\leq x < 1 \\
	0 & \textrm{ if } x\geq 1,
	\end{cases}
	\]
	and, for $r\geq1$, the smooth map 
	\[ \widehat{\eta}_{r}(x)=\begin{cases} 
	1 & \textrm{ if }x\leq r-1 \\
	b(x-r+1) & \textrm{ if } r-1 \leq x \leq r \\
	0 & \textrm{ if }x\geq r.
	\end{cases}
	\]
	We also set $\eta_{r}(z)=\widehat{\eta}_{r}\left(|z|\right)$. 
\end{definition}

\begin{prop}\label{zndeltazinterpolation} Let   $g_k(z):=c_k z^{M_k}+ r_k z\eta_{r_k}(z)$, and $\mu_k:=(g_k)_{\overline{z}}/(g_k)_{z}$. Then there exists $K'\in\mathbb{N}$ with $K' \geq 5$ such that: \begin{align} \sup_{k\geq K'} \left|\left| \mu_k \right|\right|_{L^\infty(\mathbb{C})} <  1. \end{align}
\end{prop}

\begin{proof} We abbreviate $\eta(z):=\eta_{r_k}(z)$. We use a similar strategy as in the proof of Lemma 3.1 in \cite{MR4041106}, and the initial steps of the proof are exactly the same. We have 
	
	\[ (g_k)_z(z)=M_kc_kz^{M_k-1}+r_k\eta(z)+r_kz\eta_z(z)\quad\text{and }\quad  (g_k)_{\overline{z}}(z)=r_kz\eta_{\overline{z}}(z). \]
	
	\vspace{2mm}
	
	\noindent Solving $b''(x)=0$, one sees that $|b'(x)|$ has a maximum at $x_0=(1/3)^{1/4}$ with $|b'(x_0)|<e$, so that $|b'(x)|\leq e$ for $x\in[0,1]$. Thus $|(\widehat{\eta})'(x)|\leq e$ for all $x>0$. Using the chain rule again, we have 
	
	\[ \left| \frac{\partial\eta}{\partial z}(z) \right| = \left| (\widehat{\eta})'(|z|) \right|\cdot\left| \frac{\partial|z|}{\partial z} \right| \leq \frac{e}{2}\quad \textrm{and} \quad \left| \frac{\partial\eta}{\partial \overline{z}}(z) \right| = \left| (\widehat{\eta})'(|z|) \right|\cdot\left| \frac{\partial|z|}{\partial \overline{z}} \right| \leq \frac{e}{2},\]
	
	\vspace{2mm}
	
	\noindent where we have used the fact that 
	$$\frac{\partial|z|}{\partial z}=\frac{\overline{z}}{2|z|} \quad {\rm and} \quad \frac{\partial|z|}{\partial \overline{z}}=\frac{z}{2|z|}.$$
	Hence

	\begin{align}\label{ineq1} \left|\frac{(g_k)_{\overline{z}}(z)}{(g_k)_z(z)}\right| \leq \frac{r_k|z|\frac{e}{2}}{\left| M_kc_k|z|^{M_k-1} - r_k|\eta(z)| - r_k|z|\frac{e}{2} \right| } =  \frac{\frac{e}{2}}{\left| \frac{M_kc_k|z|^{M_k-2}}{r_k} - \frac{|\eta(z)|}{|z|} - \frac{e}{2} \right| }.  \end{align} Let us consider the right-hand side of (\ref{ineq1}) for $|z|=r_k - 1$, recalling $M_k:=2^k$. We have that: 
	\begin{align*} \frac{c_k|z|^{M_k-2}}{r_k}&:=\frac{1}{r_k(r_k-1)^2}\left(\prod_{j=2}^k\frac{1}{r_{j-1}^{M_j-M_{j-1}}}\right)(r_k-1)^{2^k} \\
	&= \frac{1}{r_k(r_k-1)^2}\cdot\frac{(r_k-1)^2\cdot (r_k-1)^{2^2}\cdot...\cdot (r_k-1)^{2^{k-1}}}{r_1^2\cdot r_2^{2^2}\cdot ... \cdot r_{k-1}^{2^{k-1}}} \\
	&= \frac{(r_k-1)^{2^2}}{r_k} \cdot \frac{(r_k-1)^{2^3}}{r_1^2\cdot r_2^{2^2}\cdot r_3^{2^3}}  \cdot \frac{(r_k-1)^{2^4}\cdot...\cdot (r_k-1)^{2^{k-1}}}{r_4^{2^4} \cdot ... \cdot r_{k-1}^{2^{k-1}}}
	\end{align*} 
	By Lemma \ref{rk_prep} and Lemma \ref{rkinequalities}, we may deduce for all $k \geq 5$,
	\begin{equation}
	r_k-1>2r_{k-1},
	\end{equation} 
	\begin{equation}
	\frac{(r_k-1)^{2^3}}{r_1^2\cdot  r_2^{2^2} \cdot r_3^{2^3}} \geq 1, \textrm{ and }
	\end{equation}
	\begin{equation}
	(r_k-1)^2 > r_k.
	\end{equation}
	Combining the above inequalities, it follows that when $|z| \geq r_k -1$, we have
	\begin{equation} \frac{c_k|z|^{M_k-2}}{r_k} \geq r_{k-1}.
	\end{equation} 
	Thus it follows from (\ref{ineq1}) that in fact $|(g_k)_{\overline{z}}/(g_k)_z|\rightarrow0$ as $k\rightarrow\infty$. 	
\end{proof}

\begin{definition}\label{parameter_defn} Let $f$, $\phi$ be as in Definition \ref{parameter_defn2}, and $h:=f\circ\phi$. We define a family of entire functions $f_N:=h_N\circ\phi_N^{-1}$ as follows. Let:  \[ h_N(z):=\begin{cases} 
	g_N(z) & \textrm{ if } |z|\leq r_N \\
	h(z) & \textrm{ if } |z|\geq r_N,
	\end{cases}
	\] and $\phi_N: \mathbb{C}\rightarrow\mathbb{C}$ is the quasiconformal mapping such that: \begin{enumerate} \item $f_N$ is holomorphic, \item $\phi_N(0)=0$, and \item $|\phi_N(z)/z - 1|\rightarrow0$ as $z\rightarrow\infty$.  \end{enumerate}
\end{definition}

The fact that $\phi_N$ may be normalized so that $(3)$ is satisfied follows from an argument similar to Remark 5.2 in \cite{BurLaz}.	

\begin{rem} 
	We will always assume that $N \geq 5$.	Note that for $|z|=r_N$, we have $g_N(z)=h(z)$. 
\end{rem}
\noindent

\begin{rem} We will show that for all sufficiently large $N$, the function $f_N$ satisfies the conclusions of Theorem \ref{main_theorem}. We will occasionally omit the subscript $N$ and simply write $f$ when convenient.
\end{rem}

%To be more precise, we will actually construct a family of quasiregular functions $h_N: \C \rightarrow \C$ depending on $N \in \N$. The entire functions which satisfy Theorem \ref{main_theorem} will be of the form $f_N = h_N \circ \phi_N^{-1}$, for all $N$ that are sufficiently large. During the proof, we will sometimes switch between the notation $f$ and $f_N$ when it is convenient. 

\begin{prop}\label{constant_dilatation} For $\phi_N$ as in Definition \ref{parameter_defn}, $\sup_N K(\phi_N)<\infty$.
\end{prop}

\begin{proof} By Proposition \ref{zndeltazinterpolation}, we have $\sup_NK(\phi_N|_{\{|z|\leq r_N\}})<\infty$, and by Proposition 4.6 of \cite{BurLaz}, we have $\sup_N K(\phi_N|_{\{|z|\geq r_N\}})<\infty$.
\end{proof}

In \cite{BurLaz}, the conclusion $|\phi(z)/z-1|\xrightarrow{z\rightarrow\infty}0$ of Theorem \ref{mainthm} is deduced by an application of the Teichm\"uller-Wittich-Belinskii Theorem (see Theorem 6.1 of \cite{MR0344463}). We will need a more quantitative statement for the purposes of proving Theorem \ref{main_theorem}, in particular when we prove that the m.c.w.d.'s of Theorem \ref{main_theorem} have smooth boundary. This quantitative statement is given in Theorem \ref{shishikura} below. The proof follows from the main arguments of \cite{MR3839848}. Indeed, Theorem \ref{shishikura} is quite analagous to the main result of \cite{MR3839848}, but we will need to assume less than in \cite{MR3839848}, and accordingly we will obtain a weaker conclusion, which will nevertheless suffice to prove Theorem \ref{main_theorem}.

\begin{definition}
	\label{omega}
	For $p\geq1$ and $0 < r < 1$, we will denote \begin{align} \omega_p(r):=\left(\frac{1}{2}\right)^{p^{-1}\sqrt{\log\log r^{-1}}}.\end{align}
\end{definition}

\begin{thm}\label{shishikura} Let $\psi: \mathbb{C} \rightarrow \mathbb{C}$ be a quasiconformal mapping, $\mu:=\psi_{\overline{z}}/\psi_z$, and suppose that \begin{align}\label{logarithmic_decay} I(r) := \int \!\!\!\!\!\int_{\{ |z|<r \}} \frac{|\mu|}{1-|\mu|^2} \frac{dxdy}{|z|^2} \textrm{ is finite and has order } O\left(\omega_1(r)\right) \textrm{ as } r\searrow0. \end{align} Then $\psi$ is conformal at $0$, and for any $p>2$, we have \begin{align}\label{bigOconclusion} \psi(z)=\psi(0)+\psi'(0)z+O\left(\omega_p(|z|)\right) \textrm{ as } z\rightarrow0. \end{align}
\end{thm}

\begin{proof} By Lemma 10 of \cite{MR3839848}, we have that for any $p>2$ and $0<\rho<1$, there exists $C'=C'(p,\rho)$ such that if $0<|z_2|<\rho^2|z_1|$, then \begin{align}\label{first_inequality} \left| \int \!\!\!\!\!\int_{\mathbb{C}} \frac{\mu(z)\phi_{z_1, z_2}(z)}{1-|\mu(z)|^2}dxdy \right| \leq \frac{1}{1-\rho^2}\left| \int\!\!\!\!\! \int_{A(\rho^{-1}|z_2|, \rho|z_1|)} \frac{\mu(z)}{1-|\mu(z)|^2}\frac{dxdy}{z^2} \right| +C'I_{p,2}(\mu; |z_1|)^{1/p},\end{align} and \begin{align}\label{second_inequality}  \int \!\!\!\!\!\int_{\mathbb{C}} \frac{|\mu(z)|^2|\phi_{z_1, z_2}(z)|}{1-|\mu(z)|^2}dxdy  \leq \frac{1}{1-\rho^2} \int\!\!\!\!\! \int_{A(\rho^{-1}|z_2|, \rho|z_1|)} \frac{|\mu(z)|^2}{1-|\mu(z)|^2}\frac{dxdy}{|z|^2} +C'I_{p,2}(\mu; |z_1|)^{1/p},\end{align} where \begin{align} \phi_{z_1, z_2}(z)=\frac{z_1}{z(z-z_1)(z-z_2)}\textrm{, and } I_{p,2}:=\int \!\!\!\!\!\int_{\mathbb{C}}  \frac{|\mu(z)|^p}{(1-|\mu(z)|^2)^p} \frac{dxdy}{|z|^2(1+|z|/r)^2}. \end{align} Thus by Theorem 8 of \cite{MR3839848}, it suffices to show that the $\liminf_{z_2\rightarrow0}$ of the four terms on the right hand sides of (\ref{first_inequality}) and (\ref{second_inequality}) are $O(\omega_p(|z_1|))$ as $z_1\rightarrow0$. In fact we have better estimates on the two integral terms: they are $O(\omega_1(|z_1|))$ by the assumption (\ref{logarithmic_decay}). For the remaining two terms, we use Lemma 11 of  \cite{MR3839848}: there exist constants $C_2$ and $C_3$ depending only on $K(\psi)$ such that for $0<r<r'$, \begin{align}\label{I_term} I_{p,2}(\mu; r) \leq C_2 \int \!\!\!\!\!\int_{|z|<r'} \frac{|\mu(z)|^2}{1-|\mu(z)|^2}\frac{dxdy}{|z|^2} + \frac{C_3}{2}\left(\frac{r}{r'}\right)^2.  \end{align} Letting $r'=r^{1/2}$, we see the first term on the right-hand side of (\ref{I_term}) has order $O(\omega_1(r'))=O(\omega_1(r))$, and the second term has order $O(r)$, so that $I_{p,2}$ has order $O(\omega_1(r))$. Thus $I_{p,2}(\mu; |z_1|)^{1/p}$ has order $O(\omega_p(|z_1|))$, and so the result follows.
\end{proof}

\begin{rem}\label{constants} One readily sees from the constants in the proof of Theorem \ref{shishikura} that the big-$O$ constants in (\ref{bigOconclusion}) depend only on $K(\psi)$ and the big-$O$ constants in (\ref{logarithmic_decay}). In particular, they are independent of $N \in \N$.
	
	%Let $\psi$ be as in Theorem \ref{shishikura}, and suppose $\psi(0)=0$. Then Theorem \ref{shishikura} says that $\psi'(0)$ exists, and that there exist $s>0$ and $C>0$ such that \begin{align} |\psi(z)/z-\psi'(0)| < C\cdot\omega_p(|z|) \textrm{ for } 0<|z|<s. \end{align} One readily sees from the constants in the proof of Theorem \ref{shishikura} that the constants $C$, $s$ depend only on $K(\psi)$ and the big-$O$ constants in the $I(r)=O(\omega_1(r))$ estimate.
\end{rem}

\noindent We now apply Theorem \ref{shishikura} to our particular setting:	

\begin{thm}\label{initial_identity} There exists $C' >0$ and $R>0$ such that for any $N\in\mathbb{N}$ and any $p>2$: \begin{align}\label{near_infty}\left|\frac{\phi_N(z)}{z} - 1\right| < C'\cdot\omega_p(1/|z|) \emph{ for } |z|>R. \end{align}
\end{thm}

\begin{proof} Let $N\in\mathbb{N}$. Let $\phi:=\phi_N: \mathbb{C}\rightarrow\mathbb{C}$ be a quasiconformal mapping such that $h_{N}\circ\phi^{-1}$ is holomorphic, and $\phi(0)=0$. Consider \[ \psi(z):=1/\phi(1/z), \] and define $I(r)$ as in (\ref{logarithmic_decay}). We wish to apply Theorem \ref{shishikura}. To this end, we calculate \begin{align}\label{first_shish_estimate}  I(r) \leq \frac{k}{1-k^2}\int\!\!\!\!\! \int_{(|z|<r)\cap\textrm{supp}(\psi_{\overline{z}})} \frac{dxdy}{|z|^2} \leq \frac{k}{1-k^2} \sum_{j\geq j(r)} \int\!\!\!\!\! \int_{G_j}\frac{dxdy}{|z|^2}, \end{align} where: \begin{equation} j(r) \textrm{ is the smallest integer such that } 1/r<r_{j(r)} \cdot \exp(\pi/M_j), \end{equation} and  \begin{equation} G_j:=\{z\in\mathbb{C} : r_j^{-1} \cdot \exp(-\pi/M_j) \leq |z| \leq (r_j-1)^{-1} \}. \end{equation} 
	As in the proof of Theorem 4.8 of \cite{BurLaz}, we calculate

	\begin{align}
	\label{calculation}
	\frac{k}{1-k^2} \sum_{j\geq j(r)} \int\!\!\!\!\! \int_{G_j}\frac{dxdy}{|z|^2} &\lesssim  \sum_{j \geq j(r)} \int\!\!\!\!\! \int_{G_j}\frac{dxdy}{r_j^{-2}\exp(-2\pi/M_j)}  \\
	&= \sum_{j\geq j(r)} \frac{\pi((r_j-1)^{-2}-r_j^{-2}\exp(-2\pi/M_j))}{r_j^{-2}\exp(-2\pi/M_j)} \nonumber \\
	&\simeq  \sum_{j \geq j(r)}  \left(\left(\frac{r_j}{r_j-1}\right)^{2}\exp(2\pi/M_j)-1\right) \nonumber \\
	&\lesssim \sum_{j \geq j(r)}  \left(\left(\frac{r_j}{r_j-1}\right)^{2}-1 + \left(\frac{r_j}{r_j-1}\right)^{2}\frac{4\pi}{M_j}\right) \nonumber \\
	&\lesssim \sum_{j \geq j(r)}\left( \frac{2r_j -1}{(r_j-1)^2} + \frac{8\pi}{M_j} \right) \nonumber \\
	&\lesssim \sum_{j \geq j(r)} \left( \frac{1}{r_j} + \frac{1}{M_j} \right)  \nonumber\\
	&\lesssim   \left(\frac{1}{2} \right)^{j(r)}, \nonumber
	\end{align} 
	where we have used the fact that $M_j = 2^j$, Corollary \ref{rkinequalities}, and the inequality 
	$$\exp(x) \leq 1 + 2x \textrm{ for all } x \leq 1.$$ 
	%\[  \sum_{n=1}^\infty a_n \simeq \log\left(\prod_{n=1}^\infty (a_n+1)\right) \] for a sequence $a_n\rightarrow0$ as $n\rightarrow\infty$.
	Next, we note that \begin{align} 2^{2^{(j+1)(j+2)/2}}>2^{2^{1+...+j}+...+2^j}>r_j \cdot \exp(\pi/M_j). \end{align} Thus it is readily calculated that \begin{align} r_j \cdot \exp(\pi/M_j)> 1/r \implies (j+1)(j+2) > \log\log r^{-1}, \end{align} and so since $(1/2)^j \simeq (1/2)^{\sqrt{(j+1)(j+2)}}$, it follows that \begin{align}\label{last_shish_estimate} r_j\cdot \exp(\pi/M_j)> \exp(-\pi/M_{j(r)})/r > 1/(2r) \implies \left(\frac{1}{2}\right)^j \lesssim \left(\frac{1}{2}\right)^{\sqrt{\log\log r^{-1}}}. \end{align}

	Together, (\ref{first_shish_estimate})-(\ref{last_shish_estimate}) imply that $I(r)$ is finite and has order $O(\omega_1(r))$ as $r\searrow0$, as needed. Thus we may apply Theorem \ref{shishikura} to deduce that there exist $c>0$ and $r>0$ so that \begin{align} |\psi(z)/z-\psi'(0)| < c\cdot\omega_p(|z|) \textrm{ for } |z|<r. \end{align} By multiplying $\psi$ by a complex constant, we may assume that $\psi'(0)=1$. Since $1/\phi(1/z)=\psi(z)$, the inequality (\ref{near_infty}) follows by taking $R=1/r$ and $C'>c$. By Proposition \ref{constant_dilatation} and Remark \ref{constants}, the constants $C'$ and $R$ do not depend on $N$ since the above big-$O$ estimates for $I(r)$ do not depend on $N$.
	%We continue our calculation: \begin{align}\label{calculation} \frac{1}{2\pi}\sum_{j=1}^\infty \int_{B_j}\frac{K-1}{|z|^2}\emph{d}A(z) \leq \frac{K-1}{2\pi}\sum_{j=1}^\infty \int_{B_j} \frac{1}{x_j^{-2}\exp(-2\pi/M_j)}\emph{d}A(z) = \\ \frac{K-1}{2\pi}\sum_{j=1}^\infty \frac{\pi(x_j^{-2}-x_j^{-2}\exp(-2\pi/M_j))}{x_j^{-2}\exp(-2\pi/M_j)} =  \frac{K-1}{2}\sum_{j=1}^\infty\big(\exp(2\pi/M_j)-1\big).  \nonumber \end{align}
\end{proof}

%\begin{rem}\label{fir_normalization} There is a unique quasiconformal mapping  $\phi_n: \mathbb{C} \rightarrow \mathbb{C}$ such that \begin{enumerate} \item $\phi_n(0)=0$, \item (\ref{near_infty}) holds, and \item $f_n:=h_n\circ\phi_n^{-1}$ is holomorphic. \end{enumerate} We will henceforth refer to this mapping by $\phi_n$ without ambiguity.
%\end{rem}

\noindent 	For the rest of the paper, we fix $C',R >0$ so that Theorem \ref{initial_identity} holds.

\begin{thm}
	\label{identity_origin}
	Let $\varepsilon>0$. There exists $N_\varepsilon\in\mathbb{N}$ such that for $N>N_{\varepsilon}$ we have \begin{align}|\phi_N(z)-z| < \varepsilon \emph{ for } |z|<R.\end{align}
\end{thm}
\begin{proof} Let $\mu_N$ be the Beltrami coefficient of $\phi_N$. As $N\rightarrow\infty$, we have $\mu_N\rightarrow0$ pointwise. Thus, we have $|\phi_N(z)- z|\rightarrow 0$ uniformly on the compact set $|z|\leq R$.
\end{proof}

\noindent	We conclude by restating Proposition \ref{critical_point_listing} which listed the critical points and values of $f_N$, but now adapted to account for the new behavior of the function $f_N$ near $0$:

\begin{lem}
	\label{branch_covering_hn}
	Let $\{\zeta_j\}_{j=1}^{2^N-1}$ denote the $2^N-1$ many critical points of $g_{N}(z)$ contained in $B(0,r_N)$. Then the only critical points of $f_N$ are the simple critical points given by $\phi_N(\zeta_j)$ for $j =1,\dots,2^N-1$, and the simple critical points given by 
	\begin{equation}
	\label{final_crit_pts_listing}
	\phi_N\left( r_j\cdot \exp\left(i\frac{(2k_j-1)\pi}{M_{j}} \right) \right), \textrm{  } j \geq N, \textrm{  } 1\leq k_j \leq M_j.
	\end{equation}
	The only singular values of $f_N$ are the critical values $(\pm c_j r_j^{M_j})_{j=N}^{\infty}$ and the critical values $(g_N(\zeta_j))_{j=1}^{2^N-1}$. 
\end{lem}

\section{Mapping Behavior near $\infty$}
\label{Global Mapping}

Having proven in Section \ref{The Construction} all the estimates on the ``correction'' map $\phi$ that we will need, we can now begin describing the mapping behavior of the function $f:=h\circ\phi^{-1}$. In Section \ref{Global Mapping}, we will introduce the annuli $A_k$, $V_k$, $B_k$ for $k\geq1$: these regions will be central to the proof of Theorem \ref{main_theorem}, as discussed in Section \ref{outlineproof}. We will also prove in Lemma \ref{Ak} and \ref{Bk} the fundamental relations: \[ f(B_k)\subset B_{k+1} \textrm{ and } A_{k+1}\subset f(V_k) \textrm{ for all } k\geq1. \] 

In this Section, and in the rest of the paper, we will consider the case $p = 2 \cdot \sqrt{2}$ for Definition \ref{omega}. The following lemma gives us estimates for how $\omega_{2 \cdot \sqrt{2}}(|z|^{-1})$ decays as $z \ra \infty$. 
\begin{lem}
	\label{omega_error}
	There exists $k_0\in\mathbb{Z}$ so that if $k \geq k_0$ and $|z| \geq \frac{1}{20} r_k$, then
	\begin{equation}
	\label{omega_3}
	\omega_{2\sqrt{2}}\left(\frac{1}{|z|}\right) \leq \left( \frac{1}{2} \right)^{\frac{\sqrt{k}}{4}}.
	\end{equation}
\end{lem}
\noindent

\begin{proof}
	This is just a simple calculation using Definition \ref{omega} and Corollary \ref{rkinequalities}. Indeed, for all $k \geq 10$, we have
	\begin{align*}
	\log \log \left(\frac{r_k}{20}\right )  &\geq \log \log \left(\frac{2^{M_{k-1}}}{20}\right) \\
	&\geq \log \log (2^{M_{k-5}})\\
	&= \log ( M_{k-5} \log 2) \\
	&= \log (2^{k-5} \log 2) \\
	&= (k-5) \log 2 + \log \log 2.
	\end{align*}
	Therefore there exists a value $k_0$ so that for all $k \geq k_0$, we have $\log \log \frac{r_k}{20} \geq \frac{1}{2}k.$ For all such $k$, we verify that when $|z| \geq \frac{r_k}{20}$ we have
	\begin{align*}
	\omega_{2\sqrt{2}}\left(\frac{1}{|z|} \right) & =  \left(\frac{1}{2}\right)^{\frac{1}{2 \sqrt{2}}\sqrt{\log\log |z|}} 
	\leq  \left(\frac{1}{2}\right)^{\frac{1}{2 \sqrt{2}}\sqrt{\log\log (\frac{r_k}{20})}} 
	\leq  \left( \frac{1}{2} \right)^{\frac{1}{2 \sqrt{2}} \sqrt{\frac{k}{2}}} 
	=  \left( \frac{1}{2} \right)^{\frac{\sqrt{k}}{4}}. 
	\end{align*}
	This yields (\ref{omega_3}) as desired.
\end{proof}

\begin{rem}
	\label{k0}
	Note that by perhaps choosing $k_0$ larger, we may additionally assume that $r_k \geq \frac{1}{20}r_{k_0} \geq R$ for all $k \geq k_0$. In this case, Theorem \ref{initial_identity} and Lemma \ref{omega_error} imply that for all $k \geq k_0$, if $|z| \geq \frac{r_k}{20}$, we have $\frac{\phi_N(z)}{z} \in B(1,C' \cdot 2^{-\sqrt{k}/4}).$ 
\end{rem}

\begin{lem}
	\label{AkEst}
	Let $k_0$ be as in Remark \ref{k0}. For all $k \geq k_0$, if $|z| \geq \frac{r_k}{20}$ and $N \geq 5$, we have 
	\begin{equation}
	\label{phi_identity}
	\left(1 - C' \cdot  \left( \frac{1}{2} \right)^{\frac{\sqrt{k}}{4}}\right) |z| \leq |\phi_N(z)| \leq \left(1+ C' \cdot \left( \frac{1}{2} \right)^{\frac{\sqrt{k}}{4}}\right) |z|.
	\end{equation}
	\noindent
	
	Moreover, if $z \in \phi_N(\{w\,:\, |w| \geq \frac{r_k}{20}\})$, then 
	\begin{equation}
	\label{phi_inverse_identity}
	\frac{1}{\left(1+ C' \cdot \left( \frac{1}{2} \right)^{\frac{\sqrt{k}}{4}} \right)} |z| \leq |\phi_N^{-1}(z)| \leq \frac{1}{\left(1 - C' \cdot \left( \frac{1}{2} \right)^{\frac{\sqrt{k}}{4}} \right)} |z|.
	\end{equation}
\end{lem}

\begin{proof}
	Equation (\ref{phi_identity}) is just a rearrangement of equation (\ref{near_infty}), but using the estimate (\ref{omega_3}). For the second equation, just note that if $z \in \phi_N(\{w\,:\,|w| \geq \frac{r_k}{20}\})$, then there exists $w$ with $|w| \geq \frac{r_k}{20}$ so that $\phi_N(w) = z$. Then (\ref{phi_identity}) holds with $w = \phi_N^{-1}(z)$, so (\ref{phi_inverse_identity}) is just a rearrangement of (\ref{phi_identity}).
\end{proof}

\begin{rem}
	We will always assume that the integer $N \in \N$ satisfies $N \geq k_0$.
\end{rem}

\noindent	Next, we will do some re-indexing of variables. This will make our notation easier to read and more consistent with \cite{Bis18}.
\begin{definition}
	\label{bigR}
	Given the parameters $M_k$, $c_k$, and $r_k$ from Definition \ref{parameter_defn2}, and given any integer $N  \geq 1$, we define
	\begin{equation}
	\label{nk_def}
	n_k := M_{k+N-1} = 2^{N + k - 1},
	\end{equation} 
	and
	\begin{equation}
	\label{Rk_def}
	R_k := r_{k+N-1},
	\end{equation}
	and
	\begin{equation}
	\label{Ck_def}
	C_k := c_{k+N-1}.
	\end{equation}
	Next, for any given $N \geq 1$, we define
	\begin{equation}
	\label{alpha_k}
	\alpha_k := \frac{1}{\left(1- C' \cdot \left( \frac{1}{2} \right)^{\frac{\sqrt{k+N-1}}{4}}\right)},
	\end{equation}	
	and
	\begin{equation}
	\label{beta_k}
	\beta_k := \frac{1}{\left(1+ C' \cdot \left( \frac{1}{2} \right)^{\frac{\sqrt{k+N-1}}{4}}\right)}.
	\end{equation}
\end{definition}
\noindent

\begin{rem}
	\label{alpha_beta_close_1}
	As $k \rightarrow \infty$, the sequence $(\alpha_k)$ decreases monotonically to $1$, and $(\beta_k)$ increases monotonically to $1$. We will always assume the integer $N \in \N$ is large enough so that for all $k \geq 1$ we have 
	\begin{equation}
	\label{alpha_beta_close_1_eqn}
	\frac{99}{100} < \beta_k < 1 < \alpha_k < \frac{101}{100}.
	\end{equation}
	The specific constants above are not important, we just need $\beta_k$ and $\alpha_k$ to be sufficiently close to $1$ for all large $k \geq 1$.
\end{rem}
\noindent

\begin{rem}
	We emphasize that the parameters in Definition \ref{bigR} depend on $N \in \N$; we omit this dependence in the notation for readability. We also remark that Definition \ref{parameter_defn2} implies that we still have $R_{k+1} = C_{k} (\frac{R_k}{2})^{n_{k}}$ for all $k \geq 1$. Finally, we have the equalities $R_1 = r_N$ and $n_1 = M_N$. We will occasionally switch between the two notations. 
\end{rem}

The inequalities (\ref{lemma3.2applies}), (\ref{xkesteq}), (\ref{xkest2_eqn1}), and (\ref{xkest2_eqn2}) that apply to $(r_j)_{j=1}^{\infty}$ can now be restated as follows by applying Definition \ref{bigR}.

\begin{lem}
	\label{Rkest}
	Fix an integer $N \geq 5$. Then for all $k \geq 1$, we have
	\begin{equation}
	\label{Rkest_eqn_2}
	R_k^{n_k} \cdot C_k \geq R_k^{n_{k-1}+1},  
	\end{equation}
	\begin{equation}
	\label{Rkest_eqn}
	R_{k+1} \geq 2^{-n_{k}} \cdot R_k^{n_{k-1} +1}, 
	\end{equation}
	\begin{equation}
	\label{Rkest_eqn3}
	R_k \geq 2^{2^{k + N -2}},\textrm{ and},
	\end{equation}
	\begin{equation}
	\label{Rkest_eqn4}
	R_{k+1} \geq 4 R_k^2.
	\end{equation}
\end{lem}

\noindent The next lemma describes some relationships among the $n_k$'s that we use freely throughout the paper.
\begin{lem}
	\label{nk}
	For all $k \geq 1$, we have 
	\begin{enumerate}
		\item $2 n_k = n_{k+1}$.
		\item $2^N + \sum_{j=1}^k n_j = n_{k+1}$.
	\end{enumerate}
\end{lem}
\begin{proof}
	(1) is obvious. (2) is a simple calculation:
	\begin{align*}
	2^N + \sum_{j=1}^k n_j = 2^N\left(1+ \sum_{j=0}^{k-1} 2^j\right)
	= 2^N \cdot 2^k 
	= 2^{N+k} = n_{k+1}.
	\end{align*}
	This proves the claim. 
\end{proof}

We denote open round annuli centered at the origin by $A(r,R) = \{z\,:\, r < |z| < R\}$. Next, we will define the following sequence of annuli. See Figures \ref{Akillustration} and \ref{annuli_picture}. 

\begin{definition}\label{annulidefinition}
	Given any $N \geq 1$, we define
	\begin{equation}
	\label{annuli}
	A_k = A\left(\frac{1}{4}R_k, 4 R_k\right), \,\,\, B_k = \overline{A\left(4R_k, \frac{1}{4}R_{k+1}\right)},\,\,\, V_k = A\left(\frac{2}{5}R_k, \frac{3}{5}R_k\right).
	\end{equation}
\end{definition}

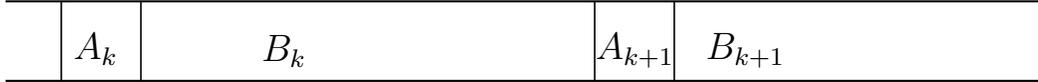
\begin{figure}[!h]
	\centering
	\scalebox{.8}{%% Creator: Inkscape 1.0 (4035a4fb49, 2020-05-01), www.inkscape.org
%% Accompanies image file 'mapping_diagrams.pdf' (pdf, eps, ps)
\begingroup%
  \makeatletter%
  \providecommand\color[2][]{%
    \errmessage{(Inkscape) Color is used for the text in Inkscape, but the package 'color.sty' is not loaded}%
    \renewcommand\color[2][]{}%
  }%
  \providecommand\transparent[1]{%
    \errmessage{(Inkscape) Transparency is used (non-zero) for the text in Inkscape, but the package 'transparent.sty' is not loaded}%
    \renewcommand\transparent[1]{}%
  }%
  \providecommand\rotatebox[2]{#2}%
  \newcommand*\fsize{\dimexpr\f@size pt\relax}%
  \newcommand*\lineheight[1]{\fontsize{\fsize}{#1\fsize}\selectfont}%
  \ifx\svgwidth\undefined%
    \setlength{\unitlength}{491.24999906bp}%
    \ifx\svgscale\undefined%
      \relax%
    \else%
      \setlength{\unitlength}{\unitlength * \real{\svgscale}}%
    \fi%
  \else%
    \setlength{\unitlength}{\svgwidth}%
  \fi%
  \global\let\svgwidth\undefined%
  \global\let\svgscale\undefined%
  \makeatother%
  \begin{picture}(1,0.07843951)%
    \lineheight{1}%
    \setlength\tabcolsep{0pt}%
    \put(0,0){\includegraphics[width=\unitlength,page=1]{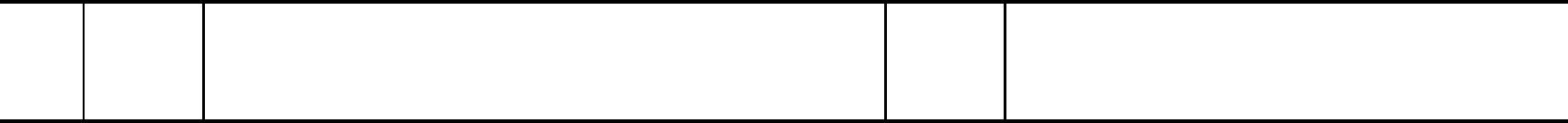}}%
    \put(0.2457096,0.02016946){\color[rgb]{0,0,0}\makebox(0,0)[lt]{\lineheight{1.25}\smash{\begin{tabular}[t]{l}\scalebox{1.5}{$B_k$}\end{tabular}}}}%
    \put(0.06525774,0.02188288){\color[rgb]{0,0,0}\makebox(0,0)[lt]{\lineheight{1.25}\smash{\begin{tabular}[t]{l}\scalebox{1.5}{$A_k$}\\\end{tabular}}}}%
    \put(0.56585122,0.02188287){\color[rgb]{0,0,0}\makebox(0,0)[lt]{\lineheight{1.25}\smash{\begin{tabular}[t]{l}\scalebox{1.5}{$A_{k+1}$}\\\end{tabular}}}}%
    \put(0.67175466,0.02188287){\color[rgb]{0,0,0}\makebox(0,0)[lt]{\lineheight{1.25}\smash{\begin{tabular}[t]{l}\scalebox{1.5}{$B_{k+1}$}\\\end{tabular}}}}%
  \end{picture}%
\endgroup%
}
	\caption{A visualization of $A_k$ and $B_k$ viewed on the cylinder. The annuli $A_k$ have constant modulus, and the annulus $B_k$ have very large and increasing moduli.}
	\label{annuli_picture}
\end{figure}

\noindent We now begin to describe the mapping behavior of $f$ in terms of the annuli in Definition \ref{annulidefinition}. 	

\begin{prop}
	\label{zeros}
	The zeros of $f_N$ that satisfy $|z| \geq \frac{1}{4} R_1$ are contained in $\cup_{k=1}^{\infty} A_k$. In fact, each $A_k$ contains exactly $n_{k}$ many simple zeros, each located inside $A(\frac{3}{5} R_k, \frac{5}{4} R_k)$.
\end{prop}
\noindent

\begin{proof}
	This follows for $f_N$ by combining (\ref{listing_of_zeros}) and Lemma \ref{AkEst}.
\end{proof}

\begin{lem}
	\label{PrePowerMap}
	There exists $M \in \N$ so that for all $N \geq M$, for all $k \geq 1$, and for all $z \in A( \frac{5}{4}R_k, \frac{3}{4} R_{k+1})$, we have
	\begin{equation}
	f_N(z) = C_{k+1}  \cdot (\phi_N^{-1}(z))^{n_{k+1}}.
	\end{equation}
\end{lem}
\begin{proof} 
	If $z \in A\left(\exp\left(\frac{\pi}{n_{k}}\right) \cdot R_k, R_{k+1}\right)$, then 
	$$f_N \circ \phi_N(z) = C_{k+1} \cdot z^{n_{k+1}}.$$
	Therefore, if $z \in \phi_N \left(A\left(\exp\left(\frac{\pi}{n_{k}}\right) \cdot R_k, R_{k+1}\right)\right)$, we must have
	$$f_N(z) = C_{k+1} \cdot (\phi_N^{-1}(z))^{n_{k+1}}.$$
	Therefore, it is sufficient to show that there is an $M$ so that for all $N \geq M$, and for all $k \geq 1$, we have 
	$$A\left( \frac{5}{4}R_k, \frac{3}{4} R_{k+1} \right) \subset \phi_N \left( A \left( \exp\left(\frac{\pi}{n_{k}}\right) \cdot R_k, R_{k+1}\right)\right).$$ 
	The existence of such an $M$ is a simple calculation using Lemma \ref{AkEst}.
\end{proof}

\begin{lem}
	\label{PowerMap}
	There exists $M \in \N$ so that for all $N \geq M$, for all $k\geq 1$, and for all $z \in  A( \frac{5}{4}R_k, \frac{3}{4} R_{k+1})$,
	\begin{equation}
	\beta^{n_{k+1}}_k C_{k+1}|z|^{n_{k+1}}\leq |f_N(z)| \leq \alpha^{n_{k+1}}_k C_{k+1} |z|^{n_{k+1}}.
	\end{equation}
\end{lem}
\begin{proof}
	This follows immediately from Lemma \ref{AkEst}, Definition \ref{bigR}, and Lemma \ref{PrePowerMap}.
\end{proof}

When estimating $f$ near $|z| = R_1$, we will require in some situations the following Lemma, which is similar to Lemma \ref{PowerMap}. Recall that by Definition \ref{parameter_defn} for all $z \in B(0,r_N - 1)$, we have $f_N(z) = q_N \circ \phi_N^{-1}(z)$, where $q_N(z) = c_N z^{M_N} + r_N z = C_1 z^{n_1} + R_1 z.$
\begin{lem}
	\label{the_k_equals_1_lemma_qN}
	There exists $M \in \N$ so that for all $N \geq M$, we have:
	\begin{equation}
	\label{the_k_equals_1_lemma_qN_eqn}
	\frac{1}{2} c_N |z|^{M_N} \leq |q_N(z)| \leq 2 c_N |z|^{M_N} \emph{ for all } z \in A(\frac{1}{20}R_1, \frac{19}{20}R_1).
	\end{equation}
\end{lem}
\begin{proof}
	This is a simple but somewhat tedious application of Lemma \ref{Rkest}. By the triangle inequality, we obtain for all $z \in A(\frac{1}{20}R_1, \frac{19}{20}R_1)$ that
	\begin{equation*}
	c_N |z|^{M_N} \left(1 - \frac{r_N}{c_N |z|^{M_{N} - 1}}\right) \leq |q_N(z)| \leq c_N|z|^{M_N}\left( 1 + \frac{r_N}{c_N |z|^{M_N - 1}}\right). 
	\end{equation*} 
	On the one hand, we have by Lemma \ref{rkinequalities} that
	\begin{align*}
	\max_{z \in A(\frac{1}{20}R_1, \frac{19}{20}R_1)} \left( 1 + \frac{r_N}{c_N |z|^{M_N - 1}}\right) &= 1 + \frac{r_N}{c_N(\frac{1}{20}r_N)^{M_N - 1}} \\
	&= 1 + \frac{r_N \frac{1}{20}r_N}{(\frac{1}{10})^{M_N} c_N(\frac{r_N}{2})^{M_N}} \\
	&= 1 + \frac{10^{M_N}}{20} \cdot \frac{r_N^{2}}{r_{N+1}} \\
	&\leq 1 + \frac{10^{M_N}}{20} \cdot \frac{r_N^2}{2^{-M_N} r_N^{M_{(N-1)}}} \\
	&\leq 1 + \frac{20^{M_N}}{20} \cdot \frac{1}{r_N^{M_{(N-2)}}} \\
	&= 1 + \frac{1}{20} \left( \frac{20}{r_N^{\frac{1}{4}}} \right)^{M_N}
	\end{align*}
	By Lemma \ref{Rkest}, there exists $M$ so that for all $N \geq M$ we have $r_N^{\frac{1}{4}} \geq 40$ and therefore we obtain
	\begin{equation*}
	\max_{z \in A(\frac{1}{20}R_1, \frac{19}{20}R_1)} \left( 1 + \frac{r_N}{c_N |z|^{M_N - 1}}\right) \leq 1+\frac{1}{20} \cdot \left(\frac{1}{2}\right)^{M_N} < 2. 
	\end{equation*}
	Therefore, we obtain for all $z \in z \in A(\frac{1}{20}R_1, \frac{19}{20}R_1)$ that
	\begin{equation}
	\label{q_N_upper_bound}
	|q_N(z)| \leq 2 c_N|z|^{M_N}.
	\end{equation}
	The proof of the other inequality is similar.
	%Indeed, we use Lemma \ref{rkinequalities} to verify that 
	%\begin{align*}
	%\min_{z \in A(\frac{1}{20}R_1, \frac{19}{20}R_1)}   \left(1 - \frac{r_N}{c_N |z|^{M_{N} - 1}}\right) &\geq 1 - \frac{r_N}{c_N (\frac{1}{20}r_N)^{M_{N} - 1}} \\
	%	&= 1 - \frac{r_N \frac{1}{20}r_N}{(\frac{1}{10})^{M_N} c_N(\frac{r_N}{2})^{M_N}} \\
	%	&= 1 - \frac{10^{M_N}}{20}  \cdot \frac{r_N^2}{r_{N+1}} \\
	%	& \geq  1 - \frac{10^{M_N}}{20} \cdot \frac{r_N^2}{2^{-M_N}r_N^{M_{(N-1)}}} \\
	%	& \geq 1 - \frac{20^{M_N}}{20}\cdot \frac{1}{r_N^{M_{(N-2)}}} \\
	%	&= 1 - \frac{1}{20} %\left(\frac{20}{r_N^{\frac{1}{4}}}\right)^{M_N}.
	%\end{align*}
	%Since $r_N^{\frac{1}{4}} \geq 40$, by perhaps taking $M$ larger we can obtain for all $N \geq M$ that
	%\begin{equation*}
	%\min_{z \in A(\frac{1}{20}R_1, \frac{19}{20}R_1)}   \left(1 - \frac{r_N}{c_N |z|^{M_{N} - 1}}\right) \geq 1 - \frac{1}{20} \left(\frac{1}{2}\right)^{M_N} > \frac{1}{2}.
	%\end{equation*}
	%Therefore, we obtain
	%\begin{equation}
	%\label{q_N_lower_bound}
	%\min_{z \in A(\frac{1}{20}R_1, \frac{19}{20}R_1)} |q_N(z)| \geq %\frac{1}{2}c_N |z|^{M_N}
	%\end{equation}
	%Inequalities (\ref{q_N_upper_bound}) and (\ref{q_N_lower_bound}) combine to prove (\ref{the_k_equals_1_lemma_qN_eqn}), as desired.
\end{proof}

\begin{lem}
	\label{the_k_equal_1_lemma}
	For all $N$ sufficiently large, we have
	\begin{equation}
	\label{the_k_equal_1_lemma_eqn}
	\frac{1}{2} \beta_1^{M_N} c_N |z|^{M_N} \leq |f_N(z)| \leq 2\alpha_1^{M_N} c_N |z|^{M_N} \emph{ for all } z \in A(\frac{1}{10}R_1, \frac{9}{10}R_1).
	\end{equation}
\end{lem}
\begin{proof}
	By Lemma \ref{AkEst}, there exists $M \in \N$ so that for all $N \geq M$ we have
	\begin{equation*}
	\phi^{-1}_N\left(A\left(\frac{1}{10}R_1, \frac{9}{10}R_1\right)\right) \subset A\left(\frac{1}{20}R_1, \frac{19}{20}R_1\right).
	\end{equation*}
	By perhaps choosing $M$ larger we have for all $N \geq M$ that $R_1 - 1 \geq \frac{19}{20}R_1$ as well. Then, for all $z \in A(\frac{1}{10}R_1, \frac{9}{10}R_1)$, we have by Lemma \ref{AkEst} and Lemma \ref{the_k_equals_1_lemma_qN} that
	\begin{equation}
	\label{the_k_equal_1_lemma_upper}
	\max_{z \in A(\frac{1}{10}R_1, \frac{9}{10}R_1)} |f_N(z)| = \max_{z \in A(\frac{1}{10}R_1, \frac{9}{10}R_1)} |q_N(\phi_N^{-1}(z))| 
	\leq 2 c_N |\phi_N^{-1}(z)|^{M_N} 
	\leq 2 c_N \alpha_1^{M_N} |z|^{M_N}.
	\end{equation}
	Similarly, we obtain
	\begin{equation}
	\label{the_k_equal_1_lemma_lower}
	\min_{z \in A(\frac{1}{10}R_1, \frac{9}{10}R_1)} |f_N(z)| = \min_{z \in A(\frac{1}{10}R_1, \frac{9}{10}R_1)} |q_N(\phi_N^{-1}(z))| 
	\geq \frac{1}{2} c_N |\phi_N^{-1}(z)|^{M_N}
	\geq \frac{1}{2} c_N \beta_1^{M_N}|z|^{M_N}
	\end{equation}
	This proves the claim.
\end{proof}

We are now ready to prove some basic lemmas about the macroscopic mapping behavior of the function $f_N$. First we will need the following basic lemma. 
\begin{lem}
	\label{holomorphic_annulus}
	Suppose that $g$ is holomorphic on an annulus $W = A(a,b)$ and continuous up to the boundary of $W$. Let $U = A(c,d)$.
	\begin{enumerate}
		\item If $|g(z)| \leq c$ on $|z| = a$ and $|g(z)| \geq d$ on $|z| = b$, then $U \subset g(W)$.
		\item Suppose $g$ has no zeros in $W$ and that $g(\partial W) \subset \overline{U}$. Then $g(W) \subset \overline{U}$.
	\end{enumerate}
\end{lem}
\begin{proof}
	The first item uses the fact that holomorphic maps are open. The second item is an application of the maximum principle. A detailed proof can be found in Lemma 11.1 of \cite{Bis18}.
\end{proof}

Next, we prove the following lemma about the mapping behavior on $A_k$, where we will see the dynamics of $f_N$ are the most interesting. 

%\begin{figure}
%	\centering
%	\scalebox{.9}{\input{map Ak.pdf_tex}}
%	\caption{The dynamics of $f_N$ are very expansive on $V_k$; $f_N$ maps the annulus $V_k$ onto an annulus that contains $A_{k+1}$.}
%	\label{Ak_picture}
%\end{figure}

\begin{lem}
	\label{Ak}
	There exists $M \in \N$ such that for all $N \geq M$ and for all $k \geq 1$, we have
	\begin{equation}	
	A_{k+1} \subset f_N(V_k) \subset f_N(A_k).
	\end{equation}
\end{lem}

\begin{proof}
	First, we prove the case of $k \geq 2$. In this setting, by Lemma \ref{PowerMap}
	\begin{align*}
	\max_{|z| = \frac{2}{5}R_k} |f_N(z)| &\leq \max_{|z| = \frac{2}{5}R_k}  \alpha^{n_{k}}_k C_{k}  |z|^{n_{k}} 
	\leq  \alpha^{n_{k}}_k C_k \left( \frac{2}{5} R_k \right)^{n_{k}} 
	\leq
	\alpha^{n_{k}}_k \left( \frac{4}{5} \right)^{n_{k}} \cdot C_{k} \cdot \left(\frac{1}{2}R_k\right)^{n_{k}}\\
	\end{align*}
	By (\ref{alpha_beta_close_1_eqn}), we have $\alpha_k \cdot\frac{4}{5} \leq \frac{7}{8}$. Therefore, if $N \geq 5$, since $R_{k+1}= C_{k} \cdot (\frac{R_k}{2})^{n_{k}}$, we end up with 
	\begin{equation}
	\label{Inner_Vk}
	\max_{|z| = \frac{2}{5}R_k} |f_N(z)| \leq  \left(\frac{7}{8}\right)^{n_{k}} R_{k+1} < \frac{1}{4} R_{k+1}.
	\end{equation}
	Next, observe that by Lemma $\ref{PowerMap}$, we have
	\begin{align*}
	\min_{|z| = \frac{3}{5}R_k} |f_N(z)| \geq \min_{|z| = \frac{3}{5}R_k} \beta_k^{n_k} C_k |z|^{n_{k}} 
	= \beta_k^{n_k} C_k \left(\frac{3}{5} R_k\right)^{n_{k}} 
	= \beta_k^{n_k} \cdot \left(\frac{6}{5}\right)^{n_k} \cdot C_{k} \cdot \left(\frac{1}{2}R_k\right)^{n_{k}}.
	\end{align*}
	By (\ref{alpha_beta_close_1_eqn}), we have $\beta_k\cdot \frac{6}{5} \geq \frac{9}{8}$. If $N \geq 5$, we end up with
	\begin{equation}
	\label{Outer_Vk}
	\min_{|z| = \frac{3}{5}R_k} |f_N(z)| \geq \left(\frac{9}{8}\right)^{n_{k}} R_{k+1} > 4 R_{k+1}.
	\end{equation}
	The lemma for the case of $k \geq 2$ now follows from Lemma \ref{holomorphic_annulus}, part $(1)$. 
	
	The case of $k =1$ is almost exactly the same, except we now have to use Lemma \ref{the_k_equal_1_lemma}. By following the exact same steps as above, we obtain since $N \geq 5$ that
	
	\begin{equation}
	\label{Inner_V1}
	\max_{|z| = \frac{2}{5}R_1} |f_N(z)| \leq 2 \left(\frac{7}{8}\right)^{n_1}R_{2} < \frac{1}{4} R_2, 
	\end{equation}
	and,
	\begin{equation}
	\label{Outer_V1}
	\min_{|z| = \frac{3}{5}R_1} |f_N(z)| \geq \frac{1}{2}\left(\frac{9}{8}\right)^{n_{1}} R_{2} > 4 R_2
	\end{equation}
	The lemma for the case of $k = 1$ now follows from Lemma \ref{holomorphic_annulus}, part $(1)$.
\end{proof}

\begin{lem}
	\label{Bk}
	There exists $M \in \N$ so that for all $N \geq M$ and for all $k \geq 1$, we have 
	\begin{equation}
	f_N(B_k)  \subset  B_{k+1}.
	\end{equation}
\end{lem}

\begin{figure}
	\centering
	\scalebox{.8}{%% Creator: Inkscape 1.0 (4035a4fb49, 2020-05-01), www.inkscape.org
%% Accompanies image file 'map_Bk.pdf' (pdf, eps, ps)
\begingroup%
  \makeatletter%
  \providecommand\color[2][]{%
    \errmessage{(Inkscape) Color is used for the text in Inkscape, but the package 'color.sty' is not loaded}%
    \renewcommand\color[2][]{}%
  }%
  \providecommand\transparent[1]{%
    \errmessage{(Inkscape) Transparency is used (non-zero) for the text in Inkscape, but the package 'transparent.sty' is not loaded}%
    \renewcommand\transparent[1]{}%
  }%
  \providecommand\rotatebox[2]{#2}%
  \newcommand*\fsize{\dimexpr\f@size pt\relax}%
  \newcommand*\lineheight[1]{\fontsize{\fsize}{#1\fsize}\selectfont}%
  \ifx\svgwidth\undefined%
    \setlength{\unitlength}{491.2500189bp}%
    \ifx\svgscale\undefined%
      \relax%
    \else%
      \setlength{\unitlength}{\unitlength * \real{\svgscale}}%
    \fi%
  \else%
    \setlength{\unitlength}{\svgwidth}%
  \fi%
  \global\let\svgwidth\undefined%
  \global\let\svgscale\undefined%
  \makeatother%
  \begin{picture}(1,0.29014155)%
    \lineheight{1}%
    \setlength\tabcolsep{0pt}%
    \put(0,0){\includegraphics[width=\unitlength,page=1]{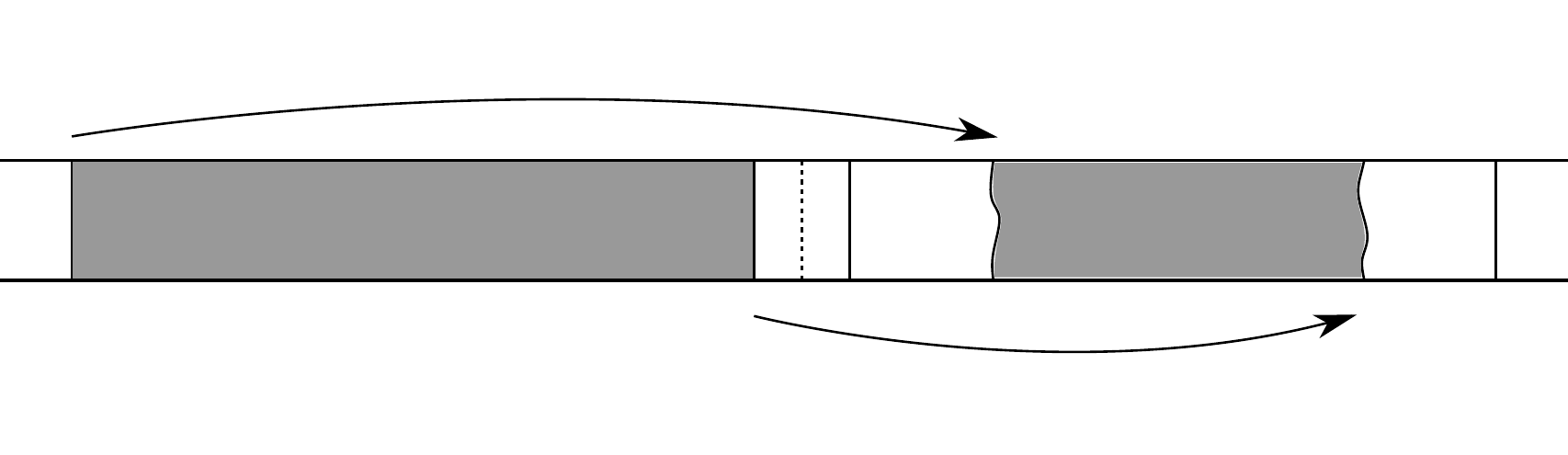}}%
    \put(0.33889837,0.24374376){\color[rgb]{0,0,0}\makebox(0,0)[lt]{\lineheight{1.25}\smash{\begin{tabular}[t]{l}\scalebox{1.5}{$f_N$}\end{tabular}}}}%
    \put(0.67872112,0.01440237){\color[rgb]{0,0,0}\makebox(0,0)[lt]{\lineheight{1.25}\smash{\begin{tabular}[t]{l}\scalebox{1.5}{$f_N$}\end{tabular}}}}%
    \put(0.22499671,0.1323112){\color[rgb]{0,0,0}\makebox(0,0)[lt]{\lineheight{1.25}\smash{\begin{tabular}[t]{l}\scalebox{1.5}{$B_k$}\\\end{tabular}}}}%
    \put(0.71213433,0.13874666){\color[rgb]{0,0,0}\makebox(0,0)[lt]{\lineheight{1.25}\smash{\begin{tabular}[t]{l}\scalebox{1.5}{$B_{k+1}$}\end{tabular}}}}%
  \end{picture}%
\endgroup%
}
	\caption{$f_N$ maps the annulus $B_k$ into the annulus $B_{k+1}$. The picture is not to scale; in reality, the modulus of $B_{k+1}$ is much larger than the modulus of $f_N(B_k)$.}
	\label{Bk_picture}
\end{figure}
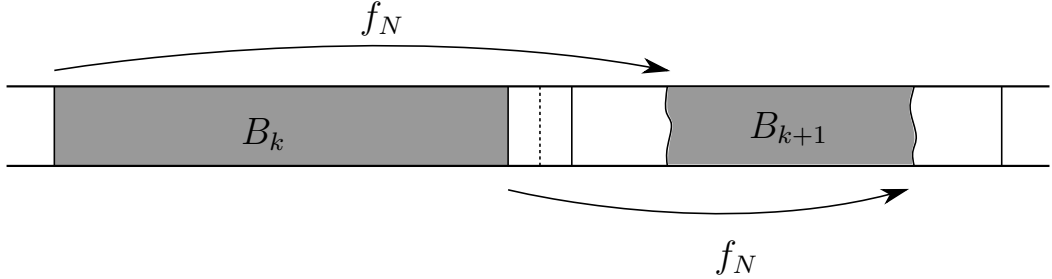

\begin{proof}
	We will adopt a similar strategy to Lemma \ref{Ak}, using Lemma \ref{holomorphic_annulus}, part $(2)$. First, we make the important observation that $\frac{C_{k+1}}{C_k} = R_k^{-n_k}$ for all $k \geq 1$. Next, observe that if $|z| = 4 R_k$, then we have by Lemma \ref{PowerMap} 
	\begin{align*}
	\max_{|z| = 4R_k} |f_N(z)| &\leq \max_{|z| = 4R_k}  \alpha_k^{n_{k+1}} C_{k+1}|z|^{n_{k+1}} \\
	&=   \alpha_k^{n_{k+1}} C_{k+1} (4R_k)^{n_{k+1}} \\
	&=   \alpha_k^{n_{k+1}}  8^{n_{k+1}} \cdot C_{k} R_k^{-n_k}\cdot \left(\frac{1}{2}R_k\right)^{n_{k+1}} \\
	&=  \alpha_k^{n_{k+1}}  8^{n_{k+1}} C_{k} \left(\frac{1}{2}R_k\right)^{n_{k}} \cdot R_k^{-n_k} \left(\frac{1}{2}R_k\right)^{n_{k}}\\
	&=  \alpha_k^{n_{k+1}} 8^{n_{k+1}} 2^{-n_{k}} R_{k+1}. \\
	&=  (\alpha_k^2)^{n_{k}} 32^{n_{k}} R_{k+1}
	\end{align*}
	By (\ref{alpha_beta_close_1_eqn}), we have $\alpha_k^2 \leq 2$ for all $k \geq 1$.
	By Lemma \ref{Rkest}, there exists $M$ so that for all $N \geq M$, and for all $k \geq 1$ we have $R_k^{\frac{1}{2}} > 256$. Therefore,
	\begin{equation}
	\label{Setup1}
	\frac{64^{n_{k+1}} R_{k+1}}{R_{k+2}} \leq \frac{64^{n_{k+1}}R_{k+1}}{\left( \frac{1}{2}\right)^{n_{k+1}} R_{k+1}^{n_k+1}} = \left(\frac{128}{R^{\frac{1}{2}}_{k+1}} \right)^{n_{k+1}} \leq \left(\frac{1}{2}\right)^{n_{k+1}}.
	\end{equation}
	By (\ref{Setup1}), and since $N \geq 5$, we have for all $k \geq 1$
	\begin{equation}
	\label{4R_k_Max}
	\max_{|z| = 4R_k} |f_N(z)| \leq 64^{n_{k+1}} R_{k+1} \leq \left( \frac{1}{2}\right)^{n_{k+1}} R_{k+2} < \frac{1}{8} R_{k+2}.
	\end{equation}
	
	Next observe that by Lemma \ref{PowerMap} we similarly have
	\begin{align*}
	\min_{|z| = 4R_k} |f_N(z)| &\geq  \min_{|z| = 4R_k}  \beta_k^{n_{k+1}} C_{k+1}|z|^{n_{k+1}} \\
	.	&= \beta_k^{n_{k+1}} C_k R_k^{-n_k} \cdot 2^{n_{k+1}} \left(\frac{R_k}{2}\right)^{n_{k}} \cdot \left(\frac{R_k}{2}\right)^{n_{k}} \\
	&=(\beta^2_k)^{n_{k}} 2^{n_k} R_{k+1} = (2\beta_k^2)^{n_k} R_{k+1}
	\end{align*} 
	By (\ref{alpha_beta_close_1_eqn}) we have $\beta_k^2 \cdot 2 \geq  \frac{3}{2}$. Therefore, since $N \geq 5$, we have for all $k \geq 1$ that
	\begin{equation}
	\label{4R_k_Min}
	\min_{|z| = 4R_k} |f_N(z)| \geq \left( \frac{3}{2} \right)^{n_{k}} R_{k+1} >  8 R_{k+1}.
	\end{equation}
	Therefore, by (\ref{4R_k_Max}) and (\ref{4R_k_Min}), for all $k \geq 1$ we have
	\begin{equation}
	\label{Inner_Bk}
	f_N(|z| = 4R_k) \subset A\left(8R_k,\frac{1}{8}R_{k+1}\right)  \subset B_{k+1}.
	\end{equation}	
	
	We can use similar techniques as above to analyze the behavior of $f_N$ on the outermost boundary of $B_k$. Indeed, we have 
	
	%The definition of $c_k$ from Theorem \ref{mainthm} along with (\ref{Ck_def}) implies that $C_{k+1} = C_{k+2} R_{k+1}^{n_{k+1}}$. By combining this with Lemma \ref{PowerMap}, we have
	%\begin{align*}
	%\max_{|z| = \frac{1}{4}R_{k+1}} |f_N(z)| &\leq \max_{|z| = \frac{1}{4}R_{k+1}}  \alpha_k^{n_{k+1}}C_{k+1}|z|^{n_{k+1}} \\
	%&=  \alpha_k^{n_{k+1}}  \left(\frac{1}{2}\right)^{n_{k+1}}  \cdot C_{k+1} \cdot \left( \frac{R_{k+1}}{2}\right)^{n_{k+1}}   \\
	%&=  \alpha_k^{n_{k+1}}  \left(\frac{1}{2}\right)^{n_{k+1}} R_{k+2}
	%\end{align*}
	% By (\ref{alpha_beta_close_1_eqn}) we have $\alpha_k  \cdot \frac{1}{2} \leq \frac{3}{4}$. Therefore, since $N \geq 5$, we have 
	
	\begin{equation}
	\label{14R_k_max}
	\max_{|z| = \frac{1}{4}R_{k+1}} |f_N(z)|  < \frac{1}{8} R_{k+2},
	\end{equation}
	and,
	%	\begin{align*}
	%	\min_{|z| = \frac{1}{4}R_{k+1}} |f_N(z)| &\geq  \min_{|z| = \frac{1}{4}R_{k+1}}   \beta_k^{n_{k+1}} C_{k+1} |z|^{n_{k+1}} \\
	%	&= \beta_k^{n_{k+1}} C_{k+1} \left(\frac{1}{4}\right)^{n_{k+1}} R_{k+1}^{n_{k+1}} \\
	%	&=  \beta_k^{n_{k+1}} \left(\frac{1}{4}\right)^{n_{k+1}} C_{k+1} R_{k+1}^{n_{k+1} -1} R_{k+1}.
	%	\end{align*}
	%	By Lemma \ref{Rkest}, we have $C_{k+1} R_{k+1}^{n_{k+1}-1} \geq R_{k+1}^{n_{k}}$, and by (\ref{alpha_beta_close_1_eqn}) we have $\beta_k  \frac{1}{4} \geq \frac{1}{8}$. It follows that 
	%	\begin{equation*}
	%		\min_{|z| = \frac{1}{4}R_{k+1}} |f_N(z)| \geq \left(\frac{1}{8}\right)^{n_{k+1}} R_{k+1}^{n_{k}} R_{k+1}.
	%	\end{equation*}
	%	By Lemma \ref{Rkest} there exists $M$ so that for all $N \geq M$, we have for all $k \geq 1$ that $R_{k+1}^{1/2} > 64$ and therefore we have
	%	\begin{equation*}
	%	\left(\frac{1}{8}\right)^{n_{k+1}} (R_{k+1}^{\frac{1}{2}})^{n_{k+1}} > 8^{n_{k+1}}.
	%	\end{equation*}
	%	It follows that for all such $N$, we have for all $k \geq 1$ since $N \geq 5$ that 
	\begin{equation}
	\label{14R_k_Min}
	\min_{|z| = \frac{1}{4}R_{k+1}} |f_N(z)|  > 8 R_{k+1}.
	\end{equation}
	Therefore, by (\ref{14R_k_Min}) and (\ref{14R_k_max}) we have 
	\begin{equation}
	\label{Outer_Bk}
	f_N(|z| = \frac{1}{4} R_{k+1}) \subset A\left(8R_k,\frac{1}{8}R_{k+1}\right) \subset B_{k+1}.
	\end{equation}  
	As we commented at the start, this proves the lemma.
\end{proof}

We conclude this Section by recording the location of the critical points and values of $f$ in relation to the annuli $A_k$, $B_k$.

\begin{prop}
	\label{Fatou_critical_points}
	There exists $M \in \N$ so that for all $N \geq M$, all critical points $z$ of $f_N$ with $|z| > \frac{1}{4} R_1$ belong to $\cup_{k=1}^{\infty} A_k$. Moreover, if $z \in A_k$ is a critical point, then $f_N(z) \in B_{k+1}$.
\end{prop}
\noindent

\begin{proof}
	By Proposition \ref{critical_point_listing} and Lemma \ref{AkEst}, there exists $M$ so that for all $N \geq M$ all critical points of $f_N$ satisfying $|z| \geq \frac{1}{4} R_1$ belong to $\cup_{k=1}^{\infty} A_k$. If $z \in A_k$ is a critical point, then Proposition \ref{critical_point_listing} also says that $|f_N(z)| = C_k R_k^{n_k}$. To see that $C_k R_k^{n_k} \in B_{k+1}$, first notice that we have the identity
	\begin{equation}
	\label{identity_CkRknk}
	C_k R_k^{n_k} = 2^{n_k} C_k (\frac{1}{2}R_k)^{n_k} = 2^{n_k} R_{k+1}.
	\end{equation}
	It follows immediately from (\ref{identity_CkRknk}) that since $N \geq 5$ we have
	\begin{equation}
	\label{crit_value_lower_bound}
	C_k R_k^{n_k} = 2^{n_k} R_{k+1} > 8R_{k+1}.
	\end{equation}
	On the other hand, we have by Lemma \ref{Rkest} that
	\begin{equation*}
	\frac{2^{n_k}R_{k+1}}{R_{k+2}} \leq \frac{8^{n_k}}{R_{k+1}^{n_k}} \leq \left(\frac{1}{2}\right)^{n_k}.
	\end{equation*}
	So since $N \geq 5$ we obtain from (\ref{identity_CkRknk}) that
	\begin{equation}
	\label{crit_value_upper_bound}
	C_k R^{n_k}_k \leq \left( \frac{1}{2} \right)^{n_{k}} R_{k+2} < \frac{1}{8} R_{k+2}
	\end{equation}
	Therefore, by (\ref{crit_value_lower_bound}) and (\ref{crit_value_upper_bound}) we have $f_N(z) \in B_{k+1}$.
\end{proof}

\begin{rem}
	\label{N is large}
	For the rest of the paper, we will always assume that $N$ is large enough so that all of the statements and inequalities in this section are valid.
\end{rem}

\section{Mapping Behavior near $0$}
\label{Origin}

Having studied the mapping behavior of $f$ in the region $|z|>R_1/4$ in Section \ref{Global Mapping}, we now study in Section \ref{Origin} the mapping behavior of $f$ in $|z|<R_1/4$. Recall that in $|z|<R_1/4$, the mapping $f$ satisfies $f(z)=q_N\circ\phi_N^{-1}(z)$, where $q_N(z) = c_N z^{M_N} + r_N z$. This polynomial was chosen so as to have a Cantor Julia set of dimension $\ll 1$. Thus, when we consider $f$ as a polynomial-like mapping by restricting the domain of $f$ to a subdomain of $|z|<R_1/4$, we will see that the Julia set of this polynomial-like mapping has dimension $\ll 1$.

We begin with the following lemma about the polynomial $q_N(z)$.
\begin{lem}
	\label{qN}
	Let $q_N(z) = c_N z^{M_N} + r_N \cdot z$. Then the derivative of $q_N(z)$ is 
	\begin{equation}
	\label{derivative_qN}
	q_N'(z) = c_N M_N z^{M_N-1} + r_N.
	\end{equation}
	The non-zero zeros of $q_N$ are given by
	\begin{equation}
	\label{zeros_qN}
	z = \left(\frac{-r_N}{c_N} \right)^{\frac{1}{M_N-1}}.
	\end{equation} 
	The critical points of $q_N$ are given by
	\begin{equation}
	\label{crit_qN}
	z = \left( \frac{-r_N}{c_N M_N}\right)^{\frac{1}{M_N-1}}.
	\end{equation}
	The critical values of $q_N$ lie on the circle
	\begin{equation}
	\label{critvalue_qN}
	|z| = \left( \frac{r_N}{c_N M_N}\right)^{\frac{1}{M_N-1}} \cdot r_N \cdot \left(1- \frac{1}{M_N} \right)
	\end{equation}
	The value of $|q_N'(z)|$ when $z$ is a zero of $q_N$ satisfies
	\begin{equation}
	\label{zeroexp_qN}
	|q_N'(z)| = r_N (M_N -1).
	\end{equation}
\end{lem}
\begin{proof}
	These are all simple calculations. We only verify (\ref{critvalue_qN}). If $z$ is a critical point of $q_N$ then we calculate that
	\begin{eqnarray*}
		q_N(z) &=& c_N \left( \frac{-r_N}{c_N M_N}\right)^{\frac{M_N}{M_N-1}} + r_N \left( \frac{-r_N}{c_N M_N}\right)^{\frac{1}{M_N-1}} \\
		&=& \left( \frac{-r_N}{c_N M_N}\right)^{\frac{1}{M_N-1}} \cdot \left(r_N - \frac{r_N}{M_N}\right) \\
		&=& \left( \frac{-r_N}{c_N M_N}\right)^{\frac{1}{M_N-1}} \cdot r_N \cdot \left( 1 - \frac{1}{M_N} \right).
	\end{eqnarray*}
	The result follows upon taking the absolute value.
\end{proof}

We will now prove that the critical values of $q_N$ map to $B_1$. This will be crucial in dimension estimates which require coverings that are built by considering the inverse $f^{-1}$. First we need the following technical Lemma:

\begin{lem}
	\label{quotient_est}
	For all $N \geq 10$, we have
	\begin{equation}
	\label{quotient_est_eqn}
	2^{M_{N-7}}r_{N-1} \leq \left( \frac{r_N}{c_N M_N} \right)^{\frac{1}{M_N - 1}} \leq \frac{1}{\sqrt{2}}r_{N-1}^2.
	\end{equation}
\end{lem}
\begin{proof}
	Recall first that $r_N = c_{N-1} (\frac{r_{N-1}}{2})^{M_{N-1}}$ by definition. Therefore, we calculate
	\begin{equation}
	\label{quotient_est_eqn1}
	\frac{r_N}{c_N} = \frac{1}{2^{M_{(N-1)}}} \frac{c_{N-1} r^{M_{(N-1)}}_{N-1}}{c_N}
	= \frac{1}{2^{M_{(N-1)}}} \frac{r_{N-1}^{M_{(N-1)}}}{r_{N-1}^{-M_{(N-1)}}} = \frac{1}{2^{M_{(N-1)}}} r_{N-1}^{M_N}.
	\end{equation} 	
	First we prove the upper bound for (\ref{quotient_est_eqn}). First note that we have for all $N \geq 5$ that
	\begin{equation}
	\label{quotient_est_eqn2}
	2^{- \frac{M_{(N-1)}}{M_N - 1}} \leq \frac{1}{\sqrt{2}},
	\end{equation}
	so that by combining (\ref{quotient_est_eqn1}) and (\ref{quotient_est_eqn2}) we obtain
	\begin{equation}
	\label{quotient_est_eqn_right}
	\left( \frac{r_N}{c_N M_N} \right)^{\frac{1}{M_N - 1}} = \left(\frac{1}{M_N}\right)^{\frac{1}{M_N -1}} \cdot 2^{- \frac{M_{(N-1)}}{M_N - 1}} \cdot r_{N-1} \cdot r_{N-1}^{\frac{1}{M_N - 1}} \leq \frac{1}{\sqrt{2}} r_{N-1}^2.
	\end{equation}
	To prove the lower bound for (\ref{quotient_est_eqn}), since $N \geq 5$ note that we have
	\begin{equation}
	\label{quotient_est_eqn3}
	\left( \frac{1}{M_N} \right)^{\frac{1}{M_{N}-1}} \geq \frac{1}{2} \textrm{, and } 2^{\frac{-M_{N-1}}{M_N -1}} \geq \frac{1}{2}.
	\end{equation}
	By two applications of Corollary \ref{rkinequalities}, since $N \geq 10$ we obtain
	\begin{equation}
	\label{quotient_est_eqn4}
	r_{N-1}^{\frac{1}{M_N-1}} \geq 2^{-\frac{M_{(N-2)}}{M_N -1}} \cdot r_{N-2}^{\frac{M_{(N-3)}}{M_N - 1}}\geq 2^{-1/2} r_{N-2}^{1/8} \geq 2^{-1/2} 2^{M_{N-3}/8} = 2^{M_{N-6} - 1/2}.
	\end{equation}
	Therefore, in a similar way to (\ref{quotient_est_eqn_right}), except this time using (\ref{quotient_est_eqn3}) and (\ref{quotient_est_eqn4}), we have 
	\begin{equation}
	\label{quotient_est_eqn_left}
	\left( \frac{r_N}{c_N M_N} \right)^{\frac{1}{M_N - 1}} = 2^{\frac{-M_{(N-1)}}{M_N -1}} \cdot \left(\frac{1}{M_N}\right)^{\frac{1}{M_N -1}} \cdot  r_{N-1} \cdot r_{N-1}^{\frac{1}{M_N - 1}} \geq 2^{M_{N-6} - 1/2 - 2} r_{N-1} \geq 2^{M_{N-7}} r_{N-1}.
	\end{equation}
	Equations (\ref{quotient_est_eqn_right}) and (\ref{quotient_est_eqn_left}) combine to prove (\ref{quotient_est_eqn}).
\end{proof}

\begin{lem}
	\label{crits_B1}
	For $N \geq 10$, the critical values of $q_N$ belong to $B_1$, and the critical points satisfy $|z| \leq  r_{N-1}^2$.
\end{lem}
\begin{proof}
	If $z$ is a critical point, by Lemma \ref{qN} and Lemma \ref{quotient_est} we have
	\begin{equation}
	\label{crits_B1_eqn1}
	|z| = \left( \frac{r_N}{c_N M_N} \right)^{\frac{1}{M_N - 1}} \leq \frac{1}{\sqrt{2}}r_{N-1}^2.
	\end{equation}
	Recall that by Corollary \ref{rkinequalities} that if $N \geq 10$ we have $r_{N-1}^2\leq \frac{1}{4}r_N$ and $r_N^2 \leq \frac{1}{4}r_{N+1}$. So by Lemma \ref{qN} and Lemma \ref{quotient_est}, if $z$ is a critical point then 
	\begin{equation}
	\label{crits_B1_eqn3}
	|q_N(z)| = \left( \frac{r_N}{c_N M_N} \right)^{\frac{1}{M_N - 1}} r_N \left(1 - \frac{1}{M_N} \right) \leq \frac{1}{\sqrt{2}} r_{N-1}^2 r_N \leq \frac{1}{4\sqrt{2}} r_N^2 < \frac{1}{16\sqrt{2}} r_{N+1},
	\end{equation}
	and,
	\begin{equation}
	\label{crits_B1_eqn4}
	|q_N(z)| \geq 2^{M_{(N-7)}} r_{N-1} r_N \cdot \left( 1 - \frac{1}{M_N} \right) > 8 r_N.
	\end{equation}	
	Therefore, by Definition \ref{bigR} we have the critical values of $q_N$ contained in $B_1$, as desired. 
\end{proof}

\noindent Now we introduce a polynomial-like mapping by restricting $f$ to a subdomain $U$, defined as follows:

\begin{definition}
	\label{poly_like_defs}
	For the rest of this section, we will use the following definitions.
	\begin{enumerate}
		\item Define $D = B(0, \frac{1}{4}R_1)$.
		\item Define $r = 16 r_{N-1}^2 R_1$. This is the same as $r = 16r_{N-1}^2 r_N$ by (\ref{Rk_def}).
		\item Define $V = B(0,r)$ and $U' = q_N^{-1}(V)$. 
		\item Define $U = \phi_N(U')$.
	\end{enumerate}
\end{definition}	

\begin{figure}[!h]
	\centering
	\scalebox{.8}{%% Creator: Inkscape 1.0 (4035a4fb49, 2020-05-01), www.inkscape.org
%% Accompanies image file 'polynomial_like_schematic.pdf' (pdf, eps, ps)
\begingroup%
  \makeatletter%
  \providecommand\color[2][]{%
    \errmessage{(Inkscape) Color is used for the text in Inkscape, but the package 'color.sty' is not loaded}%
    \renewcommand\color[2][]{}%
  }%
  \providecommand\transparent[1]{%
    \errmessage{(Inkscape) Transparency is used (non-zero) for the text in Inkscape, but the package 'transparent.sty' is not loaded}%
    \renewcommand\transparent[1]{}%
  }%
  \providecommand\rotatebox[2]{#2}%
  \newcommand*\fsize{\dimexpr\f@size pt\relax}%
  \newcommand*\lineheight[1]{\fontsize{\fsize}{#1\fsize}\selectfont}%
  \ifx\svgwidth\undefined%
    \setlength{\unitlength}{483.75001134bp}%
    \ifx\svgscale\undefined%
      \relax%
    \else%
      \setlength{\unitlength}{\unitlength * \real{\svgscale}}%
    \fi%
  \else%
    \setlength{\unitlength}{\svgwidth}%
  \fi%
  \global\let\svgwidth\undefined%
  \global\let\svgscale\undefined%
  \makeatother%
  \begin{picture}(1,0.21739038)%
    \lineheight{1}%
    \setlength\tabcolsep{0pt}%
    \put(0,0){\includegraphics[width=\unitlength,page=1]{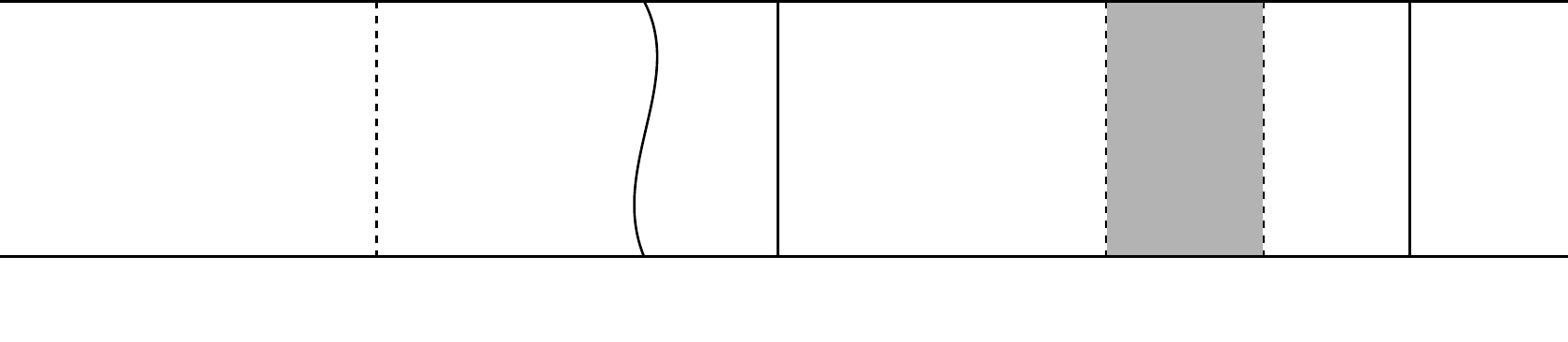}}%
    \put(0.14291888,0.01249907){\color[rgb]{0,0,0}\makebox(0,0)[lt]{\lineheight{1.25}\smash{\begin{tabular}[t]{l}$|z| = \left(\frac{r_N}{c_NM_N}\right)^{\frac{1}{M_N - 1}}$\end{tabular}}}}%
    \put(0.38355869,0.00881565){\color[rgb]{0,0,0}\makebox(0,0)[lt]{\lineheight{1.25}\smash{\begin{tabular}[t]{l}$\partial U$\end{tabular}}}}%
    \put(0.45232757,0.01167655){\color[rgb]{0,0,0}\makebox(0,0)[lt]{\lineheight{1.25}\smash{\begin{tabular}[t]{l}$|z| = \frac{1}{4}R_1$\end{tabular}}}}%
    \put(0.85172301,0.02690567){\color[rgb]{0,0,0}\makebox(0,0)[lt]{\lineheight{1.25}\smash{\begin{tabular}[t]{l}$\partial V = \partial B(0,r)$\end{tabular}}}}%
  \end{picture}%
\endgroup%
}
	\caption{A schematic for Definition \ref{poly_like_defs} and Lemma \ref{polynomial_like_polynomial}. The critical points for $q_N$ lie on the circle $|z| = (\frac{r_N}{c_NM_N})^{\frac{1}{M_N-1}}$, and the associated critical values lie in an annulus contained in $V$, illustrated in gray. So $U$ contains the critical points of $q_N$, and it will also be verified that $U \subset D$.}
	\label{Polynomial_like_picture}
\end{figure}
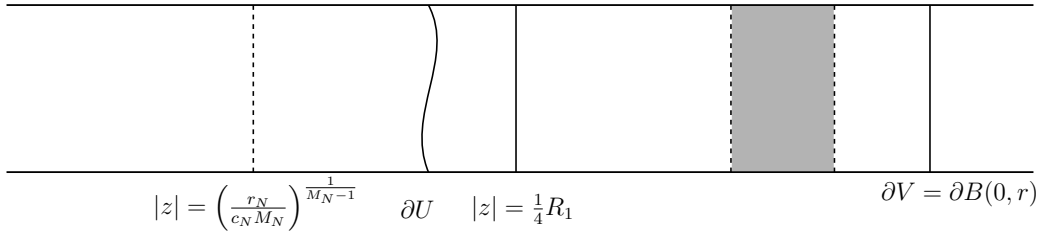

\begin{lem}
	\label{polynomial_like_polynomial}
	For all $N \geq 10$, the triple $q_N: U' \ra V$ is a degree $2^N$ polynomial-like mapping. Moreover, all $2^{N} -1$ many critical points of $q_N$ belong to $U' \subset D$. 
\end{lem}
\noindent

\begin{proof}
	We first verify that $U' \subset D$. Note that if $|z| = \frac{1}{4}R_1 = \frac{1}{4} r_N$, then by Lemma \ref{the_k_equals_1_lemma_qN} we have 
	\begin{align}
	\label{OneFourR_2}
	|q_N(z)| &\geq \frac{1}{2} c_N \left( \frac{1}{4}r_N \right)^{M_N} \\ 
	\nonumber	&= \left( \frac{1}{2}\right)^{M_N + 1} r_{N+1} \\
	\nonumber(\textrm{Lemma \ref{Rkest}})    &\geq \left( \frac{1}{4} \right)^{M_N+\frac{1}{2}} r_N^{M_{(N-1)}} r_N \\
	\nonumber	&= \frac{1}{2} \left( \frac{r^{\frac{1}{4}}_N}{4} \right)^{M_N} r_N^{M_{(N-2)}} r_N.
	\end{align}
	Therefore, since $N \geq 10$ we have $r_N^{\frac{1}{4}}/4 > 8$ and we deduce that
	\begin{equation}
	\label{OneFourR_1}
	|q_N(z)| \geq 16 r_N^3 > 16 r^2_{N-1}r_N.
	\end{equation}
	Therefore, we must have $U' \subset D$.
	
	By equation (\ref{crits_B1_eqn3}), the critical values of $q_N$ satisfy $|z| \leq r_{N-1}^2r_N$. Therefore, $V$ contains all $2^N-1$ many critical values of $q_N$, so that $U'$ contains the $2^N - 1$ many critical points of $q_N$. It now follows from Lemma \ref{preimage_helper} that $q_N: U' \ra V$ is a proper degree $2^N$ branched covering map, and it follows from Theorem \ref{RH} that $U'$ is a Jordan domain. Since $U'$ is contained in $D$ which is compactly contained in $V$,	$q_N: U' \rightarrow V$ is a degree $2^N$ polynomial-like mapping.
\end{proof}	
\begin{lem}
	\label{polynomial_like}
	Let $U$ be as in Definition \ref{poly_like_defs} and $N \geq 10$. Then $U \subset D$ and the triple $f_N:U\ra V$ is a degree $2^{N}$ polynomial-like mapping.
\end{lem}
\begin{proof}
	By Lemma \ref{AkEst} we verify that $U = \phi_N(U') \subset D$. Therefore, by Lemma \ref{branch_covering_hn}, $f_N = q_N \circ \phi_N^{-1}: U \rightarrow V$ is a proper, degree $2^N$ branched covering map, and is therefore a degree $2^N$ polynomial-like mapping. 
\end{proof}

The rest of Section \ref{Origin} is devoted to showing the filled Julia set of $f_N: U \ra V$ has dimension $\ll 1$. We will do so by constructing a cover by pulling back $B(0, 4R_1)$ under appropriate branches of $f^{-1}$.

\begin{rem}	
	\label{gamma_sigma}
	Suppose that $B(0,R)$ is the disk of radius $R$ centered at the origin, where we take any $R \in [4R_1,8R_1]$, so that $B(0,R) \subset V$. By Lemmas \ref{polynomial_like_polynomial} and \ref{polynomial_like}, $\overline{B(0,R)}$ contains the $2^N-1$ many critical points of $f_N: U \rightarrow V$. However, by (\ref{crits_B1_eqn4}), $B(0,R)$ does not contain the $2^N-1$ many critical values of $f_N: U \rightarrow V$. It follows then from Lemma \ref{component_lemma} that $f_N^{-1}(B) \subset U$ is the disjoint union of $2^N$ many Jordan domains $B_i$, $i=1,\dots,2^N$ such that $f_N: B_i \rightarrow B$ is conformal.
\end{rem}

Remark \ref{gamma_sigma} motivates the following definition (see Figure \ref{gammaillustration}).

\begin{figure}
	\centering	
	\scalebox{.5}{%% Creator: Inkscape 1.0.1 (c497b03c, 2020-09-10), www.inkscape.org
%% Accompanies image file 'gammaillustration.pdf' (pdf, eps, ps)
\begingroup%
  \makeatletter%
  \providecommand\color[2][]{%
    \errmessage{(Inkscape) Color is used for the text in Inkscape, but the package 'color.sty' is not loaded}%
    \renewcommand\color[2][]{}%
  }%
  \providecommand\transparent[1]{%
    \errmessage{(Inkscape) Transparency is used (non-zero) for the text in Inkscape, but the package 'transparent.sty' is not loaded}%
    \renewcommand\transparent[1]{}%
  }%
  \providecommand\rotatebox[2]{#2}%
  \newcommand*\fsize{\dimexpr\f@size pt\relax}%
  \newcommand*\lineheight[1]{\fontsize{\fsize}{#1\fsize}\selectfont}%
  \ifx\svgwidth\undefined%
    \setlength{\unitlength}{918.85158749bp}%
    \ifx\svgscale\undefined%
      \relax%
    \else%
      \setlength{\unitlength}{\unitlength * \real{\svgscale}}%
    \fi%
  \else%
    \setlength{\unitlength}{\svgwidth}%
  \fi%
  \global\let\svgwidth\undefined%
  \global\let\svgscale\undefined%
  \makeatother%
  \begin{picture}(1,0.60024242)%
    \lineheight{1}%
    \setlength\tabcolsep{0pt}%
    \put(0,0){\includegraphics[width=\unitlength,page=1]{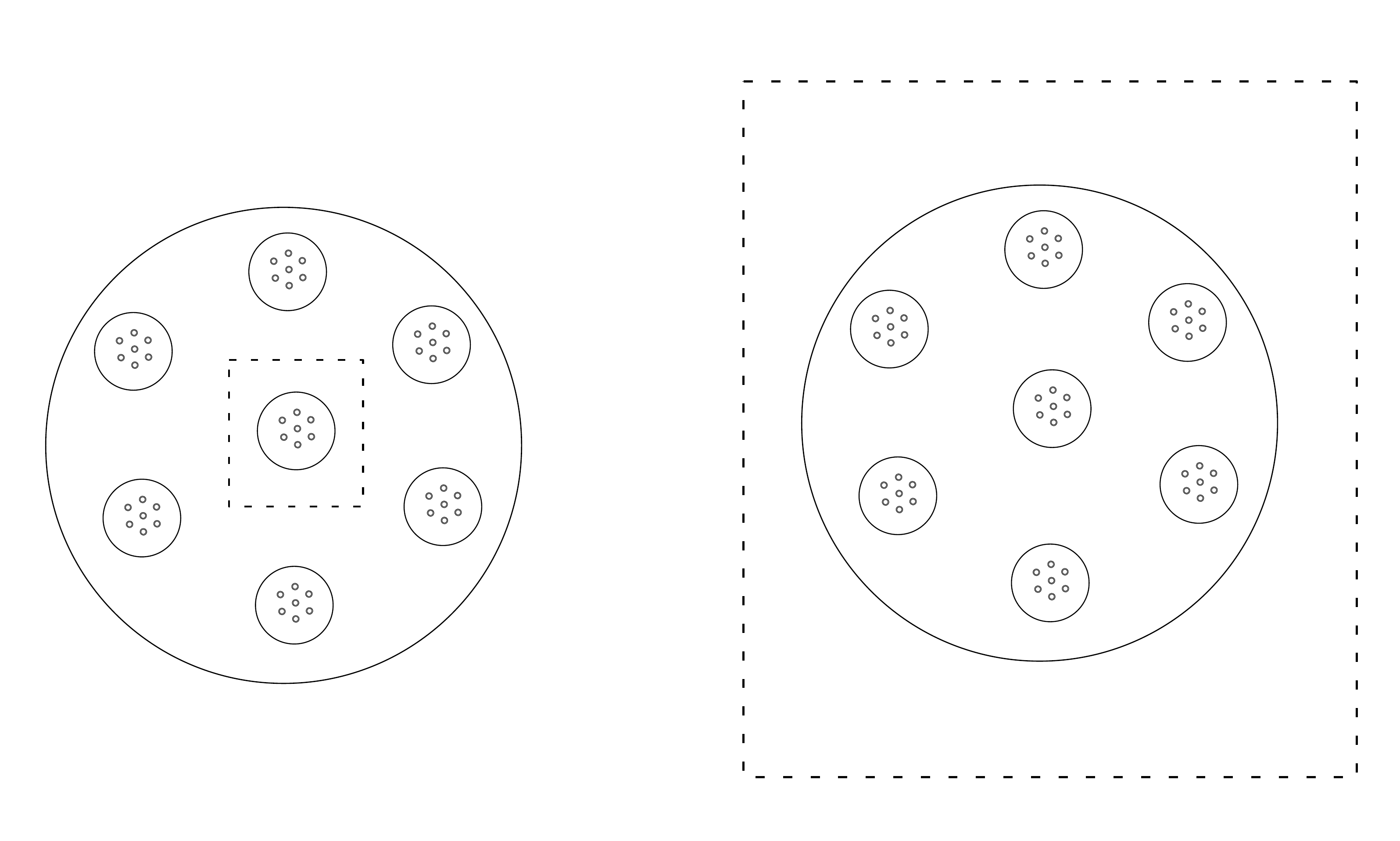}}%
    \put(0.19949581,0.47797228){\color[rgb]{0,0,0}\makebox(0,0)[lt]{\lineheight{1.25}\smash{\begin{tabular}[t]{l}\scalebox{3}{$\gamma$}\end{tabular}}}}%
    \put(0,0){\includegraphics[width=\unitlength,page=2]{gammaillustration.pdf}}%
    \put(0.04735776,0.03412872){\color[rgb]{0,0,0}\makebox(0,0)[lt]{\lineheight{1.25}\smash{\begin{tabular}[t]{l}\scalebox{3}{$\Gamma_1$}\end{tabular}}}}%
    \put(0,0){\includegraphics[width=\unitlength,page=3]{gammaillustration.pdf}}%
    \put(0.58354274,0.05925398){\color[rgb]{0,0,0}\makebox(0,0)[lt]{\lineheight{1.25}\smash{\begin{tabular}[t]{l}\scalebox{3}{$\Gamma_2$}\end{tabular}}}}%
    \put(0,0){\includegraphics[width=\unitlength,page=4]{gammaillustration.pdf}}%
  \end{picture}%
\endgroup%
}
	\caption{Illustrated is Definition \ref{GreensLines} of the families $\Gamma_n$.}
	\label{gammaillustration}
\end{figure}

\begin{definition}
	\label{GreensLines}
	Define $\gamma := \{z\,:\, |z| = 4R_1\}.$ Then $\gamma$ is a circle that surrounds the critical points of $f_N$ contained in $D$, but not the critical values associated to those critical points 
	\begin{enumerate}
		\item Let $\Gamma_1 = f^{-1}_N(\gamma)$ be the disjoint union of the $2^N$ many Jordan curves contained in $D$ which $f_N$ maps to $\gamma$. We denote the elements of $\Gamma_1$ by $\gamma_1$.
		\item Let $\Gamma_n = f^{-n}_N(\gamma)$ be the disjoint union of $2^{Nn}$ many Jordan curves in $D$ which get mapped by $f^n_N$ to $\gamma$. We denote the elements of $\Gamma_N$ by $\gamma_n$. 
		\item Given $\gamma_n \in \Gamma_n$, we define $\widehat{\gamma_n}$ to be the bounded simply connected domain with boundary given by $\gamma_n$. We define $\widehat{\Gamma_n}$ to be the set of all $\widehat{\gamma_n}.$
	\end{enumerate} 
\end{definition}

\begin{rem}
	\label{Modulus_lower_bound}
	An alternative definition for $\gamma$ in Definition \ref{GreensLines} would be $\sigma = \{z\,:\, |z| = 8R_1\}$, and we analogously could define $\Sigma_n$, elements $\sigma_n \subset \Sigma_n$, and $\widehat{\sigma_n}$. Then for each $\sigma_n \in \Sigma_n$, $\widehat{\sigma_n}$ contains exactly one element $\gamma_n \in \Gamma_n$, and the modulus of $\widehat{\sigma_n} \setminus \overline{\widehat{\gamma_n}}$ is bounded below by $(2\pi)^{-1} \log 2 > 0$. For each $\sigma_n \in \Sigma_n$, there exists some element $\sigma_{n-1} \in \Sigma_{n-1}$ such that $f_N: \widehat{\sigma_n} \rightarrow \widehat{\sigma_{n-1}}$ is conformal. This means that the corresponding mapping $f_N: \widehat{\gamma_n} \rightarrow \widehat{\gamma_{n-1}}$ is conformal, and by Remark \ref{Bounded_conformal_distortion}, Corollary \ref{Kobe} and Corollary \ref{shape} apply with constants that do not depend on the integers $N$ or $n$.
\end{rem}	

\noindent We now estimate the diameters of our covering of Definition \ref{GreensLines}:	

\begin{lem}
	\label{level_lines_p_N}
	Let $\gamma$, $\Gamma_n$, and $\widehat{\Gamma_n}$ be as in Definition \ref{GreensLines}. Then there exists a value $M \in \N$ so that for all $N \geq M$, for all $n \geq 1$, and for every $\gamma_n \in \Gamma_n$ we have
	$$\diam(f_N(\gamma_n)) \geq R_1 \diam(\gamma_n).$$
\end{lem}
\begin{proof}
	Choose some $\widehat{\gamma_1} \in \widehat{\Gamma_1}$, and let $z_0\in \widehat{\gamma_1}$ be a zero for $f_N$. Such a $z_0$ exists since $f_N(\gamma_1) = \gamma$ surrounds the origin. Then by Corollary \ref{shape} and Remark \ref{Modulus_lower_bound} applied to the appropriate branch of the inverse $f^{-1}_N: B(0,4R_1) \rightarrow \widehat{\gamma_1}$, there exists a constant $C>0$ such that
	$$ \gamma_1 \subset B\left(z_0, \frac{C(4R_1)}{|f_N'(z_0)|}\right).$$
	Therefore, by Theorem \ref{identity_origin} and Lemma \ref{qN}, there exists another constant $L>0$ so that
	$$\frac{\diam(f_N(\gamma_1))}{\diam(\gamma_1)} \geq \frac{1}{C} \cdot  |f'_N(z_0)| = \frac{1}{C} \cdot R_1(M_N -1) \cdot|(\phi_N^{-1})'(z_0)| \geq \frac{R_1}{C}\cdot L \cdot  (M_N-1).$$
	\noindent
	This proves the estimate for the case of $n =1$. For $n > 1$, by Corollary \ref{Kobe} and Remark \ref{Modulus_lower_bound}, there exists a constant $c >1$ such that
	$$\frac{\diam(f_N(\gamma_n))}{\diam(\gamma_n)} \geq \frac{1}{c} \frac{\diam(f_N^{n-1}(f_N(\gamma_n)))}{\diam(f_N^{n-1}(\gamma_n))} \geq \frac{1}{c} \frac{\diam(\gamma)}{\diam(\gamma_1)} \geq \frac{R_1}{c \cdot C} \cdot L \cdot (M_N-1).$$
	Therefore, there exists $M \in \N$ so that for all $N \geq M$ and for all $n \geq 1$, we have
	$$\frac{\diam(f_N(\gamma_n))}{\diam(\gamma_n)} \geq R_1.$$
	This is exactly what we wanted to show.
\end{proof}

\noindent	Lemma \ref{level_lines_p_N} allows us to deduce the Hausdorff dimension of the Julia set of the polynomial like mapping $f_N: U \ra V$.
\begin{thm}
	\label{dimension_origin_Cantor}
	Let $t > 0$ be given. Then there exists an integer $M$ so that for all $N \geq M$, the Hausdorff dimension of the filled Julia set of $f_N:U\ra V$ is at most $t$.
\end{thm}
\begin{proof}
	For each $n \geq 1$, $\widehat{\Gamma_n}$ is a covering of the Julia set of $(f_N,U,V)$. Fix $t > 0$. If $\widehat{\gamma_n} \in \widehat{\Gamma_n}$, then $f_N(\widehat{\gamma_n}) =: \widehat{\gamma_{n-1}} \in \widehat{\Gamma_{n-1}}$. Therefore, by Lemma \ref{level_lines_p_N}, we have
	\begin{equation}
	\label{filled_Julia_sum_1}
	\sum_{\widehat{\gamma_n} \in \widehat{\Gamma_n}} \diam(\widehat{\gamma_n})^t \leq R_1^{-t} \cdot 2^N \cdot \sum_{\widehat{\gamma_{n-1}} \in \widehat{\Gamma_{n-1}}} \diam(\widehat{\gamma_{n-1}})^t.
	\end{equation}
	It follows from (\ref{filled_Julia_sum_1}) that 
	\begin{equation}
	\label{filled_Julia_set_2}
	\sum_{n=1}^{\infty} \sum_{\widehat{\gamma_n} \in \widehat{\Gamma_n}} \diam(\widehat{\gamma_n})^t \leq  \diam(\gamma)^t \cdot \sum_{n=1}^{\infty} 2^{Nn} R_1^{-tn}
	\end{equation}
	The sum in (\ref{filled_Julia_set_2}) converges if and only if 
	$$2^{N} R_1^{-t} < 1.$$
	To see this, we use Corollary \ref{rkinequalities} to see that
	\begin{equation}
	\label{decay}
	R_1^{-t} 2^N 
	=  r_N^{-t} 2^{N} 
	\leq 2^{N-M_{N-1}t}
	\end{equation}
	Choose $M$ so that for all  $N \geq M$, we have $N - 2^{N-1}t < 0$, so that we obtain $2^N R_1^{-t} < 1$. For such a choice of $N$, (\ref{filled_Julia_set_2}) converges, and for any $\eta > 0$ there exists some value $n$ so that 
	$$\sum_{\widehat{\gamma_n} \in \widehat{\Gamma_n}} \diam(\widehat{\gamma_n})^t < \eta$$
	Since $\widehat{\Gamma_n}$ is a covering of the filled Julia set of $f_N: U \ra V$, it follows that its Hausdorff dimension is bounded above by $t$. 
\end{proof}

\noindent We conclude Section \ref{Origin} by showing that the critical values of $f_N$ lying in $|z|<R_1/4$ map to $B_1$. This will be crucial in constructing covers of the Julia set of $f_N$ by pulling back the annuli $A_k$ under $f_N$.

\begin{lem}
	\label{crit_values_avoid_A1_fN}
	For all $N \geq 10$, if $z$ is a critical point of $f_N$ contained in $D$, then $f_N(z) \in B_1$. 
\end{lem}
\begin{proof}
	The only critical points of $f_N$ contained inside of $D$ are of the form $\phi_N(z)$, where $z$ is a critical point of $q_N$. By Lemma \ref{crits_B1}, the critical values of $q_N$ are contained in $B_1$. Therefore, the critical values of $f_N$ associated to the critical points in $D$ belong to $B_1$.
\end{proof}	

\section{Location of the Julia Set}
\label{Expanding Zeros}

In this Section, we refine our understanding of the behavior of $f$ in the annuli $A_k$. Namely, we prove that unless $z$ belongs to $V_k$ or a collection of small balls (which we call \emph{petals} in Definition \ref{petal_def} below), then $f(z)\in B_{k+1}$. This is crucial to our understanding of the structure of $\mathcal{J}(f)$, since it is readily observed (see Lemma \ref{BkFatou} below) that the annuli $B_k$ lie in $\mathcal{F}(f)$. We will also further describe the behavior of $f$ in the aforementioned petals. 

% It may be quickly observed that the annuli $B_k$ all belong to $\mathcal{F}(f)$, 

%Namely, we prove that most of $A_k$ maps to 

% Namely, we prove that 

% The main result is that if $z\in\mathcal{J}(f)\cap A_k$, then either $z$ belongs to the subannulus $V_k \subset A_k$, or $z$ belongs to one of a collection of small balls (which we call \emph{petals} in Definition \ref{petal_def} below).

\begin{lem}
	\label{BkFatou}
	There exists $M \in \N$ such that for all $N \geq M$, $B_k$ belongs to the Fatou set of $f_N$ for all $k \geq 1$.
\end{lem}
\begin{proof}
	By Lemma \ref{Bk}, there exists $M \in \N$ such that for all $N \geq M$ we have $f_N(B_k) \subset B_{k+1}$. This implies that each point in $B_k$ escapes locally uniformly to $\infty$.
\end{proof}

Therefore, when $|z| \geq \frac{1}{4} R_1$, the Julia set of $f_N$ is contained in $\cup_{k=1}^{\infty} A_k$. For each $k \geq 1$, Proposition \ref{zeros} says that $f_N$ has $n_k$ many zeros contained in $A_k$. For a given $k \geq 1$, let $\{w_j^k\}_{j=1}^{n_k}$ denote these zeros. Following the terminology in \cite{Bis18}, we introduce some notation for the balls containing the zeros of $f_N$ inside of $A_k$ (see Figure \ref{petalillustration.pdf_tex}).
\begin{definition}
	\label{petal_def}
	For $k \geq 1$, let  $\mathcal{P}_k =  \cup_{j=1}^{n_k} B(w^k_j,R_k/2^{n_k})$ be the \textit{petals} of $f_N$ inside of $A_k$.  A connected component $P_k \subset \mathcal{P}_k$ will be called a \textit{petal}. 
\end{definition}

\begin{figure}
	\centering	
	\scalebox{.4}{%% Creator: Inkscape 1.0.1 (c497b03c, 2020-09-10), www.inkscape.org
%% Accompanies image file 'petalillustration.pdf' (pdf, eps, ps)
\begingroup%
  \makeatletter%
  \providecommand\color[2][]{%
    \errmessage{(Inkscape) Color is used for the text in Inkscape, but the package 'color.sty' is not loaded}%
    \renewcommand\color[2][]{}%
  }%
  \providecommand\transparent[1]{%
    \errmessage{(Inkscape) Transparency is used (non-zero) for the text in Inkscape, but the package 'transparent.sty' is not loaded}%
    \renewcommand\transparent[1]{}%
  }%
  \providecommand\rotatebox[2]{#2}%
  \newcommand*\fsize{\dimexpr\f@size pt\relax}%
  \newcommand*\lineheight[1]{\fontsize{\fsize}{#1\fsize}\selectfont}%
  \ifx\svgwidth\undefined%
    \setlength{\unitlength}{883.48123385bp}%
    \ifx\svgscale\undefined%
      \relax%
    \else%
      \setlength{\unitlength}{\unitlength * \real{\svgscale}}%
    \fi%
  \else%
    \setlength{\unitlength}{\svgwidth}%
  \fi%
  \global\let\svgwidth\undefined%
  \global\let\svgscale\undefined%
  \makeatother%
  \begin{picture}(1,0.94149117)%
    \lineheight{1}%
    \setlength\tabcolsep{0pt}%
    \put(0,0){\includegraphics[width=\unitlength,page=1]{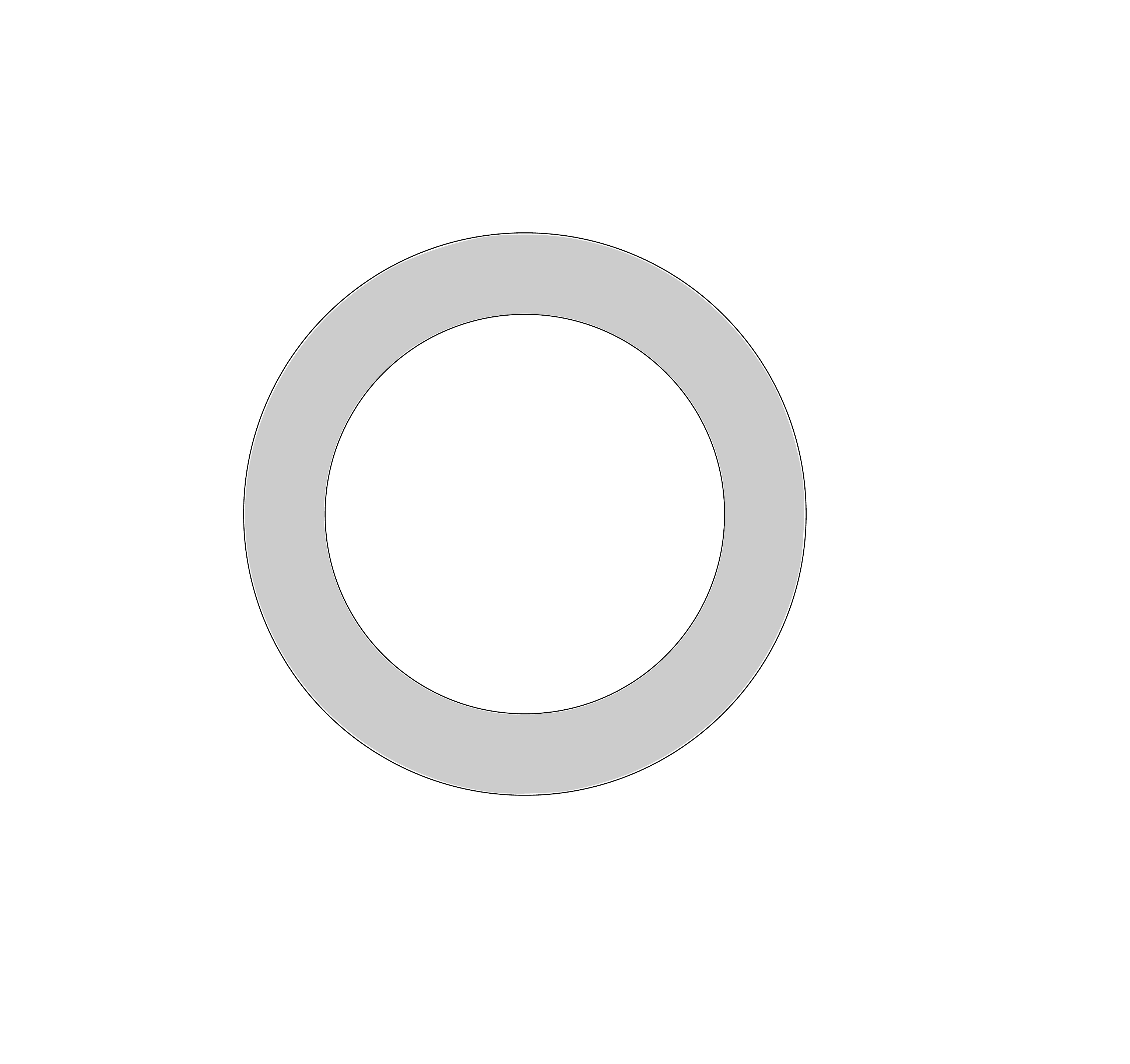}}%
    \put(0.91427829,0.66790028){\color[rgb]{0,0,0}\makebox(0,0)[lt]{\lineheight{1.25}\smash{\begin{tabular}[t]{l}\scalebox{3}{$A_k$}\end{tabular}}}}%
    \put(0,0){\includegraphics[width=\unitlength,page=2]{petalillustration.pdf}}%
    \put(0.78438528,0.82850723){\color[rgb]{0,0,0}\makebox(0,0)[lt]{\lineheight{1.25}\smash{\begin{tabular}[t]{l}\scalebox{3}{$V_k$}\end{tabular}}}}%
    \put(0,0){\includegraphics[width=\unitlength,page=3]{petalillustration.pdf}}%
    \put(0.91587025,0.50775879){\color[rgb]{0,0,0}\makebox(0,0)[lt]{\lineheight{1.25}\smash{\begin{tabular}[t]{l}\scalebox{3}{$P_j$}\end{tabular}}}}%
    \put(0.44221187,0.35635859){\color[rgb]{0,0,0}\makebox(0,0)[lt]{\lineheight{1.25}\smash{\begin{tabular}[t]{l}\scalebox{3}{$B_{k-1}$}\end{tabular}}}}%
    \put(0,0){\includegraphics[width=\unitlength,page=4]{petalillustration.pdf}}%
    \put(0.45894194,0.15178468){\color[rgb]{0,0,0}\makebox(0,0)[lt]{\lineheight{1.25}\smash{\begin{tabular}[t]{l}\scalebox{3}{$B_{k}$}\end{tabular}}}}%
    \put(0,0){\includegraphics[width=\unitlength,page=5]{petalillustration.pdf}}%
  \end{picture}%
\endgroup%
}
	\caption{Illustrated is Definition \ref{petal_def} of the petals $P_k$. The annuli $B_{k-1}$, $B_k$ are in white, the annulus $A_k$ is in light grey, and the annulus $V_k$ and petals $P_k$ are in dark grey.}
	\label{petalillustration.pdf_tex}
\end{figure}

As already mentioned, we will now prove several lemmas (Lemmas \ref{annulus_1}-\ref{hole_punched}) detailing the mapping behavior of $f$ within the annulus $A_k$, most crucially within the subannulus $V_k$ and the petals $P_k$. 

\begin{lem}
	\label{annulus_1}
	There exists $M \in \N$ so that for all $N \geq M$, if $k \geq 1$, then 
	$$f_N\left(A\left(\frac{5}{4} R_k, 4 R_k\right)\right)  \subset B_{k+1}.$$
\end{lem}
\noindent
\begin{proof}
	By Proposition \ref{zeros} and by Lemma \ref{holomorphic_annulus}, it is sufficient to verify that there exists $M \in \N$ so that for all $N \geq M$ we have $f_N(|z| = 4R_k) \subset B_{k+1}$ and $f_N(|z| = \frac{5}{4}R_k) \subset B_{k+1}$. The former has already been verified in (\ref{Inner_Bk}), so we only need to prove the latter. By the maximum principle for holomorphic functions, we have 
	\begin{equation}
	\label{54_annulus_max}
	\max_{|z| = \frac{5}{4}R_k} |f_N(z)| \leq \max_{|z| = 4R_k} |f_N(z)| < \frac{1}{8}R_{k+2}.
	\end{equation}
	By Lemma \ref{PowerMap}, we have 
	\begin{align*}
	\min_{|z| = \frac{5}{4}R_k} |f_N(z)| \geq \min_{|z| = \frac{5}{4}R_k} C_{k+1} \beta_k^{n_{k+1}} |z|^{n_{k+1}} 
	= \beta_k^{n_{k+1}} C_k \left( \frac{R_k}{2} \right)^{n_k} \left(\frac{5}{2}\right)^{n_{k+1}} \left(\frac{1}{2}\right)^{n_k} 
	\geq \beta_k^{n_{k+1}} R_{k+1} \left(\frac{5}{4}\right)^{n_{k+1}}
	\end{align*}
	By (\ref{alpha_beta_close_1_eqn}) we have $\beta_k \cdot \frac{5}{4}\geq \frac{6}{5}$. Therefore, since $N \geq 5$ we have
	\begin{equation}
	\label{54_annnulus_min}
	\min_{|z| = \frac{5}{4}R_k} |f_N(z)| \geq \left(\frac{6}{5}\right)^{n_{k+1}} R_{k+1} > 8R_{k+1}.
	\end{equation}
	It follows from (\ref{54_annnulus_min}) and (\ref{54_annulus_max}) that
	\begin{equation}
	\label{54_annulus}
	f_N\left(|z| = \frac{5}{4}R_k\right) \subset A(8R_{k+1}, \frac{1}{8}R_{k+2}) \subset B_{k+1}.
	\end{equation}
	As discussed at the beginning, this proves the claim.
\end{proof}

\begin{lem}
	\label{annulus_2}
	There exists $M \in \N$ so that for all $N \geq M$, if $k \geq 1$, then
	\begin{equation*}
	f_N\left(A\left(\frac{1}{4}R_k, \frac{2}{5}R_k\right)\right) \subset B_{k}.
	\end{equation*}
\end{lem}
\begin{proof}
	We consider the cases of $k \geq 2$ and $k = 1$ separately.
	
	When $k \geq 2$, (\ref{Outer_Bk}) implies that $f_N(|z| = \frac{1}{4}R_k) \subset B_{k}$. By (\ref{Inner_Vk}), we have $\max_{|z| = \frac{2}{5}R_k} |f_N(z)| \leq \frac{1}{4}R_{k+1}$. Next, we observe that by Lemma \ref{Rkest},
	\begin{align*}
	\min_{|z| = \frac{2}{5}R_k} |f_N(z)| \geq \min_{|z| = \frac{2}{5}R_k} \beta_k^{n_k} C_k |z|^{n_{k}} 	=  \beta_k^{n_k} C_k \left(\frac{4}{5}\right)^{n_k} \left(\frac{1}{2}R_k\right)^{n_{k}} 
	&= \beta_k^{n_k} R_{k+1} \left(\frac{4}{5}\right)^{n_k} \\ 
	&\geq \beta_k^{n_k} \left(\frac{4}{10}\right)^{n_k} R_k^{n_{k-1}} R_k. 
	\end{align*} 
	Since $N \geq 5$, by (\ref{alpha_beta_close_1_eqn}) and Lemma \ref{Rkest} we have 
	\begin{equation}
	\beta_k \frac{4}{10} R^{\frac{1}{2}}_k \geq 4.
	\end{equation}
	Therefore we obtain 
	\begin{equation}
	\min_{|z| = \frac{2}{5}R_k} |f_N(z)| \geq 4^{n_k} R_k > 8 R_k.
	\end{equation}
	It follows that 
	\begin{equation}
	\label{Inner_Vk_Bk}
	f_N(|z| = \frac{2}{5}R_k) \subset A(8R_{k}, \frac{1}{8}R_{k+1})  \subset  B_k
	\end{equation}
	whenever $k \geq 2$. Therefore, the case of $k \geq 2$ follows from part (2) of Lemma \ref{holomorphic_annulus} and Proposition \ref{zeros}.
	
	For the $k = 1$ case, we use slightly different estimates. First, notice that by following a similar argument above, but applying Lemma \ref{the_k_equal_1_lemma}, we obtain
	\begin{equation*}
	\min_{|z| = \frac{2}{5}R_1} |f_N(z)| \geq \frac{1}{2} 4^{n_1} R_1 > 8R_1.
	\end{equation*}
	This, combined with (\ref{Inner_V1}) allows us to conclude that 
	\begin{equation}
	\label{Inner_V1_B1}
	f_N(|z|= \frac{2}{5}R_1)  \subset B_1.
	\end{equation}
	Next, by following similar reasoning as in (\ref{OneFourR_2}), except this time applying Lemma \ref{the_k_equal_1_lemma}, we have 
	\begin{equation*}
	\min_{|z| = \frac{1}{4}R_1} |f_N(z)| \geq \frac{1}{4} \beta_1^{M_N} \left( \frac{r_N^{\frac{1}{2}}}{4}\right)^{M_N} r_N^{M_{N-2}} r_N = \frac{1}{4}\left( \frac{\beta_1 r_N^{\frac{1}{2}}}{4}\right)^{M_N} r_N^{M_{N-2}} r_N
	\end{equation*}
	By (\ref{alpha_beta_close_1_eqn}) and Lemma \ref{Rkest}, along with similar reasoning as (\ref{OneFourR_1}), we have
	\begin{equation}
	\label{annulus_2_eqn}
	\min_{|z| = \frac{1}{4}R_1} |f_N(z)| > 16 r^2_{N-1} r_N = 16 r^2_{N-1} R_1^2 > 4 R_1.
	\end{equation}
	So by (\ref{annulus_2_eqn}) and the Maximum principle used with (\ref{Inner_V1}), we have 
	\begin{equation}
	\label{Inner_B0}
	f_N(|z| = \frac{1}{4}R_1) \subset B_1.
	\end{equation}
	The $k=1$ case now follows from Proposition \ref{zeros} and part (2) of Lemma \ref{holomorphic_annulus}.
\end{proof}

\begin{lem}
	\label{annulus_3}
	There exists $M \in \N$ so that for all $N \geq M$, if $k \geq 1$, we have
	\begin{equation}
	\label{annulus_3_3/2}
	f_N(|z| = \frac{3}{5}R_k) \subset B_{k+1}.
	\end{equation}
\end{lem}
\begin{proof}
	We again must argue the $k=1$ and $k \geq 2$ cases separately. 
	
	When $k \geq 2$, we have by Lemma \ref{PowerMap} that 
	\begin{align*}
	\max_{|z| = \frac{3}{5}R_k} |f_N(z)| &\leq \alpha_k^{n_k} C_k \left(\frac{3}{5}R_k \right)^{n_k} 
	= \alpha_k^{n_k} C_k \left( \frac{6}{5} \right)^{n_k} \left( \frac{1}{2} R_k \right)^{n_k} 
	= \alpha_k^{n_k} \left( \frac{6}{5} \right)^{n_k} R_{k+1}.
	\end{align*}
	By (\ref{alpha_beta_close_1_eqn}) we have $\alpha_k  \frac{6}{5} \leq \frac{7}{5}$. Therefore, since $N \geq 5$, we obtain
	\begin{equation*}
	\frac{(\alpha_k \frac{6}{5})^{n_k}R_{k+1}}{R_{k+2}} \leq \frac{( \frac{7}{5})^{n_k}R_{k+1}}{2^{-n_{k+1}} R_{k+1}^{n_k + 1}} = \left( \frac{28}{5 R_{k+1}} \right)^{n_k} < \left(\frac{1}{4}\right)^{n_k}.
	\end{equation*}
	It follows that
	\begin{equation}
	\label{Outer_Vk_max}
	\max_{|z| = \frac{3}{5}R_k} |f_N(z)| \leq \left(\frac{1}{4}\right)^{n_k} R_{k+2} < \frac{1}{8}R_{k+2}.
	\end{equation}
	Therefore, we have by (\ref{Outer_Vk}) and (\ref{Outer_Vk_max}) that 
	\begin{equation}
	\label{annulus_3_eqn1}
	f_N(|z| = \frac{3}{5}R_k)\subset A(8R_{k+1}, \frac{1}{8}R_{k+2})  \subset B_{k+1}.
	\end{equation}
	
	For the $k = 1$, case, the arguments are similar, and by using Lemma \ref{the_k_equal_1_lemma} and requiring $N \geq 5$ we obtain
	\begin{equation}
	\label{Outer_V1_max}
	\max_{|z| = \frac{3}{5}R_1} \leq 2 \left( \frac{1}{4} \right)^{n_k} R_{3} < \frac{1}{8} R_3. 
	\end{equation}
	Therefore, by (\ref{Outer_V1}) and (\ref{Outer_V1_max}) we obtain
	\begin{equation}
	\label{annulus_3_eqn2}
	f_N(|z| = \frac{3}{5}R_1) \subset A(8R_2,\frac{1}{8}R_3) \subset  B_{3}.
	\end{equation}
	The Lemma now follows from (\ref{annulus_3_eqn1}) and (\ref{annulus_3_eqn2})
\end{proof}

\begin{cor}
	\label{helpful_Vk_cor}
	For all $k \geq 1$, we have $f_N(V_k) \subset B_k \cup A_{k+1} \cup B_{k+1}$
\end{cor}
\begin{proof}
	By (\ref{Inner_Vk_Bk}) and (\ref{Inner_V1_B1}), along with (\ref{annulus_3_eqn1}) and (\ref{annulus_3_eqn2}), we have $f_N(|z|=\frac{2}{5}R_k) \subset B_k$ and $f_N(|z| = \frac{3}{5}R_k) \subset B_{k+1}$ for all $k \geq 1$. Therefore, by part (2) of Lemma \ref{holomorphic_annulus}, we have $f_N(V_k) \subset B_k \cup A_{k+1} \cup B_{k+1}$, as desired.
\end{proof}

The following lemma asserts that $f_N$ is conformal on every petal $P_k \subset \mathcal{P}_k$, with large expansion. 

\begin{lem}
	\label{petal_radius}
	There exists an $M \in \N$ and a constant $\lambda > 0$ so that for all $N \geq M$, for all $k \geq 1$, and for all zeros $w_j^k$ in $A_k$, $j = 1,\dots,n_k$, the mapping
	\begin{equation}
	\label{conformal}
	f_N: B(w^k_j, \lambda ( \exp(\pi/n_k) - 1 ) R_k) \rightarrow f_N(B(w^k_j, \lambda ( \exp(\pi/n_k) - 1 ) R_k)),
	\end{equation}
	is conformal. Moreover, we have 
	\begin{equation}
	B(w_j^k, \frac{R_k}{2^{n_k}}) \subset B(w^k_j, \lambda ( \exp(\pi/n_k) - 1 ) R_k),
	\end{equation}
	and
	\begin{equation}
	\label{conformal_2}
	B(0,4 R_{k+1}) \subset f_N(B(w^k_j, \frac{R_k}{2^{n_k}})).
	\end{equation}
\end{lem}

\begin{proof}
	Note that by Lemma \ref{small_ball_rescaled} and Lemma \ref{AkEst}, there exists constants $\lambda >0$ and $\delta >0$ so that 
	\begin{equation}
	\label{conformal_3}
	B(0, \delta C_k R_k^{n_k}) \subset f_N(B(w^k_j, \lambda ( \exp(\pi/n_k) - 1 ) R_k)) \subset B(0, \frac{1}{2} C_kR_k^{n_k}).
	\end{equation}
	Moreover, Lemma \ref{small_ball_rescaled} and Lemma \ref{AkEst} imply that the mapping in equation (\ref{conformal}) is injective, and therefore conformal. Next note that
	\begin{equation}
	\label{petal_in_conformal_ball}
	\frac{2^{-n_k} R_k}{\lambda(\exp(\pi/n_k)-1)R_k} \leq \frac{2^{-n_k}}{\lambda\pi / n_k} = \frac{n_k}{2^{n_k} \pi \lambda} \xrightarrow[]{k \rightarrow \infty} 0.
	\end{equation}
	Therefore, there exists $M$ so that for all $N \geq M$, we have $(\ref{conformal})$ and $B(w_j^k, R_k/2^{n_k}) \subset B(w^k_j, \lambda ( \exp(\pi/n_k) - 1 ) R_k)$. It remains to verify that (\ref{conformal_2}) holds.
	
	To see this, note that by Theorem \ref{Kobe_Quarter} and (\ref{conformal_3}), we have 
	\begin{align*}
	|(f_N)'(w^k_j)| &\geq  \frac{1}{4} \frac{\dist(0, \partial B(0,\delta C_k R_k^{n_k}))}{\dist(w^k_j,\partial B(w^k_j, \lambda ( \exp(\pi/n_k) - 1 ) R_k) ))} \\
	&= \frac{\delta }{4 \lambda} \frac{C_kR_k^{n_k}}{(\exp(\pi/n_k) - 1 )R_k} \\
	&\geq \frac{\delta}{8 \lambda \pi} n_k C_k R_k^{n_k - 1} \\
	&= \frac{\delta}{8 \lambda \pi} 2^{n_k} n_k R_{k+1} R_k^{-1}.
	\end{align*}
	Therefore, if $w^k_j$ is a zero of $f_N$ in $A_k$, we have
	\begin{equation}
	\label{expansion}
	|(f_N)'(w^k_j)| \geq  \frac{\delta}{8 \lambda \pi} 2^{n_k} n_k R_{k+1} R_k^{-1}.
	\end{equation}
	
	Next, consider the branch of the inverse $f_N^{-1}: B(0,\delta C_k R_k^{n_k}) \ra D'$, where 
	$$D' \subset B(w^k_j, \lambda ( \exp(\pi/n_k) - 1 ) R_k).$$
	Since $\delta >0$ is fixed, there exists a perhaps larger $M \in \N$ so that for all $N \geq M$ we have $\delta 2^{n_1} > 4$. Therefore,
	\begin{equation}
	\label{conformal_4}
	\delta C_k R_k^{n_k} = \delta 2^{n_k} C_k(\frac{1}{2}R_k)^{n_k} = \delta 2^{n_k} R_{k+1} > 2^{n_{k-1}} 4R_{k+1} > 4R_{k+1}.
	\end{equation}
	We can further deduce from (\ref{conformal_4}) that the modulus of the annulus $B(0,\delta C_k R_k^{n_k}) \setminus \overline{B(0,4R_{k+1})}$ is bounded below by some fixed constant independent of $k$. Denote $D = f_N^{-1}(B(0,4R_{k+1})) \subset D'$. By applying Corollary \ref{shape} to $f_N^{-1}: B(0,4R_{k+1}) \rightarrow D$, we see that there exists a constant $L'>0$ such that
	\begin{equation}
	\label{conformal_eq_1}
	\frac{4R_{k+1}}{L'} \frac{1}{|f'(w^k_j)|} \leq r_{D,w^k_j} \leq R_{D,w^k_j} \leq L' \frac{4R_{k+1}}{|f'(w^k_j)|}.
	\end{equation}
	Since the modulus of $B(0,\delta C_k R_k^{n_k}) \setminus \overline{B(0,4R_{k+1})}$ is bounded below by some fixed constant independent of $k$, the constant $L'$ is independent of $k$ and $w^k_j$.  By perhaps increasing $M$, we have that for all $N\geq M$ that 
	\begin{equation}
	\label{conformal_eq_2}
	4 R_{k+1}L' |f'(w_j^k)|^{-1} \leq 4 R_{k+1}L'\frac{8 \lambda \pi}{\delta n_k} 2^{-n_k} R^{-1}_{k+1} R_k \leq \frac{R_k}{2^{n_k}}
	\end{equation}
	Therefore, by (\ref{conformal_eq_1}) and (\ref{conformal_eq_2}), and (\ref{outer_radius}), we have
	$$D \subset B(w_j^k, L'\frac{4R_{k+1}}{|f'(w_j^k)|}) \subset B(w_j^k,\frac{R_k}{2^{n_k}}).$$
	This proves (\ref{conformal_2}), which is exactly what we wanted to show.
\end{proof}

\begin{lem}
	\label{hole_punched}
	For all $k \geq 1$, $f_N(A(\frac{3}{5}R_k, \frac{5}{4}R_k) \setminus \overline{\cup_{j=1}^{n_k}B(w^k_j, R_k/2^{n_k})}) \subset B_{k+1}$.
\end{lem}
\begin{proof}
	First, observe that every connected component of the boundary of $A(\frac{3}{5}R_k, \frac{5}{4}R_k) \setminus \overline{\cup_{j=1}^{n_k} B(w^k_j, R_k/2^{n_k})}$ is mapped inside of $B_{k+1}$ by $f_N$. Indeed, by Lemma \ref{annulus_3}, $f_N(|z| = \frac{3}{5}R_k) \subset B_{k+1}$, and by (\ref{54_annulus}), $f_N(|z| = \frac{5}{4} R_k) \subset B_{k+1}$. The rest of the connected components of the boundary of $A(\frac{3}{5}R_k, \frac{5}{4}R_k) \setminus \overline{\cup_{j=1}^{n_k} (w^k_j, R_k/2^{n_k})}$ are the boundaries of the petals $P_k \subset \mathcal{P}_k$. By (\ref{conformal_2}), 
	\begin{equation}
	\label{hole_punched_1}
	f_N(\partial B(w_j^k, \frac{R_k}{2^{n_k}})) \subset \{|z| \geq 4 R_{k+1}\}.
	\end{equation}
	By (\ref{crit_value_upper_bound}) and (\ref{conformal_3}), we have 
	\begin{equation}
	\label{hole_punched_2}
	f_N(\partial B(w_j^k, \frac{R_k}{2^{n_k}})) \subset B(0, \frac{1}{2}C_kR_k^{n_k}) \subset B(0, \frac{1}{4}R_{k+2}).
	\end{equation}
	Equations (\ref{hole_punched_1}) and (\ref{hole_punched_2}) imply
	\begin{equation}
	\label{petal_boundary}
	f_N(\partial B(w_j^k, \frac{R_k}{2^{n_k}})) \subset B_{k+1}.
	\end{equation}
	By Proposition \ref{zeros}, $f_N$ has no additional zeros in $A(\frac{3}{5}R_k, \frac{5}{4}R_k) \setminus \overline{\cup_{j=1}^{n_k} B(w^k_j, R_k/2^{n_k})}$. The result now follows from the maximum principle and minimum principle for non-zero holomorphic functions.
\end{proof}

With Lemmas \ref{annulus_1}-\ref{hole_punched} in hand, we can now deduce the following about the Julia set of $f_N$:

\begin{thm}
	\label{Julia_place}
	There exists $M \in \N$ so that if $N \geq M$, then for all $k \geq 1$, 
	$$(\Julia(f_N) \cap A_k) \subset \left( \cup_{j=1}^{n_k} B\left(w^k_j,\frac{R_k}{2^{n_k}}\right)\right) \cup V_k.$$
\end{thm}

\begin{proof}
	We will show that all other points get mapped to $B_k$ or $B_{k+1}$, so that they belong to the Fatou set of $f_N$ by Lemma \ref{BkFatou}.	Suppose that $z \in A_k \cap \Julia(f_N)$, but $z \notin \mathcal{P}_k$. Then by Lemma \ref{annulus_1}, $z \notin A(\frac{5}{4}R_k, 4 R_k)$, and by Lemma \ref{annulus_2}, $z \notin A(\frac{1}{4}R_k, \frac{2}{5} R_k)$. Similarly, Lemma \ref{hole_punched} and the assumption that $z \notin \mathcal{P}_k$ shows that $z \notin A(\frac{3}{5}R_k, \frac{5}{4}R_k)$. Finally, observe that we cannot have $z \in \{|z| = \frac{2}{5}R_k\}$, $z \in \{|z| = \frac{3}{5}R_k\}$, or $z \in \{|z| = \frac{5}{4} R_k\}$ by equations (\ref{Inner_Vk_Bk}), (\ref{Inner_V1_B1}), (\ref{annulus_3_3/2}), and (\ref{54_annulus}), respectively. Since $z \in A_k \cap J(f_N)$, but $z \notin \mathcal{P}_k$, it follows that $z \in A(\frac{2}{5}R_k, \frac{3}{5}R_k) = V_k$, which proves the theorem.
\end{proof}

\noindent The same argument of the Proof of Theorem \ref{Julia_place} yields the following useful result:

\begin{lem}
	\label{Julia_Place_AkAk1}
	Suppose that $z \in A_k$ and $f_N(z) \in A_{j}$ for some $k \geq 1$ and some $j \in \Z$. Then 
	$$z \in V_k \cup \left( \cup_{j=1}^{n_k} B\left(w^k_j,\frac{R_k}{2^{n_k}}\right)\right)$$
\end{lem}

\section{Conformal Mapping Behavior}
\label{Branched Cover}

We begin Section \ref{Branched Cover} by defining $A_k$, $B_k$ for negative indices $k$, simply by pulling back (by $f$) the definition of $A_k$, $B_k$ for positive $k$ (see Figure \ref{AkBknegativeindices}). With this definition, we will deduce that if $z\in\mathcal{J}(f)$, then either $z$ maps to the filled Julia set of the polynomial like mapping $(f, U, V)$, or all iterates of $z$ lie in $\cup_{k\in\mathbb{Z}}A_k$. We have already estimated the dimension of the Julia set of $(f, U, V)$ in Section \ref{Origin}, so in this Section we study the set of points which have orbits lying in $\cup_{k\in\mathbb{Z}}A_k$.

%We saw at the beginning of Section \ref{Expanding Zeros} that the annuli $B_k$ all belong to $\mathcal{F}(f)$. Thus if $z\in\mathcal{J}(f)$, we must have that 

%In this section, we discuss the useful conformal mapping and branched covering behavior the function $f_N$ has. We first need to define $A_k$, $B_k$, $U_k$ and $V_k$ for indices $k \leq 0$, see Figure \ref{A0}. The definitions below are the same as in Section 15 of \cite{Bis18}.

\begin{definition}
	\label{negative_indices}
	Let $D = B(0, \frac{1}{4}R_1).$ We define
	\begin{equation}
	A_0 = \{z \,:\, z \in D, f_N(z) \in A_1\}.
	\end{equation}
	For integers $k \geq 1$, we define
	\begin{equation}
	\label{negative_Ak}
	A_{-k} = \{z \,:\, z, \dots, f_N^k(z) \in D,  f_N^{k+1}(z) \in A_1\}.
	\end{equation}
	We define $B_k$ and $V_k$ for $k \leq 0$ in the exact same way.
\end{definition}

\begin{figure}[!h]
	\centering
	\scalebox{.25}{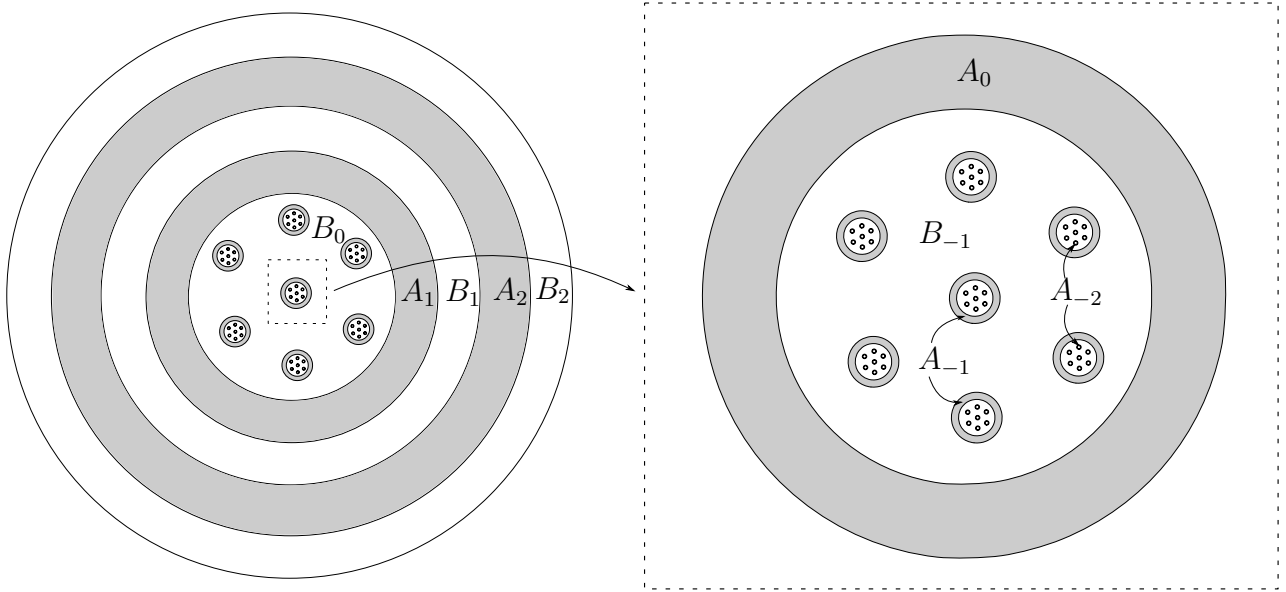}
	\caption{Illustrated is Definition \ref{negative_indices} of $A_k$, $B_k$ for negative $k$.  Each connected component of $A_0$ maps conformally onto $A_1$; therefore, each connected component of $A_0$ contains $2^N$ many connected components of $A_{-1}$. The picture repeats itself as we zoom in. The set $B_0$ is the region between the outer boundaries of the components of $A_0$ and the inner boundary of $A_1$.}
	\label{AkBknegativeindices}
\end{figure}

\begin{notation}
	We will use the following notation in this section.
	\begin{enumerate}
		\item $D = B(0, \frac{R_1}{4})$.
		\item The polynomial-like mapping $(f_N,U,V)$ will be the one defined as in Lemma \ref{polynomial_like}. 
		\item The filled Julia set of $(f_N,U,V)$ will be denoted as $E$.
	\end{enumerate}
\end{notation}

\begin{lem}
	\label{decomposition}
	Let $z \in \C$. Then exactly one of the following is true.
	\begin{enumerate}
		\item We have $z \in E$.
		\item There exists $k \in \Z$ so that $z \in B_k$.
		\item There exists $k \in \Z$ so that $z \in A_k$.
	\end{enumerate}
\end{lem}
\begin{proof}
	The filled Julia set $E$ of $(f_N,U,V)$ is forward invariant and contained in $D$. Therefore, if $z \in E$ it is impossible for $z \in A_k$ or $z \in  B_k$ for any integer $k$ by Definition \ref{negative_indices}. So we will suppose that $z \notin E$. By Definition \ref{annuli}, if $|z| > \frac{R_1}{4}$, then there exists $k \geq 0$ so that $z$ must belong to exactly one of $B_k$ or $A_k$. Therefore, we only have to focus our attention on the case $|z| \leq \frac{R_1}{4}$.
	
	Suppose first that $z \in U \subset \overline{D}$, recalling that $U \subset \overline{D}$ by Lemma \ref{polynomial_like}. Then since $z \notin E$, there exists a smallest integer $l \geq 1$ so that $f_N^l(z) \notin U$. First we consider the case that $|f_N^l(z)| > R_1/4$. Then by (\ref{Inner_B0}), $f_N(\partial D) \subset B_1$, so by our choice of $l$ and the maximum principle we must have either $f_N^l(z) \in A_1$ or $f_N^l(z) \in B_1$. It follows from Definition \ref{negative_indices} that either $z \in A_{1-l}$ or $z \in B_{1-l}$. 
	
	Next, we consider the case that $|f_N^l(z)| \leq \frac{1}{4}R_1$, so that $f^l(z) \in \overline{D} \setminus U$. Observe that by Definition \ref{poly_like_defs}, $f_N(\partial U) = \{z\,:\, |z| = 16r_{N-1}^2 R_1\} \subset B_1$. Indeed, we certainly have $16r_{N-1}^2 R_1 > 4 R_1$, and we may argue using Lemma \ref{Rkest} that $16 r^2_{N-1} R_1 < R_2/4$. We also have $f_N(\partial D) \subset B_1$ by (\ref{Inner_B0}). Therefore by Proposition \ref{zeros}, we must have $f^{l+1}(z) \in B_1$, so that $z \in B_{-l}$. 
\end{proof}

\begin{lem}
	\label{A_preimage}
	Suppose that $z \in f^{-1}_N(A_k)$ for some integer $k \in \Z$. Then $z \in A_j$ for some integer $j$. 
\end{lem}
\begin{proof}
	By assumption we have $f_N(z) \in A_k$ for some integer $k$. By Lemma \ref{decomposition}, either $z \in E$, $z \in A_j$ for some integer $j$, or $z \in B_j$ for some integer $j$. We cannot have $z \in E$, because $f_N(E) \subset E$. If $z \in B_j$ and $j \leq 0$, then by Definition \ref{negative_indices} we must have $f_N(z) \in B_{j+1}$. If $z \in B_j$ for some $j \geq 1$, then $f_N(z) \in B_{j+1}$ by Lemma \ref{Bk}. Therefore, we cannot have $z \in B_j$ for any integer $j$. The only remaining possibility is that we must have $z \in A_j$ for some integer $j$.
\end{proof}

\begin{notation}
	\label{widehat}
	For an open set $\Omega \subset \C$, we let $\widehat \Omega$ denote the union of $\Omega$ and its bounded complementary components. 
\end{notation}

\begin{definition}
	\label{Jordan_Annulus}
	A domain $A \subset \C$ is a topological annulus if the complement of $A$ has two connected components. We say $A$ is a Jordan annulus if the boundary of $A$ consists of two Jordan curves. 
\end{definition}

\begin{definition}
	\label{postcritical_set}
	Let $f: \C \rightarrow \C$ be an entire function. We define $CV(f)$ to be the set of all critical values of $f$, $AV(f)$ the set of asymptotic values of $f$, and $SV(f) = \overline{CV(f) \cup AV(f)}$ to be the set of all singular values of $f$.  We define the postsingular set of $f$ by 
	\begin{equation}
	\label{postsingular_set_eqn}
	P(f) = \overline{\{f^n(z)\,:\, z \in SV(f)\,\, , n \geq 0 \}}.
	\end{equation}
	We define the postcritical set of $f$ in a similar way.
\end{definition}

\begin{lem}
	\label{postcritical_set_B}
	For all sufficiently large $N$, the postsingular set of $f_N$ coincides with the postcritical set, and is a subset of $\cup_{k \geq 1} B_k$.
\end{lem}
\begin{proof}
	By Lemma \ref{branch_covering_hn}, $f_N$ has no asymptotic values. By Proposition \ref{Fatou_critical_points}, for all sufficiently large $N$, all of the critical points of $f_N$ with $|z| \geq \frac{R_1}{4}$ are mapped into $B_{k+1}$ for some $k \geq 1$. By Lemma \ref{crit_values_avoid_A1_fN}, all of the critical points of $f_N$ with $|z| \leq \frac{R_1}{4}$ are mapped into $B_1$. Therefore, $SV(f_N) \subset \cup_{k \geq 1} B_k$. It follows from Lemma \ref{Bk} that $P(f_N) \subset \cup_{k\geq 1} B_k$ as well. 
\end{proof}

Next, we note some of the basic covering map behavior on the annuli $A_k$.

\begin{lem}
	\label{T}
	For all $k \leq 1$, we let $Z_k$ denote $A_k$ or $V_k$. 
	\begin{enumerate}
		\item Let $k \leq 0$ and suppose that $\widehat Z_k'$ is a connected component of $\widehat Z_k$ and $\widehat Z_{k+1}'$ is a connected component of $\widehat Z_{k+1}$ so that $f_N(\widehat Z_k') = \widehat Z'_{k+1}$. Then $f_N: \widehat Z_k' \rightarrow \widehat Z'_{k+1}$ is conformal, and every connected component of $\widehat Z_k$ is a Jordan domain.
		\item Let $k \leq 0$ and suppose that $Z_k'$ is a connected component of $Z_k$ and $Z_{k+1}'$ is a connected component of $Z_{k+1}$ so that $f_N(Z_k') = Z'_{k+1}$. Then $f_N: Z_k' \rightarrow Z'_{k+1}$ is conformal, and every connected component of $Z_k$ is a Jordan annulus.
		\item For all $k \leq 1$, $Z_k$ consists of exactly $2^{(-k+1)N}$ many connected components. 
	\end{enumerate}
\end{lem}
\begin{proof}
	To prove (1), note that by Lemma \ref{postcritical_set_B}, there are no critical values of $f_N$ contained in $\overline{B(0,4R_1)}$. Therefore, the claim follows by Lemma \ref{component_lemma}, which further implies that each connected component of $\widehat Z_k$ for $k \leq 0$ is a Jordan domain. 
	
	Part (2) follows immediately from part (1).  
	
	To see part (3), note that Lemma \ref{component_lemma} implies that $Z_0$ consists of $2^N$ many connected components. Each connected component of $\widehat Z_0$ is mapped conformally onto $\widehat Z_1$ by $f_N$. Therefore $f_N(\widehat Z_0)$ contains the $2^N$ many connected components of $Z_0$, and it follows that each connected component of $\widehat Z_0$ contains $2^N$ many connected components of $Z_{-1}$. Therefore $Z_{-1}$ consists of $2^{2N}$ many connected components. By proceeding similarly, we deduce that every connected component of $\widehat Z_{k+1}$ for $k \leq -1$ contains $2^{N}$ many connected components of $Z_{k}$, and $Z_k$ therefore must have $2^{N(-k+1)}$ many connected components. 
\end{proof}

\begin{lem}
	\label{CoveringMapVk}
	Let $k \geq 1$. Let $W = f_N^{-1}(A_{k+1}) \cap V_k$, which is non-empty by Lemma \ref{Ak}. Then $f_N: W \rightarrow A_{k+1}$ is a degree $n_{k}$ covering map. 
\end{lem}
\begin{proof}
	When $k \geq 2$, observe that on $\phi_N^{-1}(W)$, we have $f_N \circ \phi_N = C_{k} z^{n_{k}}$ and that the mapping $f_N \circ \phi_N: \phi^{-1}_N(W) \rightarrow A_{k+1}$ is a degree $n_{k}$ covering map. Since $\phi_N$ is quasiconformal, it follows that $f_N: W \rightarrow A_{k}$ is a degree $n_{k}$ covering map as well. The $k = 1$ case is similar, except this time we use that fact that on $\phi_N^{-1}(W)$, $f_N \circ \phi_N(z) = q_N(z)$ is a degree $n_1$ covering map to conclude that $f_N \circ \phi_N: \phi_N^{-1}(W) \rightarrow A_{2}$ is a degree $n_1$ covering map.
\end{proof}

Recall that by Lemma \ref{BkFatou}, for $k \geq 1$ the annuli $B_k$ belong to the Fatou set of $f_N$. This motivates the following definition.

\begin{definition}
	\label{central_series}
	For $k \geq 1$, we define $\Omega_k$ to be the Fatou component that contains $B_{k}$. 
\end{definition}

\begin{rem}
	It is readily verified that each $\Omega_k$ is multiply connected. By Definition \ref{parameter_defn}, we have $f_N(0) = 0$, and $0$ is in the Julia set of $f_N$ since it is a repelling fixed point. Indeed, by Definition \ref{parameter_defn} we have
	$$|f_N'(0)| = |q_N'(0)| \cdot |\phi_N'(0)|.$$
	We verify from (\ref{derivative_qN}) that $q_N'(0) = r_N$, and by Theorem \ref{identity_origin}, for all $N$ sufficiently large we may assume that $\frac{1}{2} \leq |\phi_N'(0)| \leq \frac{3}{2}$. Therefore, for all $N$ sufficiently large, by Lemma \ref{Rkest} we have
	$$|f_N'(0)| \geq \frac{1}{2} r_N > 1.$$
	Therefore $0$ is a repelling fixed point for $f_N$. Since $B_{k} \subset \Omega_k$ for all $k \geq 1$, it follows that $\Omega_k$ cannot be simply connected.
\end{rem}

\begin{lem}
	\label{central_series_distinct}
	Suppose that $j, k \geq 1$ and $j \neq k$. Then $\Omega_j \neq \Omega_k$. 
\end{lem}
\begin{proof}
	By Lemma \ref{Bk}, we have $f_N(B_{k}) \subset B_{k+1}$ for all $k \geq 1$. Therefore, we must have $f_N(\Omega_k) \subset \Omega_{k+1}$ for all $k \geq 1$, and since $f_N$ has no asymptotic values by Lemma \ref{branch_covering_hn}, we actually have 
	\begin{equation}
	\label{Omega_Maps_Next}
	f_N(\Omega_k) = \Omega_{k+1},
	\end{equation}
	for all $k \geq 1$, (see Corollary 2 of \cite{Herring}).
	
	Suppose for the sake of contradiction that we have $\Omega_j = \Omega_k$ for some $k > j \geq 1$. It follows from Definition \ref{central_series} that $\Omega_j = \Omega_{j+1}$. This combined with (\ref{Omega_Maps_Next}) implies that $\Omega_j$ is unbounded. This contradicts Theorem 1 of \cite{Baker}.
\end{proof}

\begin{definition}
	\label{A_def}
	We define
	\begin{equation}
	\label{A}
	A := \bigcup_{k \in \Z} A_k,
	\end{equation}	
	and
	\begin{equation}
	\label{X}
	X := \{z: f^n(z) \in A\,\, , n = 0,1,\dots\}.
	\end{equation}
\end{definition}

\begin{lem}
	\label{A_backward_invariant}
	There exists $M \in \N$ so that for all $N \geq M$, we have $f_N^{-1}(A) \subset A.$
\end{lem}
\begin{proof}
	By Lemma \ref{A_preimage}, we have $f_N^{-1}(A_k) \subset A$ for each integer $k$. Since
	$$f_N^{-1}(A) = \bigcup_{k \in \Z} f_N^{-1}(A_k),$$
	the conclusion follows immediately.
\end{proof}

\begin{definition}
	\label{E'}
	Recall that the filled Julia set of the polynomial-like mapping $(f_N,U,V)$ from Lemma \ref{polynomial_like} is denoted as $E$. Define
	\begin{equation}
	E' = \bigcup_{n=0}^{\infty} f_N^{-n}(E).
	\end{equation}
\end{definition}

\begin{lem}
	\label{E'_dim}
	The Hausdorff dimension of $E$ and $E'$ are the same. 
\end{lem}
\begin{proof}
	This follows immediately from Definition \ref{E'} along with (\ref{countable_stability}) and (\ref{dimension_entire}).
\end{proof}		

The following lemmas give us some basic rules for how orbits of points in $X$ behave.	

\begin{lem}
	\label{Julia_Place_X}
	The Julia set of $f_N:\C \rightarrow \C$ is a subset of $E' \cup X$. 
\end{lem}

\begin{rem}
	In fact, we actually have $\Julia(f_N) = E' \cup X$, and this will become apparent in the later sections.
\end{rem}

\begin{proof}
	Since $(f_N,U,V)$ is a polynomial-like mapping and $E$ is its filled Julia set, $E$ coincides with the closure of the repelling periodic cycles for $(f_N,U,V)$. These are also repelling periodic cycles for $f_N$ viewed as an entire function, so $E  \subset \Julia(f_N)$. Since $\Julia(f_N)$ is backwards invariant, it follows immediately from Definition \ref{E'} that $E' \subset \Julia(f_N)$ as well.
	
	Now, suppose that $z \in \Julia(f_N)$, but $z \notin E'$. The set $E'$ is both forward and backward invariant. Therefore, for all $n \geq 0$, we have $f_N^n(z) \notin E'$, and in particular, we have $f_N^n(z) \notin E$. By Lemma \ref{BkFatou} and Definition \ref{negative_indices}, if $f_N^n(z) \in B_k$ for some $n \geq 0$ and $k \in \Z$, we must have $z$ in the Fatou set of $f_N$. Therefore, by Lemma \ref{decomposition} we must have $f_N^n(z) \in A$ for all $n \geq 0$ so that $z \in X$, as desired.
\end{proof}

\begin{lem}
	\label{other widehat help}
	Suppose that $z \in A_k$ and $f_N(z) \in A_j$ for some $j \in \Z$. Then $j \leq k+1 $.
\end{lem}
\begin{proof}
	First, suppose that $z \in A_k$ for some $k \leq 0$. Then by Definition \ref{negative_indices}, $f_N(z) \in A_{k+1}$, so that $j = k+1$.
	
	Next, suppose that $z \in A_k$ for some $k \geq 1$. Then by (\ref{Inner_Bk}), we have $f_N(|z|=4R_k) \subset B_{k+1}$. Therefore, by the maximum principle for holomorphic functions, we must have $f_N(B(0,4R_k)) \subset \widehat B_{k+1}.$ So if $f_N(z) \in A_j$, we must have $j \leq k+1$. 
\end{proof}

\begin{lem}
	\label{widehat help}
	Suppose that $z \in A_k$ and $f_N(z) \in A_j$ for some $k \geq j$. Then $k \geq 1$, and there exists a petal $P_k \subset \mathcal{P}_k$ such that $z \in P_k$.
\end{lem}
\noindent
\begin{proof}
	If $k \leq 0$, we have $f_N(z) \in A_{k+1}$ by Definition \ref{negative_indices}, so we must have $k \geq 1$. 
	
	By Lemma \ref{Julia_Place_AkAk1}, we must have $z \in V_k$ or we must have $z \in P_k$ for some petal $P_k \subset \mathcal{P}_k$. If $z \in V_k$, then by Corollary \ref{helpful_Vk_cor}, we must have $f_N(z) \in B_k \cup A_{k+1} \cup B_{k+1}$. Since $f_N(z) \in A_j$ and $j \leq k$ we must have $z \in P_k$ for some petal $P_k \subset \mathcal{P}_k.$
\end{proof}

\begin{lem}
	\label{component_counter}
	Suppose that $\Omega \subset A_{k+1}$ is a Jordan domain for some $k \geq 1$. Then the number of connected components of $f_N^{-1}(\Omega)$ that are contained in $A_k$ is $n_{k+1}$, and each component is a Jordan domain.
\end{lem}
\begin{proof}
	By Lemma \ref{postcritical_set_B}, there are no singular values of $f_N$ in $\overline{U}$. Therefore, $f_N^{-1}(U)$ is the disjoint union of Jordan domains. By Lemma \ref{CoveringMapVk}, there are exactly $n_k$ many connected components of $f_N^{-1}(U)$ contained inside of $V_k$. By Lemma \ref{petal_radius}, there is exactly one connected component of $f_N^{-1}(U)$ contained in each petal $P_k \subset \mathcal{P}_k$, which yields $n_k$ many more connected components of $f_N^{-1}(U)$. There are no other connected components contained inside of $A_k$ by Lemma \ref{Julia_Place_AkAk1}. Therefore, the total number of connected components contained inside of $A_k$ is $n_k + n_k = n_{k+1}$.
\end{proof}

\begin{definition}
	If a point $z \in X$, then for each $n\geq 0$, $f_N^n(z) \in A_{k(n,z)}$ for some integer $k(n,z)$.  By Lemma \ref{other widehat help} 
	\begin{equation}
	\label{orbit_sequence}
	k(z,n+1) \leq k(z,n) + 1.
	\end{equation}
	We call the sequence $(k(z,n))_{n=0}^{\infty}$ the orbit sequence of $z$.
\end{definition}

This inspires the following definition, which will be crucial to the proof of Theorem \ref{main_theorem}.  

\begin{definition}
	\label{xyz}
	Suppose that $z \in X$. For $n \geq 1$, if $k(z,n) < k(z,n-1) + 1$, we will say that $z$ \textit{moves backwards} on the $n^{\textrm{th}}$ iterate. We will sometimes omit the iterate $n$ and just say that $z$ moves backwards. We let $Y$ denote the set of all points in $X$ that move backwards for infinitely many distinct iterates, and we let $Z$ denote the set of all points in $X$ that move backwards for only finitely many iterates. 
\end{definition}

\begin{rem}
	Suppose that $z, f(z),\dots, f^{n}(z) \in A$, but we perhaps do not have $f^{n+1}(z) \in A$. Then we may still define the finite orbit sequence $(k(z,j))_{j=0}^n$. Therefore, we may still speak of a point $z$ moving backwards for the iterates where its finite orbit sequence is defined.
\end{rem}

\begin{rem}
	\label{moves_backwards_set}
	Let $W$ be connected with $W \subset A$ and suppose that $f^j_N(W) \subset A$ for $j = 1,\dots, n$. Then since $W$ is connected, each point $z \in W$ has the same finite orbit sequence $(k(z,j))_{j=0}^n$. In these situations, we will say that the set $W$ moves backwards whenever any of the points $z \in W$ move backwards. 
\end{rem}

\begin{rem}
	\label{yz}
	It follows immediately from Definition \ref{xyz} that \[X = Y \cup Z.\] 
\end{rem}

\section{A First Dimension Estimate}
\label{Dimension Y}

We deduced in Section \ref{Branched Cover} that $\mathcal{J}(f)\subset E'\cup X$. We have already estimated the dimension of $E'$, and we now move on to estimating the dimension of $X=Y\sqcup Z$. This Section will be devoted to estimating the dimension of $Y$ (the set of points which move backwards infinitely often). In fact we will show that the dimension of $Y$ can be taken arbitrarily close to $0$. Although the details are somewhat technical, the idea is simple: we build a sequence of coverings $\cC_m$ of $Y$ by pulling back the annuli $A_k$ under appropriate branches of the inverse of $f^m$. The diameters of elements in $\cC_m$ are estimated by standard distortion estimates for conformal mappings.

More precisely, our goal in this section is to show that for any $t >0$, there exists some $M \in \N$ so that for all $N \geq M$, we have $\dim_H(Y) \leq t$. We'll start by formally constructing a sequence of coverings $\cC_m$ of $Y \cap A_1$, for $m \geq 0$ using the dynamics of $f_N$. Our initial covering $\cC_0$ will have exactly one element, the annulus $A_1$. We first describe how to construct $\cC_1$ from $\cC_0$ (see Figure \ref{C1covering}).

\begin{figure}[!htb]
	\centering	
	\scalebox{.5}{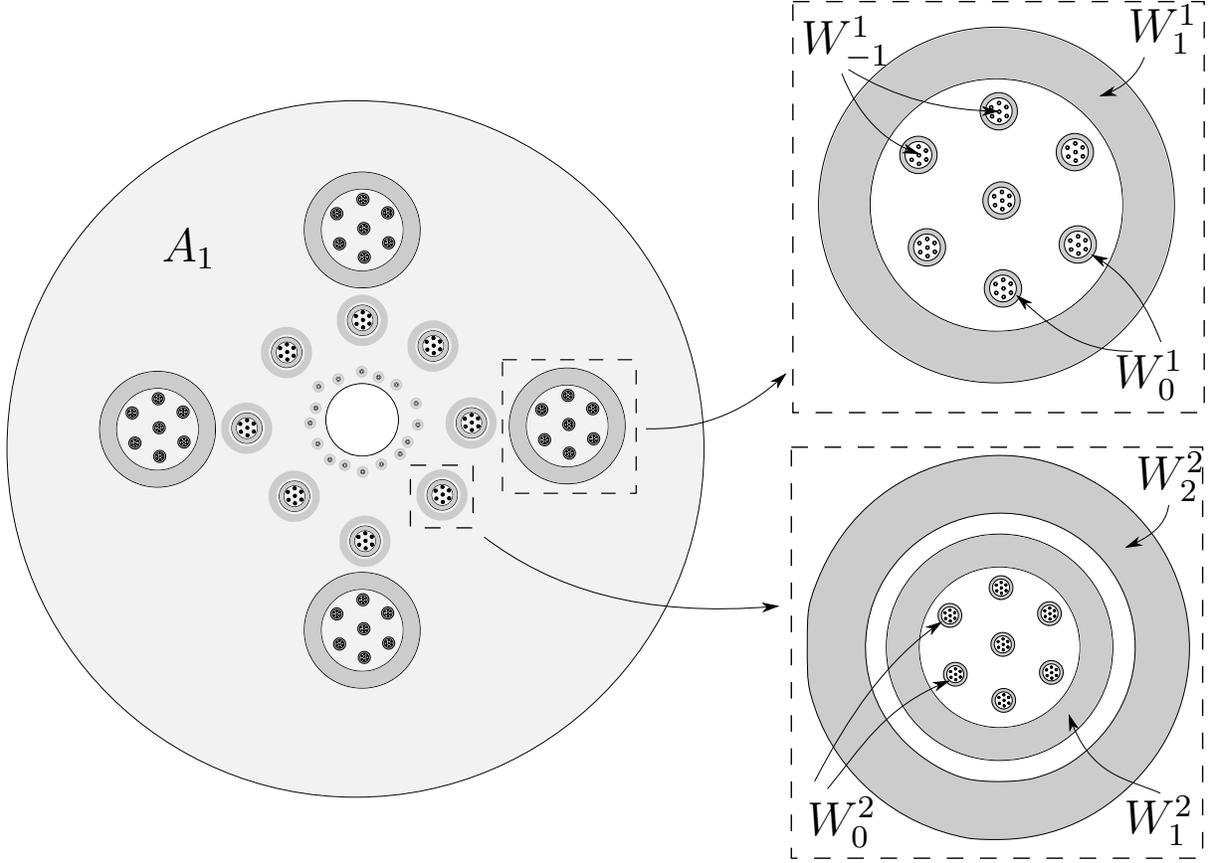}
	\caption{Illustrated is the covering $\cC_1$, and the notation $W_k^n$ for elements of $\cC_1$ (see Notation \ref{Wnk}). The elements of $\cC_m$ for $m>1$ are obtained by essentially placing a scaled-down copy of $\cC_{m-1}$ in each annulus in the covering $\cC_{m-1}$.  }
	\label{C1covering}
\end{figure}

\begin{lem}
	\label{cC1}
	There exists a collection of sets $\cC_1$ that has the following properties.
	\begin{enumerate}
		\item Every element in $\cC_{1}$ is a subset of an element in $\cC_0$.
		\item $\cC_1$ is a countable cover of $Y \cap A_1$.
		\item Let $W$ be an element of $\cC_1$. Then there exists an integer $n \geq 1$ and an integer $k \in \Z$ such that $f_N^{n-1}(W) \subset A_{n}$, and $f_N^n(W)$ is a connected component of $A_k$ for $k \leq n$.
		\item Every element of $\cC_1$ moves backwards once.
	\end{enumerate}
\end{lem}

\begin{proof}
	For each $z \in A_1 \cap Y$, by Definition \ref{xyz} and (\ref{orbit_sequence}), there is a smallest positive integer $n$ so that $f^n_N(z) \in A_k$ for some $k \leq n$. We remark that it is possible that $k$ is a non-positive integer.  Let $W$ denote the connected component of $f^{-n}_N(A_k)$ which contains $z$. The collection of all distinct components obtained by applying this procedure to all $z \in Y$ is denoted by $\cC_1$. We check that the properties in the lemma hold. 
	\begin{enumerate}
		\item By Lemma \ref{A_backward_invariant}, $W \subset A$, and since $W$ is connected and contains $z \in A_1$, we have $W \subset A_1$.
		\item Any two elements of $\cC_1$ are disjoint. Since $A_1$ is bounded and all elements of $\cC_1$ are open, the collection $\cC_1$ is countable. 
		%By Lemma \ref{petal_radius}, and, in particular, (\ref{conformal_3}), for all $k \leq 0$ there exists $z \in A_1$ so that $f_N(z) \in A_k$. It follows that the collection $\cC_1$ is not finite and therefore countable.
		\item This follows from the construction of each $W$, since we chose the smallest positive integer $n$ such that $f^n_N(z) \in A_k$ for $k \leq n$.
		\item By Lemma \ref{A_backward_invariant}, $f_N^j(W) \subset A$ for $j = 0,\dots,n$, so this claim follows since $n$ is the smallest positive integer so that $f^n_N(z) \in A_k$ for $k \leq n$.
	\end{enumerate}
	This proves the claim.
\end{proof}

We now show how to construct $\cC_m$ for $m > 1$. Our procedure is inductive and described in the Lemma below.

\begin{lem}
	\label{cCm}
	Let $m \geq 1$. Suppose there exists a collection of subsets $\cC_m$ that satisfy the following properties.
	\begin{itemize}
		\item $\cC_m$ is a countable cover of $Y \cap A_1$.
		\item Let $W$ be an element of $\cC_m$. Then there exists an integer $n \geq 1$ and an integer $k \in \Z$ such that $f_N^n(W)$ is a connected component of $A_k$ and $f_N^{n-1}(W) \subset A_j$ for some $j \geq k$. Moreover $W$ moves backwards for the $m^{\textrm{th}}$ time on the $n^{\textrm{th}}$ iterate.
	\end{itemize}
	Then there exists a collection of subsets $\cC_{m+1}$ that satisfy the following properties.
	\begin{enumerate}
		\item $\cC_{m+1}$ is a refinement of $\cC_m$, i.e, every element in $\cC_{m+1}$ is a subset of an element in $\cC_m$.
		\item Each element of $\cC_m$ contains countably many elements of $\cC_{m+1}$, and $\cC_{m+1}$ is a countable cover of $Y \cap A_1$.
		\item Every element of $\cC_{m+1}$ moves backwards $m+1$ many times.
	\end{enumerate}
	Moreover, let $W$ be an element of $\cC_{m+1}$, and let $z \in Y$ satisfy $z \in W \subset W'$ where $W'$ is an element of $\cC_m$. Let $(k(z,j))_{j=0}^{\infty}$ denote the orbit sequence of $z$. Let $n$ be the value such that $f_N^{n-1}(W) \subset f_N^{n-1}(W') \subset A_{k(z,n-1)}$, $f_N^n(W')$ is a connected component of  $A_{k(z,n)}$, where $k(z,n) \leq k(z,n-1)$ and $z$ moves backwards for the $m^{\textrm{th}}$ time on the $n^{\textrm{th}}$ iterate. Then there exists a value $q \geq 1$ such that $f_N^{n+q-1}(W) \subset A_{k(z,n+q-1)}$, $f_N^{n+q}(W)$ is a connected component of $A_{k(z,n+q)}$, where $k(z,n+q) \leq k(z,n+q-1)$, and $W$ moves backwards for the $m+1$st time on the $n+q$th iterate.
\end{lem}
\begin{proof}
	Choose $z \in Y \cap A_1$. Then there exists an element $W'$ of $\cC_m$ containing $z$. Let $n$ be the integer such that $f_N^n(W') = A_{k(z,n)}$, $k(z,n) \leq k(z,n-1)$, and $z$ moves backwards for the $m^{\textrm{th}}$ time on the $n^{\textrm{th}}$ iterate. Then since $z \in Y$, it must move backwards again. Therefore, there exists a smallest value $q \geq 1$ such that $f_N^{n+q}(z) \in A_{k(z,n+q)}$ and $k(z,n+q) \leq k(z,n+q-1)$. We let $W$ denote the connected component of $f_N^{-(n+q)}(A_{k(z,n+q)})$ which contains $z$, and we let $\cC_{m+1}$ denote the collection of all distinct components obtained by applying this procedure to all $z \in Y$. We now prove the desired properties.
	
	\begin{enumerate}
		\item Let $W$ be an element of $\cC_{m+1}$. Then by construction, there exists some point $z \in Y \cap A_1$ contained inside of $W$. Let $W'$ denote the element of $\cC_m$ that contains $z$. Let $n$ be the integer so that $f^n(W') = A_k$ for some $k$, where $W'$ moves backwards for the $m^{\textrm{th}}$ time on the $n^{\textrm{th}}$ iterate. We must have $f^n(W) \subset A$, so since $W$ is connected we have $f^n(W) \subset A_k$ and it follows that $W \subset W'$.
		\item We already know $\cC_m$ is countable and by the construction $\cC_{m+1}$ covers $Y \cap A_1$. Let $W'$ be an element of $\cC_m$, and let $n$ be the integer such that $W'$ moves backwards for the $m^{\textrm{th}}$ time on the $n^{\textrm{th}}$ iterate. Then $f_N^n(W') = A_k$ for some integer $k$. If $k \geq 1$, then by (\ref{conformal_3}) , there exists countably many elements of $\cC_{m+1}$ contained in $W'$. If $k \leq 0$ is negative, then $f_N^{n-k+1}(W') = A_1$ by Lemma \ref{T}, and we can apply the same reasoning for the $k \geq 1$ case to see that there exists countably many elements of $\cC_{m+1}$ contained in $W'$. Since there are countably many elements $W'$ in $\cC_m$, the collection $\cC_{m+1}$ is also countable.
		\item That every element of $\cC_{m+1}$ moves backwards $m+1$ many times follows from the definition of $n$ and the choice of $q$ in the construction.  
	\end{enumerate}
\end{proof}

In the rest of our analysis, we will mostly be focused on understanding what happens when we refine from $\cC_m$ to $\cC_{m+1}$. Therefore, the following notation will be convenient (see Figures \ref{C1covering} and \ref{widehat covering pic}).	
\begin{notation}
	\label{Wnk}
	Let $W'$ be an element of $\cC_m$, and let $W \subset W'$ be an element of $\cC_{m+1}$. We will denote $W'$ as $W^n_k$, where $n\geq 1$ is the iterate where $W'$ moves backwards for the $m^{\textrm{th}}$ time, and $f_N^n(W^n_k) = A_k$ for some integer $k$. We will say that $W'$ is of the form $W^n_k$ for $n \geq 1$ and $k \in \Z$. Likewise, we will denote $W$ as $W^{n+q}_j$, where $n+q$ is the iterate where $W$ moves backwards for the $m+1$st time, and $f_N^{n+q}(W)= A_j$. We will say $W$ is of the form $W^{n+q}_j$ for $q \geq 1$.
\end{notation}

The following lemma states that the mappings used to define the collections $\cC_m$ are conformal with bounded distortion. We state it precisely below.

\begin{figure}[!htb]
	\centering	
	\scalebox{.9}{%% Creator: Inkscape 1.0 (4035a4fb49, 2020-05-01), www.inkscape.org
%% Accompanies image file 'hard_covering_dots_only.pdf' (pdf, eps, ps)
\begingroup%
  \makeatletter%
  \providecommand\color[2][]{%
    \errmessage{(Inkscape) Color is used for the text in Inkscape, but the package 'color.sty' is not loaded}%
    \renewcommand\color[2][]{}%
  }%
  \providecommand\transparent[1]{%
    \errmessage{(Inkscape) Transparency is used (non-zero) for the text in Inkscape, but the package 'transparent.sty' is not loaded}%
    \renewcommand\transparent[1]{}%
  }%
  \providecommand\rotatebox[2]{#2}%
  \newcommand*\fsize{\dimexpr\f@size pt\relax}%
  \newcommand*\lineheight[1]{\fontsize{\fsize}{#1\fsize}\selectfont}%
  \ifx\svgwidth\undefined%
    \setlength{\unitlength}{485.74036744bp}%
    \ifx\svgscale\undefined%
      \relax%
    \else%
      \setlength{\unitlength}{\unitlength * \real{\svgscale}}%
    \fi%
  \else%
    \setlength{\unitlength}{\svgwidth}%
  \fi%
  \global\let\svgwidth\undefined%
  \global\let\svgscale\undefined%
  \makeatother%
  \begin{picture}(1,0.9810531)%
    \lineheight{1}%
    \setlength\tabcolsep{0pt}%
    \put(0,0){\includegraphics[width=\unitlength,page=1]{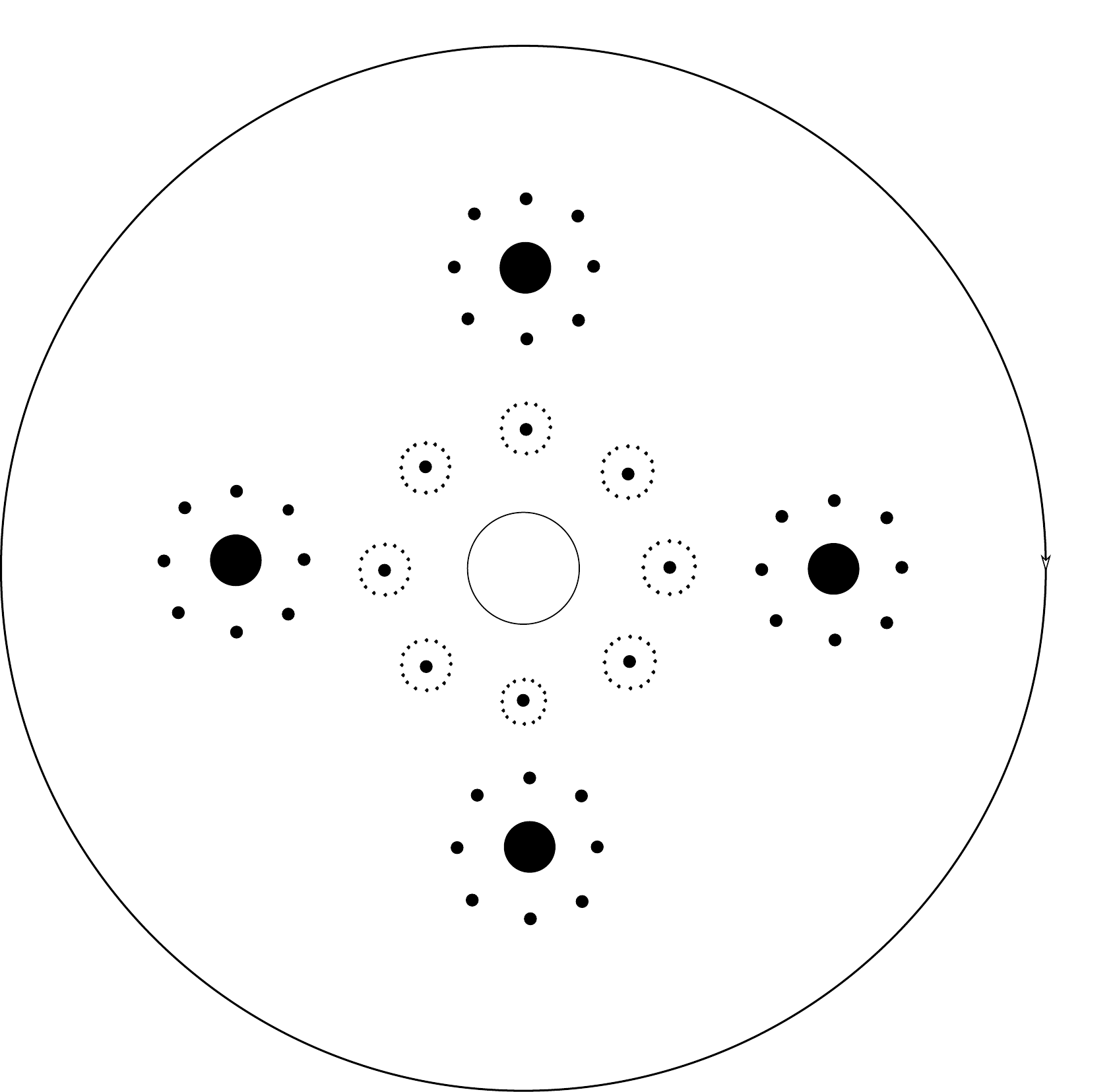}}%
    \put(0.43515876,0.95482378){\color[rgb]{0,0,0}\makebox(0,0)[lt]{\lineheight{1.25}\smash{\begin{tabular}[t]{l}$W^n_k$\end{tabular}}}}%
    \put(0.64748545,0.75113141){\color[rgb]{0,0,0}\makebox(0,0)[lt]{\lineheight{1.25}\smash{\begin{tabular}[t]{l}$\widehat W^{n+1}_{k}$\end{tabular}}}}%
    \put(0,0){\includegraphics[width=\unitlength,page=2]{hard_covering_dots_only.pdf}}%
    \put(0.21395801,0.28569543){\color[rgb]{0,0,0}\makebox(0,0)[lt]{\lineheight{1.25}\smash{\begin{tabular}[t]{l}$\widehat W^{n+2}_{k+1}$\end{tabular}}}}%
    \put(0,0){\includegraphics[width=\unitlength,page=3]{hard_covering_dots_only.pdf}}%
    \put(0.61042037,0.30503111){\color[rgb]{0,0,0}\makebox(0,0)[lt]{\lineheight{1.25}\smash{\begin{tabular}[t]{l}$\widehat W^{n+3}_{k+2}$\end{tabular}}}}%
    \put(0,0){\includegraphics[width=\unitlength,page=4]{hard_covering_dots_only.pdf}}%
  \end{picture}%
\endgroup%
}
	\caption{An illustration of components of the form $\widehat W^{n+q}_{k+q-1} \in \widehat \cC_{m+1}$ contained inside of some $W^n_k \in \cC_m$ for $q=1,2,$ and $3$. The component $W^n_k$ is bounded by the innermost and outermost circle.}
	\label{widehat covering pic}
\end{figure}

\begin{lem}
	\label{section_8_distortion}
	Let $m > 0$ and let $W^n_k$ be an element of $\cC_m$. Then there exists a Jordan domain $B$ containing $W^n_k$ such that  $f_N^n: B \rightarrow f_N^n(B)$ is conformal. Moreover, 
	\begin{enumerate}
		\item When $k \leq 0$ the modulus of $B \setminus \widehat W^n_k$ is bounded below by $(2\pi)^{-1} \log(2)$, and $f_N^n(B)$ is an element $\widehat \sigma_{1-k}$ from Remark \ref{Modulus_lower_bound}.
		\item When $k \geq 1$, $f_N^n(B) =  B(0,4R_{k+1})$, and the modulus of $B \setminus \widehat W^n_k$ is bounded below by $(2\pi)^{-1} \log(2).$
	\end{enumerate}
\end{lem}
\begin{proof}
	Fix $m > 0$ and choose some $k \geq 1$. Let $W^n_k$ be an element of $\cC_m$. By following (\ref{conformal_4}) and using the fact that $N \geq 5$, we deduce
	\begin{equation}
	\label{section_8_distortion_eq1}
	\frac{\delta}{2} C_k R_k^{n_k} > 2^{n_{k-1} - 1} 4 R_{k+1} > 16 R_{k+1}.
	\end{equation}
	Therefore, we have 
	
	\begin{equation}
	\label{conformal_ball_big_enough}
	\widehat{A}_k = B(0,4R_k) \subset  B(0,4R_{k+1}) \subset  B\left(0,\frac{\delta}{2} C_k R_k^{n_k}\right) \subset B\left(0,\frac{\delta}{2} C_j R_j^{n_j}\right) .
	\end{equation}
	Let $B$ denote the connected component of $f_N^{-n}(B(0,4R_{k+1}))$ which contains $W^n_k$. By Lemma \ref{widehat help}, $f_N^{n-1}(W^n_k)$ is the subset of a petal $P_j \subset A_j$ for some $j \geq k$. Since $f_N^n(\widehat W^n_k) = B(0,4R_k)$, $f_N^{n-1}(\widehat W^n_k)$ contains a zero $w$ of $f_N$ and we can deduce by (\ref{conformal_ball_big_enough}) that we have 
	\begin{equation*}
	f_N^{n-1}(B) \subset B(w,\lambda(\exp(\pi/n_j)-1)R_j) \subset A_j.
	\end{equation*}
	Therefore, by (\ref{conformal_3}), $f_N: f_N^{n-1}(B) \rightarrow B(0,4R_{k+1})$ is conformal. Since $f_N^{n-1}(B) \subset A_j$, we have $f_N^l(B) \subset A$ for $l = 0,\dots, n-1$ by Lemma \ref{A_backward_invariant}. Therefore, by Lemma \ref{postcritical_set_B}, $f_N^{n-1}:B \rightarrow f_N^{n-1}(B)$ is conformal. Therefore, the composition $f_N^n: B \rightarrow B(0,4R_{k+1})$
	is conformal and $B$ is a Jordan domain. The modulus lower bound for $B \setminus \widehat W^n_k$ follows from (\ref{Rkest_eqn4}).
	
	Now consider the case of $k \leq 0$. Let $A_k'$ be the connected component of $A_k$ such that $f_N^n(W^n_k) = A_k'$. Then the boundary of $\widehat{A}_k'$ is one of the elements $\gamma_{1-k}$ of $\Gamma_{1-k}$ from Definition \ref{GreensLines}. Therefore there exists an element $\sigma_{1-k}$ from Remark \ref{Modulus_lower_bound} so that the modulus of $\widehat{\sigma}_{1-k} \setminus \widehat{\gamma}_{1-k}$ is bounded below by $(2\pi)^{-1} \log(2).$ Let $B$ denote the connected component of $f_N^{-n}(\widehat{\sigma}_{k+1})$ that contains $W^n_k$. Then $f_N^n:B \rightarrow \widehat{\sigma}_{k+1}$
	is conformal by a similar argument to the $k \geq 1$ case.
\end{proof}

\begin{rem}
	\label{cCm_Jordan_Annulus}
	It follows immediately from Lemma \ref{section_8_distortion} that for all $m \geq 0$, every element of $\cC_m$ is a Jordan annulus.
\end{rem}

\begin{rem}
	\label{L''}
	Let $m \geq 0$ and let $W^n_k$ be any element of $\cC_m$, and let $K$ be any compact subset of $\widehat W^n_k$. Let $B$ be the Jordan domain from Lemma \ref{section_8_distortion} containing $W^n_k$. By Remark \ref{Bounded_conformal_distortion}, Corollary \ref{Kobe} applies with $D = B$ to $f_N^n: B \rightarrow f_N^n(B)$ with $U = W^n_k$ and constant $C = L'' >0$ that does not depend on $m$, the element $W^n_k$, or the compact set $K$.
\end{rem}	

We now begin estimating the diameters of elements in $\cC_m$ for the purpose of estimating the dimension of $Y$.  The following Lemma is the same inequality (17.1) in \cite{Bis18}. Our proof is similar, and we include all the details for the sake of the reader. 

\begin{lem}
	\label{easy_W}
	Fix some $t > 0$ and let $W^n_k \in \cC_m$ be given. Then there exists some $M \in \N$ so that for all $N \geq M$ we have
	\begin{equation}
	\label{easy_W_eqn}
	\sum_{W^n_{k-1} \subset \widehat{W^n_k}} \diam(W^n_{k-1})^t \leq \frac{1}{100} \diam(W^n_k)^t,
	\end{equation}
	where the sum in (\ref{easy_W_eqn}) is taken over all components of the form $W^n_{k-1}$ in $\cC_m$ that are contained in $\widehat W^n_k$.
\end{lem}

\begin{rem}
	The specific constant $1/100$ is not particularly important. In fact, by increasing $M$, it can be replaced by any arbitrarily small positive constant.
\end{rem}

\begin{proof}
	The proof splits into two cases: the case of $k >1$, and the case of $k \leq 1$. 
	
	Suppose that $k > 1$. Then there is exactly one element of the form $W^n_{k-1} \subset \widehat W^n_k$. Indeed, the mapping $f_N^n: \widehat W^n_{k} \rightarrow \widehat A_{k}$ is a conformal bijection, and since there is only one $A_{k-1} \subset \widehat A_k$ when $k >1$, there can only be one $W^n_{k-1} \subset \widehat W^n_k$. Thus we have
	\begin{equation}
	\label{easy_W_eqn1}
	\frac{\diam(W^n_{k-1})^t}{\diam(W^n_k)^t} \stackrel{\textrm{Remark } \ref{L''}}{\leq} (L'')^t \frac{\diam(A_{k-1})^t}{\diam(A_{k})^t} = (L'')^t \frac{R_{k-1}^t}{R_k^t} \stackrel{\textrm{Lemma } \ref{Rkest}}{\leq} (L'')^t \left(\frac{1}{4R_{k-1}}\right)^t \leq \frac{(L'')^t}{R^t_1}.
	\end{equation}

	%$y \stackrel{t=x^n}{=\joinrel=\joinrel=} x$
	
	Suppose that $k \leq 1$. Then there exists a connected component $A_k'$ of $A_k$ so that $f_N^{n}: \widehat W^n_k \rightarrow \widehat A_k'$ is conformal. By Lemma \ref{T}, we have the following composition of conformal bijections
	\begin{equation*}
	\widehat W^n_k \xrightarrow[]{f_N^n} \widehat A_k' \xrightarrow[]{f^{1-k}_N} \widehat A_1,
	\end{equation*}
	Therefore, the number of elements of the form $W^n_{k-1} \subset \widehat W^n_{k}$ is equal to the number of connected components of $A_0$, which is $2^N$ by part (3) of Lemma \ref{T}. Next, recall that the outermost boundary of each connected component of $A_k$ and $A_{k-1}$ coincides with an element $\gamma_{-k+1}$ and an element $\gamma_{-k+2}$ from Definition \ref{GreensLines}, respectively (here we take the convention that $\gamma_0$ is $\{z:|z| = 4R_1\}$). Therefore, we obtain
	\begin{align}
	\label{easy_W_eqn2}
	\frac{\sum_{W^n_{k-1} \subset \widehat{W^n_k}} \diam(W^n_{k-1})^t}{\diam(W^n_k)^t} &\stackrel{\textrm{Remark } \ref{L''}}{\leq} (L'')^t \frac{\sum_{A_{k-1}' \subset A_k'} \diam(A_{k-1}')^t}{\diam(A_k')^t} \\ &\stackrel{\textrm{Lemma }\ref{level_lines_p_N}}{\leq}  (L'')^t \frac{2^N \frac{\diam(A_k')^t}{R^t_1}}{\diam(A'_k)^t} \stackrel{\textrm{Definition } \ref{lambdaN}}{=} \frac{(L'')^t \cdot  2^N}{R^t_1}.  \nonumber\end{align}
	By (\ref{easy_W_eqn1}) and (\ref{easy_W_eqn2}) the conclusion of the Lemma now follows by choosing $M$ large enough so that for all $N \geq M$ we have
	$$\frac{(L'')^t \cdot 2^N}{R^t_1} \leq \frac{1}{100}.$$
	Such an $M$ exists by applying Lemma \ref{Rkest}, inequality (\ref{Rkest_eqn3}); see (\ref{decay}).
\end{proof}

%\begin{figure}[!htb]
%	\centering	
%	\scalebox{.8}{\input{simple covering.pdf_tex}}
%	\caption{An illustration of the relative positions of elements in $\cC_m$. On the left, we illustrate $k \geq 2$. In this case, $W^n_k$ is surrounded by one component of the form $W^n_{k+1}$ and surrounds one component of the form $W^n_{k-1}$. On the right is shown the case of $k \leq 1$: each component $W^n_k$ surrounds $2^N$ many components of the form $W^n_{k-1}$ (compare to Figure \ref{gammaillustration}).}
%	\label{hole pic}
%\end{figure}

We move on to describing how to change the covering $\cC_m$ of $Y \cap A_1$ by topological annuli into a simpler covering $\widehat\cC_m$ by topological disks.

\begin{definition}
	\label{hatcC1}
	We define $\widehat \cC_0$ to be $\widehat A_1$ and we define $\widehat \cC_1$ to be the collection
	\begin{equation}
	\label{hatcC1_eqn}
	\{\widehat W^n_n\,:\, n \geq 1 \textrm{ and } W^n_n \in \cC_1\}
	\end{equation}
\end{definition}	
\begin{rem}
	\label{maximality_1}
	When $m =1$, we note that the covering $\widehat \cC_m$ satisfies
	\begin{itemize}
		\item $\widehat \cC_m$ is a covering of $\cC_m$, and hence is a covering of $A_1 \cap Y$, and
		\item If $\widehat W^n_k \in \widehat \cC_m$, then $k \geq 1$.
	\end{itemize}
\end{rem}

\begin{definition}
	\label{hatcCm}
	Let $m \geq 1$ and assume that $\widehat \cC_m$ has been constructed and satisfies $(1)$ and $(2)$ of Remark \ref{maximality_1}. Let $\widehat W^n_k \in \cC_m$. Then $\widehat W^n_k$ contains a sequence of components $W^n_j \in \cC_m$ for $j \leq k$. Fix $W^n_j$, and consider the elements of $\cC_{m+1}$ contained inside of $W^n_j$ of the form $W^{n+q}_{j+q-1}$. If $j \geq 1$, then all $q \geq 1$ occur. If $j \leq 0$, then $q$ must satisfy $q \geq 2-j$. Either way, for each valid choice of $q$, the elements of $\widehat \cC_m$ which lie inside of $W^n_j$ are defined to be the components $\widehat W^{n+q}_{j+q-1}$. Doing this for all $j \leq k$, we obtain all of  the elements of $\widehat \cC_{m+1}$ contained in $\widehat W^n_k$. The covering $\widehat \cC_{m+1}$ is defined to be the collection of all such elements obtained in this way for each $\widehat W^n_k \in \cC_m.$
\end{definition}

Next we need the following technical Lemma. Recall that in Definition \ref{bigR} we defined $n_k = 2^{N+k-1}.$
\begin{lem}
	\label{holesum}
	Fix some $t >0$. For $k \geq 1$, define $L_k = n_1\cdots n_k$, and let $\varepsilon > 0$ be given. Then there exists $M \in \N$ such that for all $N \geq M$, we have
	\begin{equation}
	\label{holesumeq}
	\sum_{k=1}^{\infty} 2^k L_k R^{-t}_k < \varepsilon.
	\end{equation}
\end{lem}
\begin{proof}
	We will use the Ratio test. Let $a_k$ denote the $k$th term of (\ref{holesumeq}). Then for any $k \geq 1$ we have, by applying Lemma \ref{Rkest}, there exists $M$ so that for all $N \geq M$
	\begin{align}
	\label{Ratio_Test_Strong}
	\frac{a_{k+1}}{a_k} &= \frac{2^{k+1}n_1 \cdots n_kn_{k+1}R^t_{k+1}}{2^k n_1 \cdots n_k R^t_{k}} \\
	\nonumber	&= 2n_{k+1} \left(\frac{R_k}{R_{k+1}}\right)^t  \\
	\nonumber	&\leq 2n_{k+1} \left(\frac{1}{4R_k}\right)^t \\
	\nonumber	&\leq \frac{2}{4^t}n_{k+1} 2^{-t2^{k+N-2}}\\
	\nonumber	&\leq 8 n_{k-1} \left(\frac{1}{2^t}\right)^{n_{k-1}}.	
	\end{align}
	Since $n x^{n} \rightarrow 0$ as $n \rightarrow \infty$ whenever $x \in (0,1)$, the series converges by the Ratio test. In fact, a stronger statement is true. By perhaps choosing $M$ larger, we may arrange for the ratio in (\ref{Ratio_Test_Strong}) to be arbitrarily small for all $k \geq 1$. We may also arrange for the first term of (\ref{holesumeq}), $2n_1 R_1^{-t}$ to be arbitrarily small. It follows that we can make (\ref{holesumeq}) arbitrarily small.
\end{proof}

The proof of the Lemma \ref{hard_W} is the same as the proof of $(17.2)$ in \cite{Bis18}. We include the details for the sake of the reader. First, we introduce the following convenient definition.
\begin{definition}
	\label{W^n_j(q)}
	Let $m \geq 0$ and let $W^n_j \in \cC_m$ be given. Fix some value $q \geq 1$. We define $W^n_j(q)$ to be the set of all elements $\widehat W^{n+q}_{j+q-1} \in \widehat \cC_{m+1}$ that are a subset of $W^n_j$. 
\end{definition}

\begin{lem}
	\label{hard_W}
	Fix some $t >0$. Let $W^n_j \in \cC_m$ be given for $m \geq 0$. Then there exists a $M \in \N$ so that for all $N \geq M$ we have
	\begin{equation}\label{what?}
	\sum_{q =1}^{\infty} \sum_{W^n_j(q)} \diam(\widehat W^{n+q}_{j+q -1})^t \leq \frac{1}{100} \diam(W^n_j)^t
	\end{equation}
\end{lem}
\begin{proof}
	Fix an arbitrary element $W^n_j \in \cC_m$, and choose an arbitrary element of the form $\widehat W^{n+q}_{j+q-1} \in \widehat \cC_{m+1}$ contained inside of $W^n_j$ for some $q \geq \max\{1, 2-j\}$. 
	
	First, we observe that the mapping
	$$f_N^n: W^n_j \rightarrow A_j,$$
	is conformal by Lemma \ref{section_8_distortion}. Therefore, by Lemma \ref{section_8_distortion} and Corollary \ref{Kobe} we obtain
	\begin{equation}
	\label{hard_w_eq1}
	\frac{\diam(\widehat{W}^{n+q}_{j+q-1})^t}{\diam(W^n_j)^t} \leq (L'')^t \frac{\diam(f_N^n(\widehat{W}^{n+q}_{j+q-1}))^t}{\diam(A_j)^t}.
	\end{equation}
	Next, we observe that the mapping
	$$f_N^q: f_N^n(\widehat{W}^{n+q}_{j+q-1}) \rightarrow A_{j+q-1}$$
	is conformal. Noting that $j +q-1 \geq 1$, if $B$ is the Jordan domain from Lemma \ref{section_8_distortion} that contains $\widehat W^{n+q}_{j+q-1}$, then $f_N^n(B) \subset A_j$. Let $B'$ be the connected component of $f^{-(n+q)}_N(B(0,2R_{j+q}))$ so that $B' \subset B$ and the modulus of $B \setminus \overline{B'}$ is bounded above by $(2\pi)^{-1} \log(2)$. Therefore by Corollary \ref{Kobe} and Remark \ref{L''} we obtain
	\begin{equation}
	\label{hard_w_eq2}
	\frac{\diam(f_N^n(\widehat{W}^{n+q}_{j+q-1}))^t}{\diam(f_N^n(B'))^t} \leq (L'')^t \frac{\diam(A_{j+q-1})^t}{\diam(f_N^{n+q}(B'))^t} = (L'')^t \frac{\diam(A_{j+q-1})^t}{\diam(B(0,2R_{j+q}))^t}.
	\end{equation}
	Combining (\ref{hard_w_eq1}) and (\ref{hard_w_eq2}), we obtain  
	\begin{align*}
	\diam(\widehat{W}^{n+q}_{j+q-1})^t &\leq (L'')^{2t} \frac{\diam(f_N^n(B))^t}{\diam(B(0,2R_{j+q}))^t} \cdot \frac{\diam(A_{j+q-1})^t}{\diam(A_j)^t} \cdot \diam(W^n_j)^t \\
	&\leq (L'')^{2t} \frac{\diam(A_j)^t}{\diam(B(0,2R_{j+q}))^t} \cdot \frac{\diam(A_{j+q-1})^t}{\diam(A_j)^t } \cdot \diam(W^n_j)^t \\
	&= (2(L'')^2)^t \left(\frac{R_{j+q-1}}{R_{j+q}}\right)^t \cdot \diam(W^n_j)^t \\
	&\leq \left(\frac{(L'')^2}{2}\right)^t  \cdot \frac{1}{R^t_{j+q-1}}\diam(W^n_j)^t.
	\end{align*}
	
	Next, for a fixed value $q \geq 1$, we want to count the total number of components of the form $\widehat W^{n+q}_{j+q-1} \in \widehat \cC_{m+1}$ contained in $W^n_j$. Suppose that $j \geq 1$. First, the mapping $f_N^n: W^n_j \rightarrow A_j$ is a conformal mapping. Next, for each $i = 0,1,\dots, q-1$, we have $f^{n+i}_N(\widehat W^{n+q}_{j+q-1}) \subset A_{j+i}$, and $f_N^{n+q}(\widehat W^{n+q}_{j+q-1}) = A_{j+q-1}$.   
	
	Since there are $n_{j+q-1}$ many petals in $A_{j+q-1}$, there are $n_{j+q-1}$ many connected components of $f_N^{-1}(A_{j+q-1})$ contained inside of $A_{j+q-1}$ by Lemma \ref{widehat help}. For such a connected component $U$ of $f_N^{-1}(A_{j+q-1})$, by repeatedly applying Lemma \ref{component_counter}, there are exactly $n_{j+1}\cdot n_{j+2} \cdots n_{j+q-1} = 2^{q-1} n_j\cdots n_{j+q-2}$ many connected components contained in $A_j$ which map conformally onto $U$. Therefore, the total number of components of the form $\widehat W^{n+q}_{j+q-1} \in \widehat \cC_{m+1}$ contained in $W^n_j$ is bounded above by
	\begin{equation}
	\label{num_comp_j_big}
	(2^{q-1} n_{j} \cdots n_{j+q-2}) \cdot n_{j+q-1} \leq 2^q n_{j} \cdots n_{j+q-1} \leq 2^{q + j - 1} L_{q+j-1}.
	\end{equation}
	When $j \leq 0$, for a fixed value $q \geq 1$, counting the total number of components of the form $\widehat W^{n+q}_{j+q-1} \in \cC_{m+1}$ contained in $W^n_j$ is similar to the $j \geq 1$ case, with just one additional complication. In this case, we must have $q \geq 2-j$, and we have $f_N^n(W^{n+q}_{j+q-1}) \subset U$, where $U = f_N^{-(q-1)}(A_{j+q-1}) \cap A_j'$ and $A_j'$ is the connected component of $A_j$ that contains $f_N^n(W^{n+q}_{j+q-1})$. However, by Lemma \ref{T}, $f_N^{1-j}: U \rightarrow f_N^{1-j}(U)$ is conformal. So by similar reasoning as (\ref{num_comp_j_big}), the number of components of the form $W^{n+q}_{j+q-1}$ contained inside of $W^n_j$ is bounded above by
	\begin{equation}
	\label{num_comp_j_small}
	2^{j+q-2} n_1 \cdots n_{j+q-2} \cdot n_{j+q-1} \leq 2^{j+q-1} L_{j+q-1}.
	\end{equation}
	for each $q \geq 2-j$. Therefore, 
	\begin{align*}
	\sum_{q =1}^{\infty} \sum_{W^n_j(q)} \diam(\widehat W^{n+q}_{j+q -1})^t &\leq \diam(W^n_j)^t \cdot \left(\frac{(L'')^2}{2}\right)^t \sum_{q \geq \max\{1,2-j\}}^{\infty} 2^{j+q-1} L_{j+q-1} R_{j+q-1}^{-t} \\
	&\leq \diam(W^n_j)^t \cdot \left(\frac{(L'')^2}{2}\right)^t\sum_{k=1}^{\infty} 2^{k} L_{k} R_{k}^{-t}
	\end{align*}
	The result now follows by choosing $M$ so large that for all $N \geq M$ 
	\begin{equation*}
	\sum_{k=1}^{\infty}2^{k} L_{k} R_{k}^{-t}   < \frac{1}{100} \left(\frac{2}{(L'')^2}\right)^t.
	\end{equation*}
	Such an $M$ exists by Lemma \ref{holesum}.
\end{proof}

We will now show that the sum of the diameters of every distinct element $W$ of $\widehat \cC_{m+1}$ is comparable to the sum of the diameters of every distinct element $V$ of $\widehat \cC_m$.

\begin{lem}
	\label{compare_layers}
	Fix some $t>0$. Let $m \geq 1$ be given. Then 
	\begin{equation}
	\label{compare_layers_eqn}
	\sum_{W \in \widehat{\cC}_{m+1}} \diam(W)^t \leq \frac{1}{10} \sum_{V \in \widehat{\cC}_m} \diam(V)^t.
	\end{equation}
\end{lem}
\begin{proof}
	Let $\widehat W^n_k \in \widehat \cC_m$. Define
	\begin{equation*}
	G := \{W \in \widehat \cC_{m+1} \,:\, W \subset \widehat W^n_k\}.    
	\end{equation*}
	If $W \in E$, then there exists $j \leq k$ so that $W \subset W^n_j \in \cC_m$ and $W^n_j \subset \widehat W^n_k$. For fixed $j$, define
	\begin{equation*}
	G_j := \{W \in G\,:\,W \subset W^n_j\}.
	\end{equation*}
	It follows from Lemma \ref{hard_W} that 
	\begin{equation*}
	\sum_{W \in G_j} \diam(W)^t \leq \frac{1}{100} \diam(W^n_j)^t.
	\end{equation*}
	Since $G = \cup_{j\leq k} G_j$, we obtain
	\begin{equation}
	\label{compare_layers_eqn_1}
	\sum_{W \in G} \diam(W)^t \leq \frac{1}{100}\sum_{W^n_j \in \cC_m,\, W^n_j \subset \widehat W^n_k} \diam(W^n_j)^t.
	\end{equation}
	By repeatedly applying Lemma \ref{easy_W}, we have for any fixed $j \leq k$ that 
	\begin{equation}
	\label{compare_layers_eqn_2}
	\sum_{W^n_j \subset \widehat W^n_k} \diam(W^n_j)^t \leq \left( \frac{1}{100} \right)^{k-j} \diam(W^n_k)^t
	\end{equation}
	By combining (\ref{compare_layers_eqn_1}) and (\ref{compare_layers_eqn_2}), we deduce
	\begin{equation*}
	\sum_{W \in E} \diam(W)^t \leq \sum_{j \leq k} \left(\frac{1}{100}\right)^{k-j+1} \diam(W^n_k)^t \leq \frac{1}{10}\diam(W^n_k)^t.
	\end{equation*}
	The claim now follows by summing over all $\widehat W^n_k \in \widehat \cC_m$.
\end{proof}

\begin{thm}
	\label{dim_buried_points}
	With the notation as above,
	\begin{equation}
	\label{geometric_sum}
	\sum_{m = 1}^{\infty} \sum_{W \in \widehat \cC_m} \diam(W)^t < \infty.
	\end{equation}
	As a consequence, we have $\dim_H(Y \cap A_1) \leq t$.
\end{thm}
\begin{proof}
	We obtain a geometric sum by Lemma \ref{compare_layers}. Indeed,
	$$\sum_{m = 1}^{\infty} \sum_{W \in \widehat \cC_m} \diam(W)^t \leq \sum_{m=1}^{\infty} \left(\frac{1}{10}\right)^m \diam(A_1)^t < \infty.$$
	Therefore, for every $\varepsilon > 0$, there exists $m \geq 0$ so that 
	$$\sum_{W \in \widehat \cC_m} \diam(W)^t < \varepsilon.$$
	Since $\widehat \cC_m$ covers $Y \cap A_1$ for each $m \geq 0$, by applying (\ref{estimate_sum}), we deduce that $H^{t}(Y \cap A_1) = 0$. It follows immediately that $\dim_H(Y \cap A_1) \leq t.$
\end{proof}

\begin{cor}
	\label{dim_Y_t}
	We have $\dim_H(Y) \leq t$.	
\end{cor}
\begin{proof}
	First, we will observe that our arguments above apply to the case of $A_k \cap Y$ for $k >1$, with simple modifications made to the definitions of $\cC_m$. Therefore, $\dim_H(Y \cap A_k) \leq t$ for all $k \geq 1$. If $k \leq 0$, let $A_k'$ be a connected component of $A_k$. Then by repeatedly applying Lemma \ref{T}, we see that $f^{k+1}_N$ maps $A_k'$ onto $A_1$ conformally. It follows that $\dim_H(Y \cap A_k') = \dim_H(Y \cap A_1)$, and we deduce that $\dim_H(Y \cap A_k) = \dim_H(Y \cap A_1)$ for all $k \leq 0$. 
	
	Since $Y \subset A$, we conclude by (\ref{countable_stability}) that
	$$\dim_H(Y) =\dim_H(Y \cap A) = \sup_{k \in \Z} \dim(A_k \cap Y) \leq t,$$
	as desired.
\end{proof}

\section{Jordan Fatou Boundary Components}
\label{Dimension Z}

Recall that we have proven $\mathcal{J}(f)\subset E' \cup Y \cup Z$, and that $E'$ and $Y$ may be taken to have arbitrarily small (positive) dimension. We now move on to estimating the dimension of $Z$ (those points whose orbits always stay in $A$, and eventually only move forward). It will be necessary to partition $Z$ as follows:

%Recall that the set $Z$ was defined in Definition \ref{xyz}. 

\begin{definition}
	\label{Z1Z2}
	Let 
	\begin{equation}
	\label{Z_1}
	Z_1 := \left\{z\in Z: \textrm{ there exists } l \geq 0 \textrm{ such that for all } j \geq 0,\, f_N^{l+j}(z) \in \bigcup_{k \geq 1} V_k\right\},
	\end{equation}
	and
	\begin{equation}
	\label{Z_2}
	Z_2 := Z \setminus Z_1. 
	\end{equation}
	Note that $Z = Z_1 \sqcup Z_2$.
\end{definition}

\noindent Our primary objective over the next three sections is the proof of the following theorem.

\begin{thm}
	\label{boundary theorem}
	$Z_1$ is the disjoint union of countably many $C^1$ Jordan curves, and $Z_2$ has Hausdorff dimension $0$.
\end{thm}

In this Section, we will focus on proving that $Z_1$ consists of a disjoint union of Jordan curves, and in Section \ref{C1proofsection} we will prove that they are $C^1$. Lastly, in Section \ref{singleton_components} we will study $Z_2$.

For the entirety of the next three sections, we choose $M$ so large so that for all $N \geq M$ and $k \geq 1$ we have for all $z \in V_k$ that 
\begin{equation}
\label{phi_derivative_9}
\frac{1}{2} \leq |\phi'_N(z)| \leq 2.
\end{equation}	
The existence of such an $M$ follows from the Cauchy estimate and Lemma \ref{AkEst}. 

Recall from Definition \ref{central_series} that for $k \geq 1$, $\Omega_k$ is the Fatou component containing $B_k$. We first study the set $Z_1 \cap \overline{\Omega_k}$ for $k \geq 1$. 

\begin{lem}
	\label{angle_change2}
	Let $F: \C \ra \C$ be a holomorphic function, $z \in\mathbb{C}$ and suppose that $F'(z) \neq 0$, $z \neq 0$, and $F(z)\neq0$.
	Let $R_z$, $R_{F(z)}$ denote the rays starting at the origin and passing through $z$, $F(z)$, respectively. Let $v \in T_{z} \C$ denote the outward pointing tangent vector to $R_{z}$ based at $z$ and let $w \in T_{F(z)} \C$ be the image of the outward pointing tangent vector to $R_{F(z)}$ based at $F(z)$. Then the angle between $DF_z(v)$ and $w$ is given by $\arg(\frac{z}{F(z)}F'(z)).$
\end{lem}
\begin{figure}[!htb]
	\centering	
	\scalebox{.9}{%% Creator: Inkscape 1.0 (4035a4fb49, 2020-05-01), www.inkscape.org
%% Accompanies image file 'angle_formula.pdf' (pdf, eps, ps)
\begingroup%
  \makeatletter%
  \providecommand\color[2][]{%
    \errmessage{(Inkscape) Color is used for the text in Inkscape, but the package 'color.sty' is not loaded}%
    \renewcommand\color[2][]{}%
  }%
  \providecommand\transparent[1]{%
    \errmessage{(Inkscape) Transparency is used (non-zero) for the text in Inkscape, but the package 'transparent.sty' is not loaded}%
    \renewcommand\transparent[1]{}%
  }%
  \providecommand\rotatebox[2]{#2}%
  \newcommand*\fsize{\dimexpr\f@size pt\relax}%
  \newcommand*\lineheight[1]{\fontsize{\fsize}{#1\fsize}\selectfont}%
  \ifx\svgwidth\undefined%
    \setlength{\unitlength}{207.87339529bp}%
    \ifx\svgscale\undefined%
      \relax%
    \else%
      \setlength{\unitlength}{\unitlength * \real{\svgscale}}%
    \fi%
  \else%
    \setlength{\unitlength}{\svgwidth}%
  \fi%
  \global\let\svgwidth\undefined%
  \global\let\svgscale\undefined%
  \makeatother%
  \begin{picture}(1,0.60444312)%
    \lineheight{1}%
    \setlength\tabcolsep{0pt}%
    \put(0,0){\includegraphics[width=\unitlength,page=1]{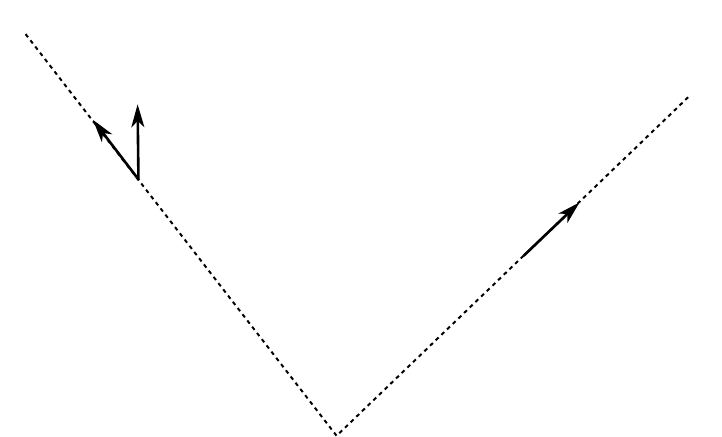}}%
    \put(0.92842781,0.49408244){\color[rgb]{0,0,0}\makebox(0,0)[lt]{\lineheight{1.25}\smash{\begin{tabular}[t]{l}$R_z$\end{tabular}}}}%
    \put(-0.00089586,0.57295929){\color[rgb]{0,0,0}\makebox(0,0)[lt]{\lineheight{1.25}\smash{\begin{tabular}[t]{l}$R_{F(z)}$\end{tabular}}}}%
    \put(0.71454578,0.29501787){\color[rgb]{0,0,0}\makebox(0,0)[lt]{\lineheight{1.25}\smash{\begin{tabular}[t]{l}$v$\end{tabular}}}}%
    \put(0.70312858,0.20407172){\color[rgb]{0,0,0}\makebox(0,0)[lt]{\lineheight{1.25}\smash{\begin{tabular}[t]{l}$z$\end{tabular}}}}%
    \put(0.20946106,0.34630545){\color[rgb]{0,0,0}\makebox(0,0)[lt]{\lineheight{1.25}\smash{\begin{tabular}[t]{l}$F(z)$\end{tabular}}}}%
    \put(0.09909375,0.36894437){\color[rgb]{0,0,0}\makebox(0,0)[lt]{\lineheight{1.25}\smash{\begin{tabular}[t]{l}$w$\end{tabular}}}}%
    \put(0.19747491,0.41160092){\color[rgb]{0,0,0}\makebox(0,0)[lt]{\lineheight{1.25}\smash{\begin{tabular}[t]{l}$DF_z(v)$\end{tabular}}}}%
    \put(0,0){\includegraphics[width=\unitlength,page=2]{angle_formula.pdf}}%
  \end{picture}%
\endgroup%
}
	\caption{A schematic for the statement of Lemma \ref{angle_change2}.}
	\label{angle formula pic}
\end{figure}
\begin{proof}
	First we consider the case that $F(z)=z$. Then, letting $v \in T_z \C$ denote the unit tangent vector pointing in the direction of $R_z$, we have
	$$DF_z(v) = F'(z) v = |F'(z)| \arg(F'(z)) v \in T_z \C.$$
	This proves the result in the case that $w=z$. When $F(z)\not=z$, the result follows from applying the above reasoning to the function $\zeta\mapsto\frac{z}{F(z)}\cdot F(\zeta)$.
\end{proof}

The following Lemma follows a similar strategy to Lemma 18.1 in \cite{Bis18}. 	

\noindent

\begin{lem}
	\label{angledistortion}
	Let $\varepsilon > 0$ and $n,k \in \N$ be given. Suppose that $\phi$ is a univalent function on the annulus $A(\frac{1}{4}R_k,\frac{3}{4}R_k)$ and suppose that $|\phi(z)/z - 1| < \varepsilon$ on $A(\frac{2}{5}R_k,\frac{3}{5}R_k)$. Define 
	$$F(z) = (\phi(z))^n.$$
	For any fixed $\tau \in [0,2\pi)$, parameterize the segment $S(\tau) = \{r e^{i\tau} \, : \, \frac{2}{5}R_k \leq r \leq \frac{3}{5}R_k\}$ as $\gamma_{\tau}(r) = re^{i\tau}$, $r \in [\frac{2}{5}R_k,\frac{3}{5}R_k]$. Suppose that $F \circ \gamma_{\tau}$ and $\gamma_{\varphi}$ intersect at some point $z$. Then the angle between the tangent vectors of $F \circ \gamma_{\tau}$ and $\gamma_{\varphi}$ based at $F(z)$ is $O(\varepsilon)$ as $\varepsilon \rightarrow 0$. 
\end{lem}
\begin{proof}
	Following Lemma \ref{angle_change2}, it is sufficient to estimate $\arg(z F'(z)/F(z))$. To that end, first observe that by the chain rule we have
	\begin{align}
	z \frac{F'(z)}{F(z)} = z \frac{n (\phi(z))^{n-1} \phi'(z)}{(\phi(z))^n} 
	= \frac{z}{\phi(z)} \cdot n \phi'(z).
	\end{align}
	Let $g(\zeta) = \phi(\zeta) - \zeta$. Then $g'(\zeta) = \phi'(\zeta) - 1$. If $z \in A(\frac{2}{5}R_k,\frac{3}{5}R_k)$, then $B(z,\frac{1}{10}R_k) \subset A(\frac{1}{4}R_k,\frac{3}{4}R_k)$, so that Cauchy estimates say we must have
	\begin{align*}
	|g'(z)| &\leq \frac{\max_{B(z,\frac{1}{10}R_k)} |g(\zeta)| }{\frac{1}{10}R_k}	\\
	&= \frac{10}{R_k} \cdot \max_{B(z,\frac{1}{10}R_k)} |\phi(\zeta) -\zeta| \\
	& = \frac{10}{R_k} \cdot  \max_{B(z,\frac{1}{10}R_k)}|\zeta|\cdot  |\phi(\zeta)/\zeta - 1| \\
	&\leq 10 \cdot \varepsilon
	\end{align*}
	It follows that for all $z \in A(\frac{2}{5}R_k,\frac{3}{5}R_k)$, we have $|\phi'(z) -1| < 10 \varepsilon$, so that $\phi'(z) \in B(1, 10\varepsilon).$ This means that $n\phi'(z) \in B(n,10 n\varepsilon).$ By assumption, we have $\frac{\phi(z)}{z} \in B(1,\varepsilon)$. Therefore if $\varepsilon < \frac{1}{2}$ we have $\frac{z}{\phi(z)} \in B(1,2\varepsilon)$.
	
	\begin{figure}[!htb]
		\centering	
		\scalebox{.9}{%% Creator: Inkscape 1.0 (4035a4fb49, 2020-05-01), www.inkscape.org
%% Accompanies image file 'small_angle.pdf' (pdf, eps, ps)
\begingroup%
  \makeatletter%
  \providecommand\color[2][]{%
    \errmessage{(Inkscape) Color is used for the text in Inkscape, but the package 'color.sty' is not loaded}%
    \renewcommand\color[2][]{}%
  }%
  \providecommand\transparent[1]{%
    \errmessage{(Inkscape) Transparency is used (non-zero) for the text in Inkscape, but the package 'transparent.sty' is not loaded}%
    \renewcommand\transparent[1]{}%
  }%
  \providecommand\rotatebox[2]{#2}%
  \newcommand*\fsize{\dimexpr\f@size pt\relax}%
  \newcommand*\lineheight[1]{\fontsize{\fsize}{#1\fsize}\selectfont}%
  \ifx\svgwidth\undefined%
    \setlength{\unitlength}{296.70390199bp}%
    \ifx\svgscale\undefined%
      \relax%
    \else%
      \setlength{\unitlength}{\unitlength * \real{\svgscale}}%
    \fi%
  \else%
    \setlength{\unitlength}{\svgwidth}%
  \fi%
  \global\let\svgwidth\undefined%
  \global\let\svgscale\undefined%
  \makeatother%
  \begin{picture}(1,0.49291566)%
    \lineheight{1}%
    \setlength\tabcolsep{0pt}%
    \put(0,0){\includegraphics[width=\unitlength,page=1]{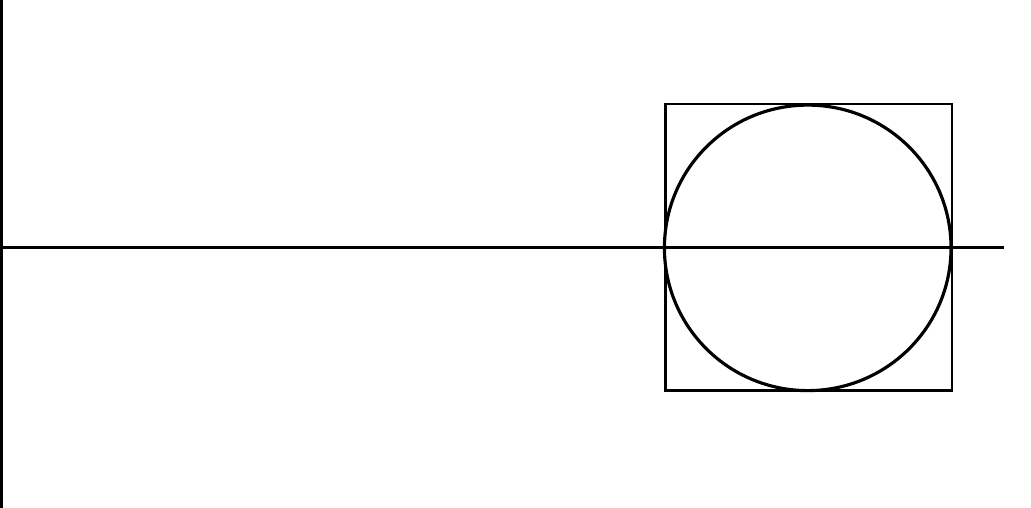}}%
    \put(0.76373551,0.21188332){\color[rgb]{0,0,0}\makebox(0,0)[lt]{\lineheight{1.25}\smash{\begin{tabular}[t]{l}$n$\end{tabular}}}}%
    \put(0.199504,0.25875501){\color[rgb]{0,0,0}\makebox(0,0)[lt]{\lineheight{1.25}\smash{\begin{tabular}[t]{l}$\theta$\end{tabular}}}}%
    \put(0.69456975,0.15940558){\color[rgb]{0,0,0}\makebox(0,0)[lt]{\lineheight{1.25}\smash{\begin{tabular}[t]{l}$B(n,24n\epsilon)$\end{tabular}}}}%
    \put(0,0){\includegraphics[width=\unitlength,page=2]{small_angle.pdf}}%
  \end{picture}%
\endgroup%
}
		\caption{A schematic for the proof of Lemma \ref{angledistortion}.}
		\label{angledistortion_pic}
	\end{figure}

	Putting everything together, we have
	\begin{equation}
	\label{multiplying_balls}
	\frac{z}{\phi(z)} \cdot n \phi'(z) \in B(n,24n\varepsilon).
	\end{equation}
	Indeed, if $a \in B(1,2\varepsilon)$ and $b \in B(n, 10n\varepsilon)$ we have
	\begin{align*}
	|ab-n| &= |a(b-n) + an - n| \leq |a| \cdot |b-n| + n \cdot |a-1| \\
	&\leq (1+2\varepsilon) 10 n \varepsilon + 2\varepsilon n \\
	&= 12n\varepsilon + 20n \varepsilon^2 \\
	&= 12n\varepsilon (1+ \frac{5}{3} \varepsilon) \\
	&< 24n \varepsilon 
	\end{align*}
	whenever $\varepsilon < 3/5$. 
	Therefore, for all $\varepsilon$ sufficiently small we have
	$$\arg\left(z \frac{F'(z)}{F(z)}\right) \leq \arctan(24 \cdot 2 \varepsilon) = O(\varepsilon) \textrm{ as } \varepsilon \rightarrow 0.$$
	This proves the claim; see Figure \ref{angledistortion_pic}.	
\end{proof}

\begin{definition}\label{gamma_definition}	
	Let $\Omega_k$ be given for some $k \geq 1$.  Then the outermost boundary component of $\Omega_k$ is contained in $V_{k+1}$ by Theorem \ref{Julia_place}. Define
	\begin{equation}
	\label{Gammakn}
	\Gamma_{k,n} := \{z \in V_{k+1}\,:\, f_N^j(z) \in V_{k+j+1}, j = 0, 1,\dots,n\}.
	\end{equation}

	By Lemma \ref{Ak} and Lemma \ref{CoveringMapVk}, for $n \geq 1$ each $\Gamma_{k,n}$ is a topological annulus compactly contained inside of $\Gamma_{k, n-1}$, and $\Gamma_{k,1}$ is compactly contained inside of $\Gamma_{k,0} = V_{k+1}$. We define
	\begin{equation}
	\label{Gammak}
	\Gamma_k := \bigcap_{n=1}^{\infty} \Gamma_{k,n}.
	\end{equation}	
\end{definition}

 The remainder of this Section will be devoted to showing that $\Gamma_k$ is in fact a Jordan curve, and in the next Section we will show that it is $C^1$. 

\begin{definition}
	\label{foliation}
	Fix some $k \geq 1$ and let $n \geq 0$ be arbitrary. Each $V_{k+n+1}$ has a foliation of closed circles centered around the origin, including the inner and outer boundary of $V_{k+n+1}$. When $n = 0$ this is a foliation of $\Gamma_{k,0}$ which we denote by $\mathcal{U}_{k,0}$. When $n \geq 1$, by pulling this foliation back to $\Gamma_{k,n}$ by $f^{n}_N$ we obtain a foliation $\mathcal{U}_{k,n}$ of $\Gamma_{k,n}$ by analytic Jordan curves by Lemma \ref{CoveringMapVk}. 
\end{definition}

\begin{rem}
	Let $n \geq 1$. It is readily verified from (\ref{Gammakn}) that $f_N(\Gamma_{k,n}) = \Gamma_{k+1,n-1}$. Similarly, we can verify using Definition \ref{foliation} that if $\gamma \in \mathcal{U}_{k,n}$, then $f_N(\gamma) \in \mathcal{U}_{k+1,n-1}$.
\end{rem}

\begin{lem}
	\label{key_angle}
	Let $k \geq 1$, and suppose that $\gamma_n \in \mathcal{U}_{k,n}$ and $\gamma_m \in \mathcal{U}_{k,m}$ for $m > n \geq 0$. Suppose that $\gamma_n$ and $\gamma_m$ intersect at some point $z$. Let $\tau_n(z)$ and $\tau_m(z)$ denote the counter-clockwise oriented unit tangent vectors of $\gamma_n$ and $\gamma_m$ at $z$. Likewise, let $\nu_n(z)$ and $\nu_m(z)$ denote the outward pointing normal vectors of $\gamma_n$ and $\gamma_m$ at $z$. Then 
	\begin{equation}
	\label{key_angle_eq}
	|\tau_n(z) - \tau_m(z)| = |\nu_n(z) - \nu_m(z)| \leq O \left( \sum_{l=n}^{m-1} 2^{-\frac{\sqrt{N+k +l}}{4}} \right).
	\end{equation}
\end{lem}

\begin{figure}[!htb]
	\centering	
	\scalebox{.9}{\input{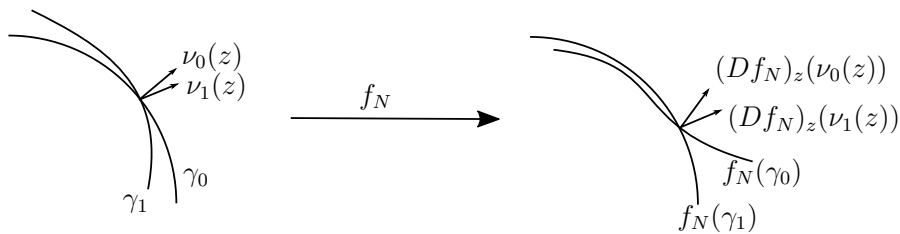}}
	\caption{A schematic for the proof of Lemma \ref{key_angle} in the key step of $m =1$ and $n=0$. In this case, $\gamma_0$ and $f_N(\gamma_1)$ are circles, and $(Df_N):T_z\C \rightarrow T_{f_N(z)}\C$ maps the outward pointing normal $\nu_1(z)$ to $\gamma_1$ to an outward pointing normal $(Df_N)_z(\nu_1(z))$ of the circle $f_N(\gamma_1)$. This allows us to apply Lemma \ref{angledistortion} to estimate the angle between $(Df_N)_z(\nu_1(z))$ and $(Df_N)_z(\nu_0(z))$, which coincides with the angle between $\nu_1(z)$ and $\nu_0(z)$.}
	\label{key_angle_pic}
\end{figure}

\begin{proof}
	We first consider the case $m =1$ and $n = 0$. In this case, $\gamma_0$ is a circle, and $f_N(\gamma_1)$ is a circle in $V_{k+2}$. Let $z$ be a point of intersection of $\gamma_0$ and $\gamma_1$, so that $\nu_0(z)$ and $\nu_1(z)$ are the corresponding outward pointing normal vectors based at $z$. Let $R$ be the ray through the origin passing through $z$, let $R'$ be the ray passing through $f_N(z)$. Since $f_N(\gamma_1)$ is a circle, $Df_N: T_z \C \rightarrow T_{f_N(z)}\C$ maps $\nu_1(z)$ to the outward pointing normal of a circle based at $f_N(z)$. Therefore, $\nu_0(z)$ coincides with the tangent vector to $R$ based at $z$, $(Df_N)_z(\nu_0)$ coincides with the tangent vector to $f_N(R)$ based at $f_N(z)$, and $(Df_N)_z(\nu_1)$ coincides with the tangent vector to $R'$ based at $f_N(z)$. Therefore, by Lemma \ref{angledistortion}, the angle between $(Df_N)_z(\nu_0)$ and $(Df_N)_z(\nu_1)$ is $O(2^{-\sqrt{k+N}})$. Since $f_N$ is conformal at $z$, we deduce that 
	\begin{equation*}
	|\tau_1(z) - \tau_0(z)| = |\nu_1(z) - \nu_0(z)| \leq O(2^{-\sqrt{k+N}}).
	\end{equation*}
	
	Next we consider the case of $m >1$ and $n = m-1$. By Lemma \ref{CoveringMapVk}, the angle between $\gamma_{m}$ and $\gamma_{m-1}$ at $z$ is the same as the angle between $f_N^{m-1}(\gamma_m)$ and $f_N^{m-1}(\gamma_{m-1})$ at the point $f_N^{m-1}(z)$. We also have that $f_N^{m-1}(\gamma_m) \in \mathcal{U}_{k+m-1,1}$ and $f_N^{m-1}(\gamma_{m-1}) \in \mathcal{U}_{k+m-1,0}$, so that $f_N^{m-1}(\gamma_{m-1})$ is a circle. Therefore, we have by Lemma \ref{angledistortion} that 
	$$|\tau_m(z) - \tau_{m-1}(z)| \leq O(2^{-\frac{\sqrt{N+k+m-1}}{4}}).$$	
	The Lemma now follows by applying the triangle inequality. Let $m > n$. Then
	$$|\tau_m(z) - \tau_n(z)| \leq \sum_{l=n}^{m-1} |\tau_{l}(z) - \tau_{l+1}(z)| \leq O\left(\sum_{l=n}^{m-1} 2^{-\frac{\sqrt{N+k +l}}{4}} \right).$$
	This proves the claim.
\end{proof}

\begin{cor}
	\label{foliation_almost_circles}
	Let $k \geq 1$, $n \geq 0$, and let $\gamma$ be any element of $\mathcal{U}_{k,n}$. Then there exists $M \in \N$ so that if $N \geq M$, $\gamma \cap \{re^{i\theta}\,: 0<r<\infty\}$ is a single point for any $\theta$.
\end{cor}
\begin{figure}[!htb]
	\centering	
	\scalebox{.9}{%% Creator: Inkscape 1.0 (4035a4fb49, 2020-05-01), www.inkscape.org
%% Accompanies image file 'two_point_intersection.pdf' (pdf, eps, ps)
\begingroup%
  \makeatletter%
  \providecommand\color[2][]{%
    \errmessage{(Inkscape) Color is used for the text in Inkscape, but the package 'color.sty' is not loaded}%
    \renewcommand\color[2][]{}%
  }%
  \providecommand\transparent[1]{%
    \errmessage{(Inkscape) Transparency is used (non-zero) for the text in Inkscape, but the package 'transparent.sty' is not loaded}%
    \renewcommand\transparent[1]{}%
  }%
  \providecommand\rotatebox[2]{#2}%
  \newcommand*\fsize{\dimexpr\f@size pt\relax}%
  \newcommand*\lineheight[1]{\fontsize{\fsize}{#1\fsize}\selectfont}%
  \ifx\svgwidth\undefined%
    \setlength{\unitlength}{158.57358184bp}%
    \ifx\svgscale\undefined%
      \relax%
    \else%
      \setlength{\unitlength}{\unitlength * \real{\svgscale}}%
    \fi%
  \else%
    \setlength{\unitlength}{\svgwidth}%
  \fi%
  \global\let\svgwidth\undefined%
  \global\let\svgscale\undefined%
  \makeatother%
  \begin{picture}(1,0.68727099)%
    \lineheight{1}%
    \setlength\tabcolsep{0pt}%
    \put(0,0){\includegraphics[width=\unitlength,page=1]{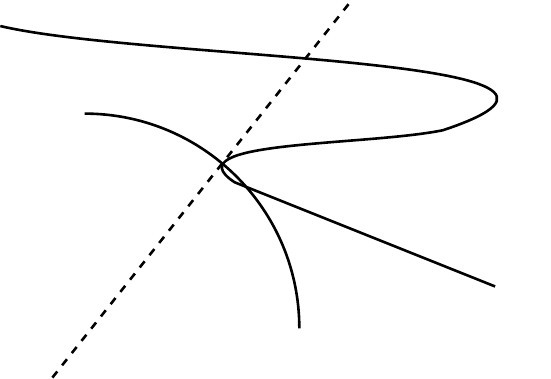}}%
    \put(0.4992787,0.0258948){\color[rgb]{0,0,0}\makebox(0,0)[lt]{\lineheight{1.25}\smash{\begin{tabular}[t]{l}$C$\\\end{tabular}}}}%
    \put(0.91327633,0.16141884){\color[rgb]{0,0,0}\makebox(0,0)[lt]{\lineheight{1.25}\smash{\begin{tabular}[t]{l}$\gamma$\end{tabular}}}}%
    \put(0,0){\includegraphics[width=\unitlength,page=2]{two_point_intersection.pdf}}%
    \put(0.50523282,0.46248051){\color[rgb]{0,0,0}\makebox(0,0)[lt]{\lineheight{1.25}\smash{\begin{tabular}[t]{l}$\nu_C(z)$\end{tabular}}}}%
    \put(0.29254319,0.51121009){\color[rgb]{0,0,0}\makebox(0,0)[lt]{\lineheight{1.25}\smash{\begin{tabular}[t]{l}$\nu_{\gamma}(z)$\end{tabular}}}}%
  \end{picture}%
\endgroup%
}
	\caption{A schematic for the proof of Corollary \ref{foliation_almost_circles}. If some ray $R$ passed through $\gamma$ at more than one point, there is a ray $R'$ tangent to $\gamma$ at some other point $z$. If $C$ is the circle centered at the origin passing through $z$, the normal vectors $\nu_{\gamma}(z)$ and $\nu_C(z)$ make an angle of $\pi/2$ with each other.}
	\label{two point intersection pic}
\end{figure}
\begin{proof}
	Suppose that some ray $R$ through the origin intersected $\gamma$ more than once. Since $\gamma$ is an analytic Jordan curve, this implies that there exists a point $z$ on $\gamma$ and a ray $R'$ such that $R'$ passes through $z$ and is tangent to $\gamma$. This in turn implies that the circle passing through $z$ centered at the origin makes an angle of $\pi/2$ with $\gamma$. For all $N$ sufficiently large, this is a contradiction to Lemma \ref{key_angle}.
\end{proof}

\begin{lem}
	\label{subannulus_Vk}
	Suppose that $k \geq 2$, $z \in V_k$, and $f_N(z) \in V_{k+1}$. Then 
	$$\phi_N^{-1}(z) \in A\left(\frac{1}{2} \left(\frac{1}{4}\right)^{\frac{1}{n_k}}R_k, \frac{1}{2} \left(\frac{3}{4}\right)^{\frac{1}{n_k}} R_k\right).$$ 
\end{lem}
\begin{proof}
	Recall that by Lemma \ref{PrePowerMap} for $z \in V_k$ we have $f_N = h_N \circ \phi_N^{-1}(z) = C_k(\phi_N^{-1}(z))^{n_k}.$ Suppose for the sake of a contradiction that $|\phi_N^{-1}(z)| \leq \frac{1}{2} (\frac{1}{4})^{\frac{1}{n_k}}R_k$. Then 
	\begin{align*}
	|f_N(z)| &= C_k(|\phi_N^{-1}(z)|)^{n_k} 
	\leq C_k\left(\frac{1}{2} \left(\frac{1}{4}\right)^{\frac{1}{n_k}}R_k\right)^{n_k} 
	= \frac{1}{4} R_{k+1} 
	< \frac{2}{5} R_{k+1}.
	\end{align*}
	Since we have $f_N(z) \in V_{k+1}$, we have a contradiction and deduce that we must have $|\phi_N^{-1}(z)| > \frac{1}{2} (\frac{1}{4})^{\frac{1}{n_k}}R_k$ 
	
	Similarly, suppose for the sake of a contradiction that $|\phi_N^{-1}(z)| \geq \frac{1}{2} (\frac{3}{4})^{\frac{1}{n_k}}R_k$. Then
	\begin{align*}
	|f_N(z)| \geq \frac{3}{4}R_{k+1} > \frac{3}{5} R_{k+1}.
	\end{align*}
	Since we have $f_N(z) \in V_{k+1}$, we have a contradiction and deduce that we must have $|\phi_N^{-1}(z)| < \frac{1}{2} (\frac{3}{4})^{\frac{1}{n_k}}R_k$. The claim follows.  
\end{proof}

\begin{lem}
	\label{foliation_derivative_estimate}
	Let $k, n\geq1$ and suppose that $z \in \Gamma_{k,n}$. Then 
	\begin{equation}
	\label{foliation_derivative_estimate_eqn}
	|f_N'(z)| \geq \frac{1}{4} n_{k+1} \frac{R_{k+2}}{R_{k+1}}
	\end{equation}
\end{lem}
\begin{proof}
	By the chain rule we have \[ f_N'(z) = n_{k+1} C_{k+1}(\phi_N^{-1}(z))^{n_{k+1}-1}(\phi_N^{-1})'(z).  \] Since $z \in \Gamma_{k,n}$ and $k\geq1$, we have $f_N(z) \in V_{k+2}$. Therefore by Lemma \ref{subannulus_Vk}, we must have \[\phi_N^{-1}(z) \in A\left(\frac{1}{2} \left(\frac{1}{4}\right)^{\frac{1}{n_{k+1}}}R_{k+1}, \frac{1}{2} \left(\frac{3}{4}\right)^{\frac{1}{n_{k+1}}} R_{k+1}\right). \] Therefore, by (\ref{phi_derivative_9}), we have
	\begin{align*}
	|f'_N(\phi_N^{-1})(z)| &\geq n_{k+1} C_{k+1}(|\phi_N^{-1}|)^{n_{k+1}-1} |(\phi_N^{-1})'(z)| \\
	&\geq \frac{1}{2} n_{k+1} C_{k+1}\left(\frac{1}{2} \left(\frac{1}{4}\right)^{\frac{1}{n_{k+1}}} R_{k+1}\right)^{n_{k+1}-1} \\
	&= \frac{1}{2} n_{k+1} \left(\frac{1}{4}\right)^{\frac{n_{k+1}-1}{n_{k+1}}} \frac{1}{\frac{1}{2}R_{k+1}} R_{k+2} \\
	&\geq \frac{2}{2\cdot 4} n_{k+1} \frac{R_{k+2}}{R_{k+1}} \\
	&= \frac{1}{4} n_{k+1} \frac{R_{k+2}}{R_{k+1}}
	\end{align*}
	This is precisely what we wanted to show.
\end{proof}

\begin{rem}
	It follows from Corollary \ref{foliation_almost_circles} that if $\Gamma$ is an element of $\mathcal{U}_{k,n}$, then we can parameterize $\Gamma$  as 
	$$\gamma: [0,2\pi] \rightarrow \Gamma$$
	$$\gamma(\theta) = r(\theta)\cdot e^{i\theta}$$ for some $\mathbb{R}^+$-valued function $r(\theta)$ on $[0,2\pi]$. Equivalently, we have
	$$\gamma(\theta) = (r(\theta)\cos(\theta), r(\theta)\sin(\theta)).$$
\end{rem}

\begin{definition}
	\label{width}
	Fix $k \geq 1$ and $m \geq 1$. Let $R_{\theta}$ be the ray starting at the origin with angle $\theta$. Define
	\begin{equation}
	\label{width_eqn}
	w_{k,m}(\theta) = \length(R_{\theta} \cap \Gamma_{k,m}),
	\end{equation}
	and,
	\begin{equation}
	\label{width_eqn_sup}
	w_{k,m} = \sup_{\theta \in [0,2\pi)} w_{k,m}(\theta)
	\end{equation}
\end{definition}

\noindent
\begin{lem}
	Fix $k \geq 1$. Then 
	$$w_{k,m} \leq \frac{8^{m-1}}{n_{k+1} \cdots n_{k+m}} R_{k+1}.$$
	In particular, $w_{k,m} \rightarrow 0$ as $m \rightarrow \infty$.
\end{lem}
\begin{proof}
	Fix $k \geq 1$. Note that by Lemma \ref{subannulus_Vk}, we have 
	\begin{align*}
	w_{k,1} \leq \frac{1}{2}R_{k+1}\left( \left(\frac{3}{4}\right)^{\frac{1}{n_{k+1}}} - \left(\frac{1}{4}\right)^{\frac{1}{n_{k+1}}}\right) \leq \frac{1}{2}R_{k+1} \frac{2}{n_{k+1}} = \frac{R_{k+1}}{n_{k+1}},
	\end{align*}
	where for the second inequality we have used the easily verified fact that $(3/4)^x-(1/4)^x\leq 2x$ for all sufficiently small $x>0$.
	
	Therefore, for all $k \geq 1$, we have 
	\begin{equation}
	w_{k,1} \leq \frac{R_{k+1}}{n_{k+1}}
	\end{equation}
	
	Next, fix some $m > 1$, and define $S_{\theta} = R_{\theta} \cap \Gamma_{k,m}.$ Then $f_N(S_{\theta})$ is a curve in $\Gamma_{k+1,m-1}$ with one endpoint on the inner boundary of $\Gamma_{k+1,m-1}$ and the other endpoint on the outer boundary of $\Gamma_{k+1,m-1}$. Then by Lemma \ref{angledistortion}, the angle between $f_N(S_{\theta})$ and any radial segment it meets is $O(2^{-\frac{\sqrt{N+k}}{4}}).$ We will now show that this implies that the length of $f_N(S_{\theta})$ is bounded above by $2w_{k+1,m-1}$.
	
	Indeed, first observe that by Lemma \ref{angledistortion}, $f_N(S_{\theta})$ intersects any circle centered at $0$ at most once. Thus, we may parameterize $f_N(S_{\theta})$ as $\gamma(r) = r\exp(i\theta(r))$ for $r \in [r_1,r_2]$ with $r_2 - r_1 \leq w_{k+1,m-1}$ and some $[0,2\pi)$-valued function $\theta(r)$. Suppose that the radial arc \[ \sigma(r) := r\exp(i\theta_0)\textrm{, } r \in [r_1,r_2] \] intersects $f_N(S_{\theta})$ at some point $z_0 = r_0e^{i\theta_0}$. Then the angle $\varphi$ between the tangent vectors of $\sigma$ and $f_N(S_{\theta})$ at the point $z_0$ is given by the usual dot-product formula
	$$\cos(\varphi) = \frac{ \textrm{Re}(\sigma'(r_0))\textrm{Re}(\gamma'(r_0)) +  \textrm{Im}(\sigma'(r_0))\textrm{Im}(\gamma'(r_0))}{||\sigma'(r_0)|| \cdot ||\gamma'(r_0)||} = \frac{1}{\sqrt{1 + \left(r_0\theta'(r_0)\right)^2}}.$$
	Recall that $\varphi = O(2^{-\sqrt{N+k}})$ by Lemma \ref{angledistortion}. Thus, for all sufficiently large $N$, we have that $\cos(\varphi) \in [0.9,1]$.  It follows that $|\theta'(r_0)| \leq \frac{1}{r_0}.$ The above reasoning holds for all $r_0\in[r_1,r_2]$, and so it follows that:
	\begin{equation*}
	\length(f_N(S_{\theta})) = \int_{r_1}^{r_2} ||\gamma'(r)|| dr = \int_{r_1}^{r_2}  \sqrt{1 + \left(r\theta'(r)\right)^2}dr  \leq \sqrt{2}(r_2 - r_1) < 2 w_{k+1,m-1}.
	\end{equation*}
	On the other hand, we can establish a lower bound for the length of $f_N(S_{\theta})$ using Lemma \ref{foliation_derivative_estimate}. Indeed, we have:
	\begin{equation*}
	\length(f_N(S_{\theta})) = \int_{S_{\theta}} |f'_N(z)| |dz| \geq w_{k,m}(\theta) \cdot \frac{n_{k+1}}{4}\frac{R_{k+2}}{R_{k+1}}.
	\end{equation*}
	Therefore, we have
	\begin{equation*}
	w_{k,m} \leq 2 w_{k+1,m-1} \frac{4}{n_{k+1}} \frac{R_{k+1}}{R_{k+2}} = \frac{8 w_{k+1,m-1} }{n_{k+1}} \frac{R_{k+1}}{R_{k+2}}.
	\end{equation*}
	Therefore, we have for all $k \geq 1$ and all $m \geq 1$ that 
	\begin{equation*}
	w_{k,m} \leq \frac{8^{m-1}}{n_{k+1}\cdots n_{k+m}} w_{k+m-1,1} \frac{R_{k+1}}{R_{k+m}} \leq \frac{8^{m-1}}{n_{k+1}\cdots n_{k+m-1}} \frac{R_{k+1}}{R_{k+m}}\frac{R_{k+m}}{n_{k+m}} = \frac{8^{m-1}}{n_{k+1}\cdots n_{k+m}} R_{k+1}
	\end{equation*}
	This is what we wanted to show, and it also follows that $w_{k,m} \rightarrow \infty$ as $m \rightarrow \infty$, as desired.
\end{proof}

\begin{thm}\label{jordancurvetheorem}
	For each $k \geq 1$, $\Gamma_k$ is a Jordan curve. Furthermore, $\Gamma_k$ intersects any ray $\{z: \emph{arg}(z)=\theta\}$ in exactly one point. 
\end{thm}	
\begin{proof}
	Fix $k\geq1$. By Corollary \ref{foliation_almost_circles}, there exist Jordan-curve parameterizations of the form \begin{align}\label{ccparametrization} \gamma_n^{\textrm{in}}(\theta) = r_n^{\textrm{in}}(\theta)e^{i\theta}\textrm{,    } \theta \in [0,2\pi] \\ \gamma_n^{\textrm{out}}(\theta) = r_n^{\textrm{out}}(\theta)e^{i\theta}\textrm{,    } \theta \in [0,2\pi] \end{align} of the inner and outer boundaries (respectively) of $\Gamma_{k,n}$. Let $m\geq n$. Then, by Lemma \ref{cauchy_estimate} and since $\Gamma_{k,m} \subset \Gamma_{k,n}$, we have the estimate
	\begin{equation}\label{innerouter} |r_m^{\textrm{in}}(\theta) - r_n^{\textrm{in}}(\theta)| = r_m^{\textrm{in}}(\theta) - r_n^{\textrm{in}}(\theta) \leq r_n^{\textrm{out}}(\theta) - r_n^{\textrm{in}}(\theta) = w_{k,n} \xrightarrow{n\rightarrow\infty}0. \end{equation} By (\ref{innerouter}) we can conclude that $\gamma_n^{\textrm{in}}$ has a continuous limit: \[ \gamma^{\textrm{in}}(\theta):=r^{\textrm{in}}(\theta)e^{i\theta}\textrm{,    } \theta \in [0,2\pi]. \] Similar reasoning allows us to conclude that $\gamma_n^{\textrm{out}}$ has a continuous limit: \[ \gamma^{\textrm{out}}(\theta):=r^{\textrm{out}}(\theta)e^{i\theta}\textrm{,    } \theta \in [0,2\pi], \] and moreover by (\ref{innerouter}), \begin{equation}\label{innerouter=}  \gamma^{\textrm{out}}([0,2\pi])=\gamma^{\textrm{in}}([0,2\pi]). \end{equation} Note that $\gamma^{\textrm{in}}([0,2\pi])$ is a Jordan curve since $r^{\textrm{in}}$ is continuous. Furthermore, $\gamma^{\textrm{in}}([0,2\pi]) \subset \Gamma_k$ since $r^{\textrm{in}}(\theta)\geq r_n^{\textrm{in}}(\theta)$ for all $n$ and $\theta$. Moreover, we must have $\Gamma_k \subset \gamma^{\textrm{in}}([0,2\pi])$ by (\ref{innerouter=}). Thus $\Gamma_k=\gamma^{\textrm{in}}([0,2\pi]) $ is a Jordan curve.

	%	 if we suppose by way of contradiction that $\Gamma_k \subsetneq \gamma^{\textrm{in}}([0,2\pi])$, then we would have 

	%	It follows by Lemma \ref{cauchy_estimate} that $r_n(\theta)$ is uniformly Cauchy, and hence $r_n(\theta)$ converges uniformly to some continuous function $r(\theta)$. Define \[ \gamma(\theta):=r(\theta)e^{i\theta}\textrm{,    } \theta \in [0,2\pi]. \] Observe $\gamma([0,2\pi])$ is a Jordan curve since $r(\theta)$ is continuous. We claim that $\gamma([0,2\pi])=\Gamma_k$.  First observe that $\gamma([0,2\pi]) \subset \Gamma_k$ since $r(\theta)\geq r_n(\theta)$ for all $n$ and $\theta$.
	
	%It remains to show that $\Gamma_k \subset \gamma([0,2\pi])$. Suppose by way of contradiction that  $\Gamma_k \subsetneq \gamma([0,2\pi])$. By completely analogous considerations, we may consider parameterizations $\tilde{\gamma}_n(\theta)=\tilde{r}_n(\theta)e^{i\theta}$ of the outer boundaries of $\Gamma_{k,n}$ for all $n$ and prove that there is a limit $\tilde{\gamma}(\theta)=\tilde{r}(\theta)e^{i\theta}$. If $\Gamma_k \subsetneq \gamma([0,2\pi])$, then it must be the case that $\tilde{\gamma}[0,2\pi]\not=$ 

	%	\vspace{5mm} 
	
	%	so that $\gamma_n(\theta)$ converges uniformly to some continuous and injective function $\gamma(\theta)$ which satisfies $\gamma(0) = \gamma(2\pi)$. This limit function $\gamma(\theta)$ gives a parameterization of $\Gamma_k$ as a Jordan curve.
\end{proof}

\section{Smooth Fatou Boundary Components}\label{C1proofsection}

In this Section, we continue our study of the set $Z_1$. We will first prove that each Jordan curve $\Gamma_k$ is in fact a $C^1$ curve (see Theorem \ref{C1} below). Then we will conclude that the set $Z_1$ is a disjoint union of $C^1$ curves, and in particular has dimension $1$. We begin with a precise definition of a $C^1$ curve:

\begin{definition} We say that a Jordan curve $\Gamma$ is $C^1$ if there exists a $C^1$ parametrization $\gamma: [0,2\pi]\rightarrow\mathbb{C}$ of the curve $\Gamma$ satisfying $\gamma'(\theta)\not=0$ for all $\theta\in[0,2\pi]$.
\end{definition}

\begin{thm}
	\label{C1}
	For every $k \geq 1$, $\Gamma_k$ is $C^1$. 
\end{thm}

\noindent We will consider the case $k=1$ to simplify notation, and we will fix a point $z_0\in\Gamma_1$ throughout this Section. We will sometimes omit the subscript $N$ from $f_N$ and $\phi_N$ and simply write $f$ or $\phi$.

\begin{definition} For $m\geq1$, let $s_m$ be such that the circle $|z|=s_m$ passes through $f^m(z_0)\in V_{2+m}$ (see Figure \ref{pullbackcurves}), and define \begin{equation} \gamma_m^m(\theta):=s_m\exp(i\theta) \textrm{ for } \theta\in[0,2\pi]. \end{equation} For $0\leq k<m$, define \begin{equation}\label{parametrizations} \gamma_k^m(\theta):= \phi\left(\hspace{-2mm}\sqrt[\leftroot{-1}\uproot{7}n_{k+2}]{\frac{\gamma_{k+1}^m(n_{k+2}\cdot\theta)}{C_{k+2}}}\right) \textrm{ for } \theta\in[0,2\pi],\end{equation} where the branch of $\sqrt[\leftroot{-0}\uproot{7}n_{k+2}]{\cdot}$ chosen depends on $\theta$ and is such that (\ref{parametrizations}) defines a parametrization of a Jordan curve surrounding $0$.
\end{definition}

	\begin{figure}[!h]
		\centering
		\scalebox{.4}{%% Creator: Inkscape 1.0.1 (c497b03c, 2020-09-10), www.inkscape.org
%% Accompanies image file 'pullbackcurves.pdf' (pdf, eps, ps)
\begingroup%
  \makeatletter%
  \providecommand\color[2][]{%
    \errmessage{(Inkscape) Color is used for the text in Inkscape, but the package 'color.sty' is not loaded}%
    \renewcommand\color[2][]{}%
  }%
  \providecommand\transparent[1]{%
    \errmessage{(Inkscape) Transparency is used (non-zero) for the text in Inkscape, but the package 'transparent.sty' is not loaded}%
    \renewcommand\transparent[1]{}%
  }%
  \providecommand\rotatebox[2]{#2}%
  \newcommand*\fsize{\dimexpr\f@size pt\relax}%
  \newcommand*\lineheight[1]{\fontsize{\fsize}{#1\fsize}\selectfont}%
  \ifx\svgwidth\undefined%
    \setlength{\unitlength}{597.47994971bp}%
    \ifx\svgscale\undefined%
      \relax%
    \else%
      \setlength{\unitlength}{\unitlength * \real{\svgscale}}%
    \fi%
  \else%
    \setlength{\unitlength}{\svgwidth}%
  \fi%
  \global\let\svgwidth\undefined%
  \global\let\svgscale\undefined%
  \makeatother%
  \begin{picture}(1,0.91671408)%
    \lineheight{1}%
    \setlength\tabcolsep{0pt}%
    \put(0.50822096,0.4648449){\color[rgb]{0,0,0}\makebox(0,0)[lt]{\lineheight{1.25}\smash{\begin{tabular}[t]{l}\scalebox{3}{$z_0$}\end{tabular}}}}%
    \put(0,0){\includegraphics[width=\unitlength,page=1]{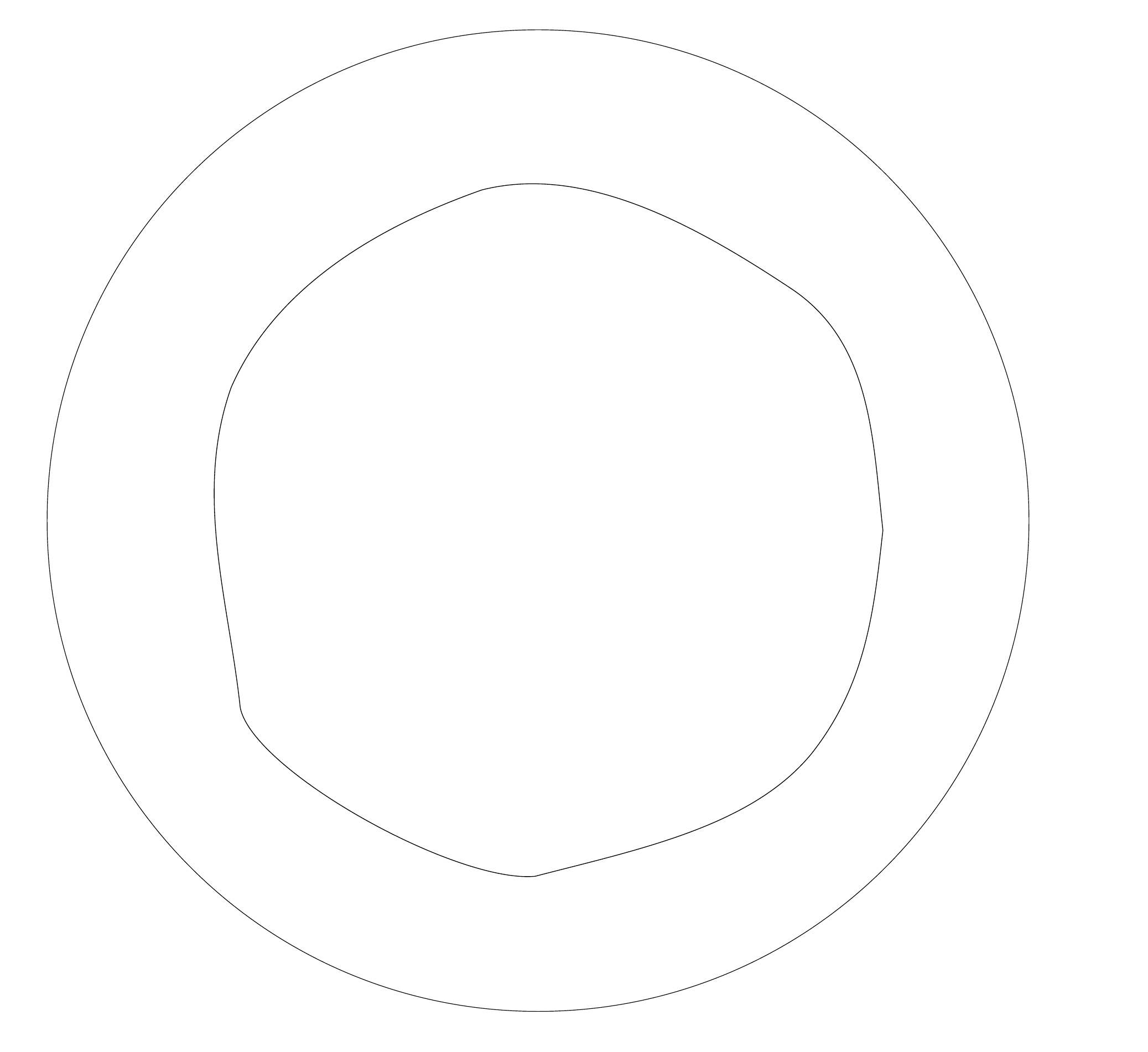}}%
    \put(0.78489529,0.78623233){\color[rgb]{0,0,0}\makebox(0,0)[lt]{\lineheight{1.25}\smash{\begin{tabular}[t]{l}\scalebox{3}{$\gamma_m^m$}\end{tabular}}}}%
    \put(0.67334758,0.69533376){\color[rgb]{0,0,0}\makebox(0,0)[lt]{\lineheight{1.25}\smash{\begin{tabular}[t]{l}\scalebox{3}{$\gamma_{m-1}^m$}\end{tabular}}}}%
    \put(0,0){\includegraphics[width=\unitlength,page=2]{pullbackcurves.pdf}}%
    \put(0.53340038,0.58880533){\color[rgb]{0,0,0}\makebox(0,0)[lt]{\lineheight{1.25}\smash{\begin{tabular}[t]{l}\scalebox{3}{$\gamma_{0}^m$}\end{tabular}}}}%
    \put(0,0){\includegraphics[width=\unitlength,page=3]{pullbackcurves.pdf}}%
  \end{picture}%
\endgroup%
}
		\caption{Illustrated is a brief sketch of the curves $\gamma_k^m$ for $0\leq k \leq m$. }
		\label{pullbackcurves}
	\end{figure}

\begin{rem} By precomposing $\gamma_0^m$ with a translation of $[0,2\pi]$ mod $2\pi$, we may assume there is $\theta_0\in[0,2\pi]$ not depending on $m$ with $z_0=\gamma_0^m(\theta_0)\in\gamma_0^m([0,2\pi])$.
\end{rem}

\begin{lem}\label{chainruleformula} Let $\theta\in[0,2\pi]$. Then \begin{equation}\label{simplifiedeqtn2} (\gamma_0^m)'(\theta)=i\phi(\gamma_0^m(\theta))\cdot \prod_{k=1}^{m-1}\frac{\phi(f^k\circ\gamma^m_0(\theta))}{f^k\circ\gamma^m_0(\theta)}\cdot \prod_{k=0}^{m-1}\frac{1}{\phi'(f^k\circ\gamma^m_0(\theta))}. \end{equation}
\end{lem}

\begin{proof} By Lemma \ref{PrePowerMap} and Lemma \ref{Ak}, we have that \begin{equation}\label{gamma_location}\gamma_k^m([0,2\pi])\subset V_{k+2}\textrm{ for } 0\leq k \leq m.\end{equation} Thus the Definition (\ref{parametrizations}) is such  that: \begin{equation}\label{functional_equation} f^m\circ\gamma_0^m(\theta)=s_m\exp(in_2...n_{m+1}\theta)  \textrm{ for } \theta\in[0,2\pi]. \end{equation} An application of the chain rule then yields: \begin{equation}\label{chain_rule_application} f'(f^{m-1}\circ\gamma_0^m(\theta))\cdot...\cdot f'(\gamma^m_0(\theta))\cdot(\gamma^m_0)'(\theta) = s_min_2...n_{m+1}\exp(in_2...n_{m+1}\theta). \end{equation} By Lemma \ref{PrePowerMap} and (\ref{gamma_location}), we have that \begin{equation}\label{f_derivative} f'(z)=C_kn_k(\phi(z))^{n_k-1}\phi'(z)=n_kf(z)\frac{\phi'(z)}{\phi(z)} \textrm{ for } z\in \gamma^m_{k-2}([0,2\pi]). \end{equation} Thus (\ref{chain_rule_application}) and (\ref{f_derivative}) yield: \begin{equation}\label{simplifiedeqtn} (\gamma^m_0)'(\theta) \cdot \prod_{k=1}^m  \left( n_{k+1} \cdot f^k(\gamma^m_0(\theta))\cdot\frac{\phi'(f^{k-1}\circ\gamma^m_0(\theta))}{\phi(f^{k-1}\circ\gamma^m_0(\theta))} \right) = s_min_2...n_{m+1}\exp(in_2...n_{m+1}\theta) \end{equation} Thus, by using (\ref{functional_equation}), isolating for $(\gamma^m_0)'(\theta)$ in the equation (\ref{simplifiedeqtn}) yields (\ref{simplifiedeqtn2}).
\end{proof}

Note that although the curve $\gamma^m_0([0,2\pi])$ depends on $m$, the point $z_0=\gamma_0^m(\theta_0)$ does not depend on $m$.

\begin{lem}\label{convergenceattheta0} The sequence $(\gamma_0^m)'(\theta_0)$ converges as $m\rightarrow\infty$. 
\end{lem}

%Moreover, by translating $[0,2\pi]$ mod $2\pi$, we may assume there is $\theta_0\in[0,2\pi]$ not depending on $m$ with $z_0=\gamma_0^m(\theta_0)\in\gamma_0^m([0,2\pi])$.
\begin{proof}
	Let $k\geq0$. Since $f^k\circ\gamma^m_0(\theta_0)\in V_{k+2}$ by (\ref{gamma_location}), we have by Lemma \ref{AkEst} that \begin{equation}\label{firstabs} \left| \frac{\phi(f^k\circ\gamma^m_0(\theta))}{f^k\circ\gamma^m_0(\theta)} - 1 \right| \leq C'\cdot2^{-\sqrt{k+N+2}/4}. \end{equation} The standard Cauchy estimates then apply to show that \begin{equation}\label{secondabs} \left|\phi'(f^k\circ\gamma^m_0(\theta_0)) - 1\right| \leq C'\cdot2^{-\sqrt{k+N+2}/4}. \end{equation} Note that the right-hand sides of (\ref{firstabs}) and (\ref{secondabs}) are summable over $k$. Moreover, it is easily verified that if two complex sums $\sum_k|z_k-1|$ and $\sum_k|w_k-1|$ both converge, then so does $\prod_kz_kw_k$. Thus we conclude from Lemma \ref{chainruleformula}, (\ref{firstabs}) and (\ref{secondabs}) that $(\gamma_0^m)'(\theta_0)$ converges.
\end{proof}

To summarize, we have thus far defined the curves $\gamma^m_0([0,2\pi])$ and their parametrizations, and we have shown that $(\gamma^m_0)'(\theta_0)$ converges as $m\rightarrow\infty$. Next we will show that $(\gamma^m_0)$ converges on the following dense subset of $[0,2\pi]$.

\begin{definition}  Let $\Theta_m\in[0,2\pi]$ be such that $f^m\circ\gamma_0^m(\theta_0)=\gamma_m^m(\Theta_m)$. By (\ref{parametrizations}), there are $n_{m+1}\cdot...\cdot n_2$ many points $\theta \in [0,2\pi]$ such that \begin{equation}\label{periodicity_equation}f^m(\gamma_0^m(\theta))=\gamma_m^m(\Theta_m). \end{equation} Denote the collection of $\theta$ satisfying (\ref{periodicity_equation}) as $\mathcal{A}_m$, and define \begin{equation} \mathcal{A}:=\bigcup_{m\geq1}\mathcal{A}_m. \end{equation}
\end{definition}

\begin{lem}\label{uniform_over_dense} Let $\theta\in\mathcal{A}$. Then the sequence $(\gamma_0^k)'(\theta)$ converges as $k\rightarrow\infty$, uniformly over $\mathcal{A}$.
\end{lem}

\begin{proof} Since $\theta\in\mathcal{A}$, we have that $\theta\in\mathcal{A}_m$ for some $m$. Consider $\gamma_0^{m+1}$. Since $z_0\in\gamma_0^{m+1}([0,2\pi])$, it follows from the definition of $\Theta_m$ that $\gamma_m^{m+1}([0,2\pi])$ passes through the point $\gamma_m^m(\Theta_m)$. Moreover, by precomposing $\gamma_m^{m+1}$ with a translation of $[0,2\pi]$ mod $2\pi$ if necessary, we may assume that \begin{equation} \gamma_m^{m+1}(\Theta_m) =  \gamma_m^m(\Theta_m). \end{equation} Thus we conclude that \begin{equation} \gamma_0^{m+1}(\theta)=\gamma_0^m(\theta) \textrm{ for } \theta\in\mathcal{A}_m, \end{equation} and arguing recursively we see that for all $k\geq1$, we have: \begin{equation} \gamma_0^{m+k}(\theta)=\gamma_0^m(\theta) \textrm{ for } \theta\in\mathcal{A}_m. \end{equation} Thus, the sequence  \begin{equation} \left(f^k\circ\gamma_0^m(\theta)\right)_{k=1}^\infty \end{equation} in fact does not depend on $m$. Thus the equation (\ref{simplifiedeqtn2}) and the same exact argument as for Lemma \ref{convergenceattheta0} shows that in fact $(\gamma_0^k)'(\theta)$ converges as $k\rightarrow\infty$ for any $\theta\in\mathcal{A}_m$, with convergence which is uniform over $m$ and the set $\mathcal{A}_m$.

\end{proof}

In order to deduce convergence of $(\gamma_0^k)'$ on all of $[0,2\pi]$, we will need the following Proposition, whose proof is elementary and hence is omitted.

\begin{prop}\label{complete_space} Let $X$ be a complete metric space, $f_n: X \rightarrow \mathbb{C}$ a sequence of uniformly continuous functions, and assume $(f_n)$ converges uniformly on a dense subset of $X$. Then the sequence $f_n$ converges uniformly on all of $X$. 
\end{prop}

\begin{lem}\label{uniform_conv_gamma} The functions $\gamma_k': [0,2\pi]\rightarrow\mathbb{C}$ converge uniformly. 
\end{lem}

\begin{proof} Since $\mathcal{A}$ is dense in $[0,2\pi]$ by (\ref{parametrizations}), Lemma \ref{uniform_over_dense} implies that the functions $(\gamma_0^k)'$ converge uniformly (as $m\rightarrow\infty$) on a dense subset of $[0,2\pi]$. We conclude by Proposition \ref{complete_space} that the functions  $(\gamma_0^k)'$ converge uniformly on $[0,2\pi]$.
\end{proof}

\noindent To deduce that the functions $\gamma_k$ converge, we will use the following elementary result (see Theorem 3.7.1 of \cite{MR3310023}):

\begin{prop}\label{tao_prelim} Let $\gamma_k: [0,2\pi] \rightarrow \mathbb{C}$ be a sequence of $C^1$ functions. Suppose that the functions $\gamma_k'$ converge uniformly to a function $g$, and suppose furthermore that $\gamma_k(\theta_0)$ converges for some $\theta_0\in[0,2\pi]$. Then the functions $\gamma_k$ converge uniformly to a $C^1$ function $\gamma_\infty: [0,2\pi]\rightarrow\mathbb{C}$, and $\gamma_\infty'=g$.  
\end{prop}

\begin{lem} The functions $\gamma_k: [0,2\pi]\rightarrow\mathbb{C}$ converge uniformly to a $C^1$ function $\gamma_\infty: [0,2\pi] \rightarrow\mathbb{C}$.
\end{lem}

\begin{proof} This is a direct application of Lemma \ref{uniform_conv_gamma} and Proposition \ref{tao_prelim}. 
\end{proof}

\noindent In order to prove that $\Gamma_1$ is a $C^1$ curve, it remains to show that $\gamma_\infty'$ does not vanish, and that $\gamma_\infty([0,2\pi])=\Gamma_1$.

\begin{lem} The curve $\gamma_\infty$ satisfies $\gamma_\infty'(\theta)\not=0$ for all $\theta\in[0,2\pi]$.
\end{lem}

\begin{proof} Consider the expression (\ref{simplifiedeqtn2}) for $\theta\in\mathcal{A}$.  If we suppose by way of contradiction that $\gamma_\infty'(\theta)=0$, then one of the infinite products in (\ref{simplifiedeqtn2}) must converge to $0$, and so either \begin{align}\label{firstnondivergence} \sum_{k=1}^\infty \log\left(\frac{\phi(f^k\circ\gamma^m_0(\theta))}{f^k\circ\gamma^m_0(\theta)}\right) \textrm{, or } \\ \label{secondnondivergence} \sum_{k=1}^\infty \log\left(  \phi'(f^k\circ\gamma^m_0(\theta)) \right) \end{align} must diverge. We will show that in fact both of the sums  (\ref{firstnondivergence}), (\ref{secondnondivergence}) converge. Indeed, we have \begin{equation}\label{notzero} \left| \sum_{k=1}^\infty \log\left(\frac{\phi(f^k\circ\gamma^m_0(\theta))}{f^k\circ\gamma^m_0(\theta)}\right) \right| \lesssim   \sum_{k=1}^\infty  \left|\frac{\phi(f^k\circ\gamma^m_0(\theta))}{f^k\circ\gamma^m_0(\theta)} - 1 \right|, \end{equation} and the right-hand side of (\ref{notzero}) converges by (\ref{firstabs}). Thus (\ref{firstnondivergence}) converges, and similarly we can use (\ref{secondabs}) to show that (\ref{secondnondivergence}) converges. Moreover, we deduce that the sums (\ref{firstnondivergence}), (\ref{secondnondivergence}) are bounded uniformly over $\theta\in\mathcal{A}$. Thus we have proven that the sums (\ref{firstnondivergence}) (\ref{secondnondivergence}) are bounded uniformly over a dense subset of $[0,2\pi]$, and hence $\gamma_\infty'$ is bounded away from $0$ uniformly over a dense subset of $[0,2\pi]$. Hence $\gamma_\infty'$ does not vanish on $[0,2\pi]$. 
\end{proof}

\begin{lem} The function $\gamma_\infty$ parametrizes $\Gamma_1$, in other words $\gamma_\infty([0,2\pi])=\Gamma_1$. 
\end{lem}

\begin{proof} It is straightforward to see that $\gamma_\infty([0,2\pi])\subset\Gamma_1$. Indeed, since each $\theta\in\mathcal{A}$ satisfies $f^n(\gamma_\infty(\theta)) \in \cup_kV_k$ for all $n$, we have that $\gamma_\infty(\theta) \in \Gamma_1$ for $\theta\in\mathcal{A}$. Since $\mathcal{A}$ is dense in $[0,2\pi]$ and $\Gamma_1$ is closed, it follows that $\gamma_\infty([0,2\pi])\subset\Gamma_1$. In order to show that $\gamma_\infty([0,2\pi])=\Gamma_1$, we will need to use the fact (proven in Theorem \ref{jordancurvetheorem}) that $\Gamma_1$ is a Jordan curve. Indeed, suppose by way of contradiction that $\gamma_\infty([0,2\pi])\subsetneq\Gamma_1$. Then $\gamma_\infty([0,2\pi])$ is a strict subset of $\Gamma_1$, and since $\gamma_\infty([0,2\pi])$ is closed (as $\gamma_\infty$ is continuous), it follows that there is an open interval $I\subset[0,2\pi]$ such that $\gamma_\infty([0,2\pi])\subset\Gamma_1\setminus\Gamma_1(I)$, where we use $\Gamma_1$ to also denote the parametrization of $\Gamma_1$. However, by Theorem \ref{jordancurvetheorem}, this means that $\gamma_\infty([0,2\pi])$ has empty intersection with a sector of the form $\{z\in\mathbb{C} : \theta_1 < \textrm{arg}(z) <  \theta_2\}$. But then by uniform convergence, this would mean that for all sufficiently large $m$ we have that $\gamma_0^m([0,2\pi])$ has empty intersection with $\{z\in\mathbb{C} : \theta_1 < \textrm{arg}(z) <  \theta_2\}$, and this is a contradiction since each $\gamma_0^m([0,2\pi])$ is a Jordan curve surrounding $0$.
\end{proof}

\noindent Thus we have proven Theorem \ref{C1}. We will deduce that $Z_1$ is $1$-dimensional, but first we need a few preliminary results.

% \begin{equation} \gamma_0^{m}(\theta_j), \gamma_1^{m}(\theta_j), ..., \gamma_m^{m}(\theta_j). \end{equation}

%Thus we see that, for fixed $j$, the sequences $\gamma_0^m(\theta_j)$ are constant and hence converge as $m\rightarrow\infty$. Since $m$ was arbitrary in the above argument and $n_m\rightarrow\infty$ as $m\rightarrow\infty$, we see that the functions $\gamma_0^m$ converge pointwise (as $m\rightarrow\infty$) on a dense subset of $[0,2\pi]$, and hence on all of $[0,2\pi]$. We call the limiting function $\gamma_\infty$. It remains to show that \begin{enumerate} \item $\gamma_\infty$ is $C^1$, \item $\gamma_\infty'$ does not vanish, and \item $\gamma_\infty([0,2\pi])=\Gamma_1$. \end{enumerate}

\begin{lem}
	\label{extra_remarks}
	For each $k \geq 1$, $\Gamma_k$ is a connected component of $\Julia(f_N)$. Moreover, the outer boundary of $\Omega_k$ is equal to the inner boundary of $\Omega_{k+1}$ is equal to $\Gamma_k$.
\end{lem}	
\begin{proof} We first show that $\Gamma_k\subset \Julia(f_N)$. If $z\in\Gamma_k$ and $\varepsilon>0$, then for all sufficiently large $n$ there exists a petal $P\subset A_{n}$ such that $f_N^{-n}(P)\subset B(z,\varepsilon)$, where we use a branch of the inverse of $f_N: \Gamma_{k,n} \rightarrow V_{n}$ (see Figure \ref{C1covering}). Since any petal contains a $0$ of $f_N$, and $0\in \Julia(f_N)$, it follows that $B(z,\varepsilon)\cap \Julia(f_N)\not=\emptyset$. Thus as $\varepsilon$ is arbitrary, we have proven $\Gamma_k\subset \Julia(f_N)$.
	
	Next we show that $\Gamma_k\subset \Julia(f_N)$ is indeed a component of $\Julia(f_N)$. Let $K$ denote the component of $\Julia(f_N)$ which contains $\Gamma_k$. Since $\Gamma_k$ is connected, we have $\Gamma_k\subset K$. Suppose by way of contradiction that $\Gamma_k\subsetneq K$. Note that \[\Gamma_k:=\bigcap_{n=1}^\infty \Gamma_{k,n}=\bigcap_{n=1}^\infty\{\zeta \in V_k : f_N^n(\zeta) \in V_{k+n} \textrm{ for all } n\geq1. \}\] Thus the assumption  $\Gamma_k\subsetneq K$ implies that there must be some point $\zeta \in K\setminus\Gamma_k$ and $n$ such that $f_N^n(\zeta)$ is on the boundary of $V_{k+n}$. But the boundary of $V_{k+n}$ is mapped to the Fatou set, and this is a contradiction. Thus $\Gamma_k\subset \Julia(f_N)$ is indeed a component of $\Julia(f_N)$.
	
	Next, we show that $\Gamma_k$ coincides with the inner boundary of $\Omega_{k+1}$. Recall that $\Omega_{k+1}$ was defined to be the Fatou component containing $B_{k+1}$. Since we have proved $\Gamma_k$ is a Jordan curve component of $\Julia(f_N)$, it suffices to show that if $z\in\Gamma_k$ and $\varepsilon>0$, then $B(z,\varepsilon)\cap\Omega_{k+1}\not=\emptyset$. Let $z\in\Gamma_k$ and $\varepsilon>0$. As observed in the previous paragraph, the boundary of each $V_k$ belongs to the Fatou set. Moreover, $B_{k+1}$ and the outer boundary of $V_k$ both belong to $\Omega_{k+1}$ by Theorem \ref{Julia_place}. By similar reasoning, the outer boundary of $V_k$ and the outer boundary of $\Gamma_{k,1}$ belong to $\Omega_{k+1}$, and recursively we see that the outer boundary of $\Gamma_{k,n}$ belongs to  $\Omega_{k+1}$ for all $n\geq1$. Since the outer boundaries of $\Gamma_{k,n}$ limit on $\Gamma_k$ by the proof of Theorem \ref{jordancurvetheorem}, we see that $B(z,\varepsilon)\cap\Omega_{k+1}\not=\emptyset$ as needed.
	
	Lastly, it remains to show that $\Gamma_k$ coincides with the outer boundary of $\Omega_{k}$. It suffices to show that if $z\in\Gamma_k$ and $\varepsilon>0$, then $B(z,\varepsilon)\cap\Omega_{k}\not=\emptyset$. Our reasoning is similar to that given in the previous paragraph. Namely, note that outer boundary of $B_{k}$ and the inner boundary of $V_k$ both belong $\Omega_k$ by Theorem \ref{Julia_place}. Similarly, the inner boundary of $V_k$ and the inner boundary of $\Gamma_{k,1}$ belong to the same Fatou component $\Omega_k$. Recursively, we see that the inner boundaries of $\Gamma_{k,n}$ all belong to the same Fatou component $\Omega_k$ for all $n$. By the proof of Theorem \ref{jordancurvetheorem}, we see that $B(z,\varepsilon)\cap\Omega_{k}\not=\emptyset$ as needed.

	%Then as we reasoned before, for all sufficiently large $n$, $B(z,\varepsilon)$ must contain points which are mapped to 
	
	%\vspace{10mm}
	%	If $z \in \Gamma_k$, then $\{f_N^n\}_{n=1}^{\infty}$ can never be equicontinuous in a neighborhood of $z$. Indeed, one can show that for each $z \in \Gamma_k$ and every $\varepsilon>0$, there exists a point $w \in B(z,\varepsilon)$ and $m \geq 0$ such that $f_N^m(w) \in E$. Similarly we can show that for any $\varepsilon > 0$, there exists points $w_1,w_2 \in B(z,\varepsilon)$ and so that the exists an integer $m \geq 0$ so that $f^m_N(w_1) \in B_{k+m}$ and $f^m_N(w_2) \in B_{k+m+1}$. By Definition \ref{central_series} and Lemma \ref{central_series_distinct}, we deduce that $w_1 \in \Omega_{k}$ and $w_2 \in \Omega_{k+1}$.  The details are similar to Theorem 9.3 and Lemma 9.5 of \cite{BurPack}.
\end{proof}

\begin{lem} $Z_1\subset\Julia(f_N)$.
\end{lem}

\begin{proof} Let $z\in Z_1$, so that by definition  there exists $l\geq0$ so that, for all $j\geq0$, $f^l_N(z)\in \cup_{k\geq1}V_k$. Since $z\in Z$, we may, by perhaps increasing $l$, further assume that $f^l_N(z)$ never moves backwards. Let $m\geq1$ be such that $f^l_N(z) \in V_m$. Since $f^l_N(z)$ doesn't move backwards and $f^l_N(z)\in \cup_{k\geq1}V_k$, we deduce that $f^{l+1}_N(z) \in V_{m+1}$. By similar reasoning we see that in fact $f^{l+j}_N(z) \in V_{m+j}$ for all $j\geq0$. Thus, by Definition \ref{gamma_definition}, $f^l_N(z)\in \Gamma_{m-1}$. By Lemma \ref{extra_remarks}, $\Gamma_{m-1}\subset\Julia(f_N)$, and so $f^l_N(z)\in\Julia(f_N)$. 
	
	% In particular, that means $z\in Z$ and hence there exists $l\geq0$ so that, for all $j\geq0$, $f^l_N(z)\in \cup_{k\geq1}V_k$ and $f
	
\end{proof}

\begin{lem}\label{Z_1_key} Let $\Gamma$ be a component of $Z_1$. Then there exist $p, n\geq1$ and a Jordan domain $B$ containing $\Gamma$ such that $f_N^n|_B$ is conformal, and $f_N^n(\Gamma)=\Gamma_p$. 
\end{lem}

\begin{proof} It will be convenient to denote $V:=\cup_{j\geq1}V_j$. Let $\Gamma$ be a component of $Z_1$, and let $z\in\Gamma$. Since $z\in Z$, there is a positive integer $m\geq0$ so that $z$ moves backwards precisely $m$ times. Thus there is an element $W_k^n\in\mathcal{C}_m$ containing $z$. Let us first assume $k\geq1$. By Lemma \ref{section_8_distortion}, there exists a Jordan domain $B$ containing $W_k^n$ such that $f_N^n(B)\rightarrow B(0,4R_{k+1})$ is conformal and $f_N^n(W_k^n)=A_k$. In particular, since $\partial A_k\subset\mathcal{F}(f_N)$, we deduce that $\Gamma\subset W_k^n$. In particular, we have that $f_N^n(\Gamma)\subset A_k$. By our choice of $W_k^n\in\mathcal{C}_m$, we have that $f_N^n(\Gamma)$ can only move forward. Moreover by Lemma \ref{Julia_place}, we have that either $f_N^n(\Gamma) \subset V_k$, or $f_N^n(\Gamma)$ is a subset of a petal $P_k\subset\mathcal{P}_k$. Since $z\in Z_1$, there exists a smallest $l\geq0$ and $p\geq1$ such that $f_N^{n+l}(z)\in V_p$ and $f_N^{n+l+j}(z)\in V$ for all $j\geq0$. Moreover, by Lemma \ref{petal_radius} there exists a Jordan domain $B'$ with $ W_n^k \subset B'\subset B$ such that $f_N^l: f_N^n(B') \rightarrow f_N^{n+l}(B')$ is conformal. Thus $f_N^{n+l}(\Gamma)\subset V_p$. Now consider an arbitrary $z'\in \Gamma$. Since $f_N^{n+l}(z')$ only moves forward and $f_N^{n+l+j}(z')\in V$ for all $j\geq0$, we have that $f_N^{n+l+j}(z')\in V_{p+j}$ for all $j\geq0$. Thus by Definition \ref{gamma_definition}, $f_N^{n+l}(z')\in \Gamma_{p-1}$. Since $z'\in \Gamma$ was arbitrary, we have that $f_N^{n+l}(\Gamma)\subset \Gamma_{p-1}$. Lastly, since $f_N^{n+l}: B' \rightarrow f_N^{n+l}(B')$ is conformal, we have that  $f_N^{n+l}(\Gamma) = \Gamma_{p-1}$, and hence the proof is finished in the case that $k\geq1$. If $k<1$, by Lemma \ref{section_8_distortion} we have a Jordan domain $B$ containing $W_k^n$ such that $f_N^n|_B$ is conformal, and $f_N^n(W_k^n)=A_k$ is mapped conformally onto $A_1$, whence the above reasoning applies.

	%\vspace{10mm}
	
	% Since $f_N^{n+l}(z)$ only moves forward and $f_N^{n+l+j}(z)\in V$ for all $j\geq0$, we have that $f_N^{n+l+j}(z)\in V_{p+j}$ for all $j\geq0$. Thus by Definition \ref{gamma_definition}, $f_N^{n+l+j}(z)\in \Gamma_{p+1}$. 

	%Moreover, if $l>0$, then by Lemma \ref{petal_radius} we have that $f_N^l: f_N^n(B) \rightarrow f_N^{n+l}(B)$ is conformal. Thus, as before, we have that $f_N^{n+l}(\Gamma)\subset V_p$. 
\end{proof}

\begin{cor}\label{Z_1iscountableunion} The set $Z_1$ is a countable disjoint union of $C^1$ Jordan curves. In particular, the Hausdorff dimension of $Z_1$ is equal to $1$.
\end{cor}

\begin{proof} By Theorem \ref{C1} and Lemma \ref{Z_1_key}, each component of $Z_1$ is a conformal image of a $C^1$ Jordan curve, hence each component of $Z_1$ is a $C^1$ Jordan curve. Since any Jordan curve has non-empty interior, there can be at most countably many components of $Z_1$. Lastly, $\textrm{dim}(Z_1)=1$ follows from Lemma \ref{Hdim_Facts}.
\end{proof}

\section{Singleton Boundary Components}\label{singleton_components}

In this last Section, we analyze finally the set $Z_2$. Recall that we have proven that \[ \mathcal{J}(f) \subset E' \cup Y \cup Z_1 \cup Z_2, \] and so to prove (1) of Theorem \ref{main_theorem} it only remains to estimate the dimension of $Z_2$. In this Section we will prove that in fact $Z_2$ has dimension $0$ and consists of uncountably many singletons. We begin by constructing a sequence of covers for $Z_2$.

 Recall from Definition \ref{Z1Z2} that: \begin{equation} Z_2=\left\{ z \in Z : \textrm{ there exist arbitrarily large } n \textrm{ such that } f_N^n(z)\not\in\cup_{k\geq1}V_k \right\}. \end{equation} We first analyze $Z_2$ intersected with the closure of a Fatou component $\Omega_k$. 

\begin{lem}
	\label{Z2_in_Ak}
	Suppose that $z \in Z_2 \cap \overline{\Omega_k}$ for some $k\geq1$. Then $z \in A_k$. 
\end{lem}	
\begin{proof}
	Since $z \in Z_2 \subset X$, the orbit sequence of $z$ is $(k(z,n))_{n=0}^{\infty}$, and we have 
	$$f^n_N(z) \in A_{k(z,n)}$$
	for all $n \geq 0$. 
	
	Note that $\overline{\Omega_k} \subset A_k \cup B_k \cup  A_{k+1}$. Since $z \in X$, we must have $z \in A_k$ or $z \in A_{k+1}$. Suppose for the sake of contradiction that we had $z \in A_{k+1}$. Since $z \in \overline{\Omega_k}$, we have $f_N^l(z) \in \overline{ \Omega_{k+l}}$ for all $l \geq 0$. The outermost boundary component of $\Omega_{k}$ is $\Gamma_k$ by Lemma \ref{extra_remarks}, and $\Gamma_k \subset V_{k+1}$ by (\ref{Gammak}). Therefore, we must have 
	\begin{equation*}
	z \in A(\frac{1}{4}R_{k+1},\frac{3}{5}R_{k+1}).
	\end{equation*}
	By Lemma \ref{annulus_2}, $f_N(A(\frac{1}{4}R_{k+1},\frac{2}{5}R_{k+1})) \subset B_{k+1}$, so since $z \in X$ we must have $z \in V_{k+1}$. Since $f_N(V_{k+1}) \subset B_{k+1} \cup A_{k+2} \cup B_{k+2}$ by Corollary 
	\ref{helpful_Vk_cor}, we deduce that $f_N(z) \in A_{k+2} \cap \overline{\Omega_{k+1}}$.
	
	By repeating the reasoning above, we deduce that $f_N^l(z) \in V_{k+l+1}$ for all $l \geq 0$. This contradicts the fact that $z \in Z_2$, so we must have $z \in A_k$.
\end{proof}

\begin{lem}
	\label{Z2_Boundary} 
	Suppose that $z \in \overline{\Omega_k} \cap Z_2$ for some $k \geq 1$. Then $z \in \partial \Omega_k$.
\end{lem}
\begin{proof}
	Recall that 
	\begin{equation}
	\label{recall_omega1}
	\{ z \in \mathbb{C} : 4R_k \leq |z| \leq R_{k+1}/4 \} \subset \Omega_k.
	\end{equation}
	Suppose for the sake of contradiction that $z \in \Omega_k$.  Then there exists $\varepsilon >0$ so that $B(z,\varepsilon) \subset \Omega_k.$ By Theorem 1.2 of \cite{BergRipStalWD}, there exists $m >0$ and $\alpha > 0$ so that for all $n \geq m$ we have 
	\begin{equation}
	\label{absorbing_annulus}
	A(|f_N^n(z)|^{1-\alpha}, |f_N^n(z)|^{1+\alpha}) \subset f_N^n(B(z,\varepsilon)) \subset \Omega_{k+n}.
	\end{equation}
	%Since $z\in \Omega_k$, we have by (\ref{recall_omega}) that $z\in A_k$ or $z\in A_{k+1}$. 
	By Lemma \ref{Z2_in_Ak} 
	\begin{align}
	\label{first_case} 
	\frac{1}{4}R_{k+j} \leq |f_N^j(z)| \leq 4 R_{k+j}
	\end{align} 
	holds for all $j\geq 0$. Notice that by Lemma \ref{Rkest},
	\begin{equation*}
	\frac{\frac{1}{4^{1+\alpha}}R_{k+n}^{1+\alpha}}{4^{1-\alpha}R_{k+n}^{1-\alpha}} = \frac{1}{16} R_{k+n}^{2\alpha} \xrightarrow{n \rightarrow \infty} \infty.
	\end{equation*}
	Then by perhaps increasing $m$ we have for all $n \geq m$ that 
	\begin{equation*}
	\frac{\frac{1}{4^{1+\alpha}}R_{k+n}^{1+\alpha}}{4^{1-\alpha}R_{k+n}^{1-\alpha}} > 1.
	\end{equation*}
	Therefore, for all $n \geq m$, the annulus $A(4^{1-\alpha}R_{k+n}^{1-\alpha}, \frac{1}{4^{(1+\alpha)}} R_{k+n}^{1+\alpha})$ is not empty and
	\begin{equation}
	A(4^{1-\alpha}R_{k+n}^{1-\alpha}, \frac{1}{4^{(1+\alpha)}} R_{k+n}^{1+\alpha}) \subset A(|f_N^n(z)|^{1-\alpha}, |f_N^n(z)|^{1+\alpha}) \subset \Omega_{k+n}. 
	\end{equation}	
	By perhaps increasing $m$ one last time, we can use Lemma \ref{Rkest} to deduce that for all $n \geq m$ we have
	\begin{equation}
	\label{absorbing_annulus_previous}
	4^{1-\alpha} R_{k+n} R_{k+n}^{-\alpha} < \frac{1}{4} R_{k+n}.
	\end{equation}
	By (\ref{absorbing_annulus}) and (\ref{absorbing_annulus_previous}), we deduce that $f_N^n(B(z,\varepsilon))$ contains a point $w \in B_{k+n-1} \subset \Omega_{k+n-1}$. This is a contradiction to the fact that $f_N^n(B(z,\varepsilon)) \subset \Omega_{k+n}$ for all $n \geq m$. 
\end{proof}	

%Choose some point $z$ with $|z| = 8R_1$. Then define %$(l_n)_{n=1}^{\infty}$ by 
%\begin{equation}
%l_n := |f^{n-1}_N(z)|.
%\end{equation}
%Since $B_k \subset \Omega_k$ and $f_N(B_k) \subset B_{k+1}$ for all $k \geq 1$, the circle of radius $l_k$ is a subset of $\Omega_k$ for all $k \geq 1$. We define the \textit{inner connectivity} of $\Omega_k$ to be the number of complementary components of $\Omega_k$ contained in $B(0,l_k)$, and we define the outer connectivity of $\Omega_k$ to be the number of complementary components of $\Omega_k$ contained in $\C \setminus B(0,l_k)$.

%\begin{lem}
%The inner connectivity of $\Omega_k$ is infinite for all $k \geq 1$, and the outer connectivity of $\Omega_k$ is $2$ for all $k \geq 1$.
%\end{lem}

%By Theorem 7.1 of \cite{RSEremenkoPoints}, each $\Omega_k$ has uncountably many complementary components which accumulate on the innermost boundary of $\Omega_k$. By the results of section 9, we know that $\Omega_k$ has countably many complementary components bounded by $C^1$ smooth Jordan curves. 

Recall that we introduced the petals $P_j \subset \mathcal{P}_j$ for all $j \geq 1$ in Definition \ref{petal_def}.

\begin{lem}
	\label{Z2_in_petals_gen}
	Suppose that $z \in Z_2 \cap A_k$, and suppose that the orbit sequence of $z$ is $(k(z,n))_{n=0}^{\infty} = (k,k+1,k+2,\dots).$ Then 
	\begin{equation}
	\label{Z2_in_petals_gen_eqn}
	z \in \bigcap_{l= 1}^{\infty} \left(\bigcup_{j \geq l} f_N^{-j}(\mathcal{P}_{k+j})  \cap A_k\right).
	\end{equation}
\end{lem}
\begin{proof}
	Since $z \in A_k$ and $(k(z,n))_{n=0}^{\infty} = (k,k+1,k+2,\dots)$, we have $f_N^l(z) \in A_{k+l}$ for all $l \geq 0$. Since $z \notin Z_1$, there exists infinitely many positive integers $j$ so that $f_N^j(z) \notin V_{k+j}$.  For those values $j$, we still must have have $f_N^{j+1}(z) \in A_{k+j+1}$, so Lemma \ref{Julia_Place_AkAk1} implies that $f_N^{j}(z) \in \mathcal{P}_{k+j}$. The inclusion (\ref{Z2_in_petals_gen_eqn}) follows immediately. 
\end{proof}

\begin{lem}
	\label{Boundary_Z}
	Suppose that $z \in \partial \Omega_k$ for some $k \geq 1$. Then $z \in Z$ and the orbit sequence of $z$ is either $(k(z,n))_{n=0}^{\infty} = (k,k+1,\dots)$ or $(k(z,n))_{n=0}^{\infty} = (k+1,k+2,\dots)$. In the latter case, we must have $z \in \Gamma_k \subset Z_1$. 
\end{lem}
\begin{proof}
	Let $z\in\partial \Omega_k$. First recall that $\partial \Omega_k \subset A_k \cup A_{k+1}$. Since $\Omega_k$ is a bounded Fatou component, we have $f(\partial \Omega_k) = \partial \Omega_{k+1}$ (see the paragraph above Theorem 3.2 in \cite{BergRipStalWD}). Therefore, we have $f^n_N(z) \in A_{k+n} \cup A_{k+n+1}$ for all $n \geq 0$. 
	
	Suppose that $z \in A_{k+1} \cap \partial \Omega_k.$ We argue similarly to Lemma \ref{Z2_in_Ak}. By Lemma \ref{extra_remarks} and (\ref{Gammak}), the outermost boundary of $\Omega_k$ is a subset of $V_{k+1}$. Therefore, we must have 
	\begin{equation*}
	z \in A(\frac{1}{4}R_{k+1},\frac{3}{5}R_{k+1}).
	\end{equation*}	
	Since $f_N(A(\frac{1}{4}R_{k+1},\frac{2}{5}R_{k+1})) \subset B_{k+1}$ and $z \in \Julia(f_N)$, by Lemma \ref{BkFatou} we must have $z \in V_{k+1}$. Since $f_N(V_{k+1}) \subset B_{k+1} \cup A_{k+2} \cup B_{k+2}$ and $f_N(z) \in \partial \Omega_{k+1}$, we obtain $f_N(z) \in A_{k+2}$.
	
	By iterating the reasoning above, we deduce that $f_N^l(z) \in V_{k+l+1}$ for all $l \geq 0$, so that $z \in \Gamma_k$ and $(k(z,n))_{n=0}^{\infty} = (k+1,k+2,\dots)$. 
	
	The other possibility is that $z \in A_k \cap \partial \Omega_k$. Since $f_N(z) \in \partial \Omega_{k+1}$, we must have $f_N(z) \in A_{k+1} \cup A_{k+2}$. By Lemma \ref{other widehat help}, we must have $f_N(z) \in A_{k+1}$. By repeating this reasoning, we deduce that $f_N^l(z) \in A_{k+l}$ for all $l \geq 0$, and $(k(z,n))_{n=0}^{\infty} = (k,k+1,\dots).$
\end{proof}	

\begin{cor}
	\label{Z2_in_petals}
	For all $k \geq 1$, we have   
	\begin{equation}
	\label{Z2_in_petals_eqn}
	Z_2 \cap \partial\Omega_k \subset \bigcap_{l= 1}^{\infty} \left(\bigcup_{j \geq l} f_N^{-j}(\mathcal{P}_{k+j})  \cap A_k\right).
	\end{equation}
\end{cor} 
\begin{proof}
	Let $z \in Z_2 \cap \partial\Omega_k$. By Lemma \ref{Z2_in_Ak}, we have $z \in A_k$, and by Lemma \ref{Boundary_Z} the orbit sequence of $z$ must be $(k(z,n))_{n=0}^{\infty} = (k,k+1,k+2,\dots)$. The result now follows from Lemma \ref{Z2_in_petals_gen}.
\end{proof}

To estimate the Hausdorff dimension of $Z_2 \cap \partial \Omega_k$, we will need the following estimates on the expansion of $f_N$ on the petals $\mathcal{P}_k$ for all $k \geq 1$.	
%\begin{lem}
%\label{Derivative_Vk_Lem}	
%Suppose that $z \in V_k$ for some $k \geq 1$. Then we have 
%\begin{equation}
%\label{Derivative_Vk}
%|f_N'(z)| \geq n_k \left(\frac{2}{3}\right)^{n_k} \frac{R_{k+1}}{R_k}
%\end{equation}
%\end{lem}
%\begin{proof}
%First, we consider the case of $k = 1$. In this case, we have $f_N'(z) = q_N'(\phi_N^{-1}(z)) \cdot (\phi_N^{-1})'(z)$ for all $z \in V_1$. Since $q_N'(z) = c_N M_N z^{M_N - 1} + r_N$, we deduce using (\ref{phi_growth_9}) that
%		\begin{align*}
%		|q_N'(\phi_N^{-1}(z))| &\geq M_N c_N \left(\frac{1}{3}r_N\right)^{M_N - 1} - r_N \\
%		&= \frac{3 M_N}{r_N} \left( \frac{2}{3}\right)^{M_N} r_{N+1} - r_N \\
%		&\geq 3 M_N \left( \frac{2}{3}\right)^{M_N} \frac{r_{N+1}}{r_N}\left(1 - \frac{r_N^2}{3 M_Nr_{N+1}(\frac{2}{3})^{M_N}}\right)\\ 
%	(\textrm{Lemma } \ref{rkinequalities})	&\geq 2 M_N \left( \frac{2}{3}\right)^{M_N} \frac{r_{N+1}}{r_N} = 2n_1 \left( \frac{2}{3}\right)^{M_N} \frac{R_2}{R_1}
%		\end{align*}
%The case with $k \geq 2$ is similar. We compute that
%\begin{align*}
%|f_N'(\phi_N^{-1}(z))| \geq n_{k} C_k \left(\frac{1}{3}R_k\right)^{n_k - 1} = \frac{3}{R_k}n_k \left(\frac{2}{3}\right)^{n_k} R_{k+1} 
%= 3n_k \left(\frac{2}{3}\right)^{n_k} \frac{R_{k+1}}{R_k} \\
%\end{align*}
%Equation (\ref{Derivative_Vk}) now follows from the chain rule and (\ref{phi_derivative_9}).
%\end{proof}

\begin{lem}
	\label{Derivative_Petal}
	There exists $M$ so that for all $N \geq M$ and for all $k \geq 1$ and all $z \in \mathcal{P}_k$ we have
	\begin{equation}
	\label{Derivative_Petal_eqn}
	|f_N'(z)| \geq \frac{1}{4}n_k\frac{R_{k+1}}{R_k}.
	\end{equation}
\end{lem}
\begin{proof}
	Let $w$ be a zero of $f_N$ contained inside of some connected component $P_k$ of $\mathcal{P}_k$. First note that there exists $M$ so that for all $N \geq M$ and all $k \geq 1$ the modulus of $B(w, \lambda(\exp(\pi/n_k) - 1)R_k) \setminus \overline{B(w, R_k/2^{n_{k}})}$ is bounded below $(2\pi)^{-1} \log 2$. Therefore, by Theorem \ref{Derivative_Distortion_Estimate} and Lemma \ref{petal_radius}, there exists a constant $P \geq 1$ that does not depend on $k$ or $N$ so that for all $z \in B(w, R_k/2^{n_{k}})$ we have 
	\begin{equation}
	\frac{1}{P} \leq \frac{|f_N'(z)|}{|f_N'(w)|} \leq P
	\end{equation}
	Therefore, by (\ref{expansion}) we have 
	\begin{align*}
	|f_N'(z)| &\geq \frac{1}{P} |f_N'(w)| \geq \frac{1}{P} \frac{\delta 2^{n_k}}{8\lambda\pi} n_k \frac{R_{k+1}}{R_k}.
	\end{align*}
	By Lemma \ref{Rkest} there exists $M \in N$  so that for all $N \geq M$, we have 
	\begin{equation}
	\label{Koebe_Distortion_P}
	\frac{1}{P} \frac{\delta 2^{n_k}}{8\lambda\pi} \geq \frac{1}{4}.
	\end{equation}
	Equation (\ref{Derivative_Petal_eqn}) follows immediately.
\end{proof}

\begin{cor}
	\label{Derivative_Product}
	Fix some $k \geq 2$. Suppose that $z \in f_N^{-j}(\mathcal{P}_{k+j}) \cap A_k$ for some $j \geq 1$. Then 
	\begin{equation}
	\label{Derivative_Product_eqn}
	|(f_N^j)'(z)| \geq  \frac{1}{4^j} \frac{R_{k+j}}{R_k} \prod_{l=0}^{j-1}n_{l+k}.
	\end{equation}
\end{cor}
\begin{proof}
	By repeatedly applying Lemma \ref{Julia_Place_AkAk1}, we see that for each $l = 1,\dots, j$, we have either $f^{-l}_N(P_{k+j})$ belongs to $\mathcal{P}_{k+j-l}$ or $V_{k+j-l}$. Therefore, if $z \in f^{-j}_N(P_{k+j}) \cap A_k$, we have by Lemmas \ref{foliation_derivative_estimate}  and \ref{Derivative_Petal} and the chain rule that
	\begin{align*}
	|(f^j_N)'(z)| &\geq \prod_{l=0}^{j-1} \frac{1}{4} n_{l+k} \frac{R_{l+k+1}}{R_{l+k}} =  \frac{1}{4^j} \frac{R_{k+j}}{R_k} \cdot \prod_{l=0}^{j-1} n_{l+k}.
	\end{align*}
	This is exactly what we wanted to show.
\end{proof}

\begin{thm}
	We have $\dim_H(Z_2 \cap \partial \Omega_k) = 0$.
\end{thm}
\begin{proof}
	Let $j \geq 1$ and $W$ be a connected component of $f^{-j}_N(P_{k+j}) \cap A_k$. If $z \in W$, then we have by Theorem \ref{BilipConformal} that there exists a constant $P'$ independent of $N$, $k$, and $j$ such that 
	\begin{equation}
	\label{small prepetal}
	\diam(W) \leq \frac{P'}{|(f_N^j)'(z)|} \diam(P_{k+j}) 
	\leq  \frac{P' 4^j}{\prod_{l=0}^{j-1} n_{l+k}} \frac{R_k}{R_{k+j}}  \frac{R_{k+j}}{2^{n_{k+j}}}
	\leq P' \frac{R_k}{2^{n_{k+j}}}.
	\end{equation}
	
	Fix some $t > 0$. For $j \geq 1$, we define
	$$G_j = \{W\,:W \subset A_k \textrm{ and } f_N^j(W) = P_{k+j} \textrm{ for some } P_{k+j} \subset \mathcal{P}_{k+j}\}.$$ 
	For each petal $P_{k+j}$, there are $n_{k+1}\cdots n_{k+j}$ many connected components $W \in G_j$ by Lemma \ref{component_counter}. Since there are $n_{k+j}$ many connected components of $\mathcal{P}_{k+j}$, and recalling that $L_k = n_1 \cdots n_k$, we count the number of connected components of $G_j$ as  
	$$n_{k+1} \cdots n_{k+j} \cdot n_{k+j} = 2^j n_k \cdots n_{k+j-1} \cdot n_{k+j} \leq 2^j L_{k+j}.$$
	Therefore, we have
	\begin{equation}
	\sum_{W\in G_j} \diam(W)^t \leq 2^j L_{k+j}(P')^t \frac{ R^t_k}{2^{t \cdot n_{k+j}}}.
	\end{equation}
	Therefore, for any fixed $l \geq 1$
	\begin{align}
	\label{key_sum_Z2}
	\sum_{j\geq l} \sum_{W \in G_j} \diam(W)^t \leq (P')^t \cdot R^t_k \sum_{j \geq l} 2^j L_{k+j} \left(\frac{1}{2^t}\right)^{n_{k+j}}
	\end{align} 
	This series converges by the Ratio test. Indeed,
	\begin{align}
	\label{small_sum_Z2}
	\lim_{j \rightarrow \infty} \frac{2^{j+1}}{2^j} \frac{L_{k+j+1}}{L_{k+j}}  \left(\frac{1}{2^t}\right)^{n_{k+j}}&= \lim_{j \rightarrow \infty} 4n_{k+j}  \left(\frac{1}{2^t}\right)^{n_{k+j}} =0.
	\end{align}
	Since (\ref{key_sum_Z2}) converges, we have that for any $\varepsilon >0$, if $l$ is sufficiently large then:
	$$\sum_{j\geq l}\sum_{W \in G_j} \diam(W)^t < \varepsilon.$$
	By Corollary \ref{Z2_in_petals}, we conclude that $H^t(Z_2 \cap \partial \Omega_k) = 0$. Since $t >0$ was arbitrary, we further conclude that $\dim_H(Z_2 \cap \partial \Omega_k) = 0$.
\end{proof}

Now that we know that $\dim_H(Z_2 \cap \partial \Omega_k) = 0$ for all $k \geq 0$, we move on to estimating $\dim_H(Z_2).$	
\begin{lem}
	\label{in_boundary}
	Suppose that $z \in A_k$ for some $k \geq 1$. Suppose further that the orbit sequence of $z$ is $k(z,n) = (k,k+1,k+2,\dots)$. Then $z \in \overline \Omega_k$. 
\end{lem}	
\begin{proof}
	Since $z \in Z_2 \cap A_k$, by Lemma \ref{Z2_in_petals_gen} we have
	\begin{equation*}
	z \in \bigcap_{l= 1}^{\infty} \left(\bigcup_{j \geq l} f_N^{-j}(\mathcal{P}_{k+j}) \cap A_k\right).
	\end{equation*}
	Therefore there exists a sequence $(l_j)$ of increasing integers so that $f_N^{l_j}(z) \in P_{k+l_j}$ for some petal $P_{k+l_j} \in \mathcal{P}_{k+l_j}$. Let $W_{l_j}$ denote the connected component of $f_N^{-l_j}(P_{k+l_j})$ which contains $z$. By (\ref{small prepetal}), $\diam(W_{l_j}) \rightarrow 0$ as $j \rightarrow \infty$.
	
	Let $\varepsilon >0$ be given. Then there exists $l_j$ such that $W_{l_j} \subset B(z,\varepsilon)$. By (\ref{conformal_2}), there exists a point $w \in P_{k+l_j}$ so that $|f_N(w)| = 3R_{k+l_j+1}$. By Lemma \ref{annulus_1}, $f^2_N(w) \in B_{k+l_j+2}$, so that $f^2_N(w) \in \Omega_{k+l_j+2}$. Therefore, the element of $f_N^{-l_j}(w)$ that belongs to $W_{l_j}$ belongs to $\Omega_k$. Therefore, $B(z,\varepsilon) \cap \Omega_k$ is not empty, and since $\varepsilon >0$ was arbitrary, it follows that $z \in \overline{\Omega_k}.$
\end{proof}

\begin{cor}
	We have
	\begin{equation}
	\label{Z2_Countable_Stab}
	Z_2 \subset \bigcup_{j \geq 0} f_N^{-j} \left(\bigcup_{k=1}^{\infty} Z_2 \cap \partial \Omega_k \right). 
	\end{equation}
	Moreover, we have $\dim_H(Z_2) = 0$.
\end{cor}
\begin{proof}
	Since $z \in Z_2$, there exists $m \geq 0$ and $k \geq 1$ so that $f_N^m(z) \in A_k$ and the orbit sequence of $f_N^m(z)$ is strictly increasing and given by $(k, k+1,\dots)$. It follows that $f_N^m(z) \in \overline{\Omega_k}$ by Lemma \ref{in_boundary}. Thus, by Lemma \ref{Z2_Boundary}, we have that $f_N^m(z) \in \partial \Omega_k$. Therefore, (\ref{Z2_Countable_Stab}) holds. 
	
	By (\ref{dimension_entire}) and (\ref{countable_stability}), it follows from (\ref{Z2_Countable_Stab}) that $\dim_H(Z_2) = 0$.
\end{proof}

\begin{cor}
	\label{singletons}
	The set $Z_2$ is totally disconnected. In particular, every connected component of $Z_2$ is a point. 
\end{cor}
\begin{proof}
	The Hausdorff dimension of any non-singleton connected set is bounded below by $1$. Since $\dim_H(Z_2) = 0 < 1$, $Z_2$ cannot have any non-singleton connected components.
\end{proof}

\begin{cor}
	\label{Singleton_C1_Decomp}
	Let $k \geq 1$. Then $\partial \Omega_k$ consists of countably many $C^1$ smooth Jordan curves and uncountably many singleton components. The singleton components coincide with $\Omega_k \cap Z_2$ and the $C^1$ smooth components coincide with $\Omega_k \cap Z_1$. 
\end{cor}
\begin{proof}
	Since $\partial \Omega_k \subset Z$ by Lemma \ref{Boundary_Z}, we have $$\partial \Omega_k = \left(\partial \Omega_k \cap Z_1\right) \sqcup \left(\partial \Omega_k \cap Z_2\right).$$
	Every component of $\Omega_k \cap Z_1$ is a $C^1$ smooth Jordan curve by Corollary \ref{Z_1iscountableunion}, and every component of $\Omega_k \cap Z_2$ is a singleton by Corollary \ref{singletons}. There are uncountably many such components by Theorem 7.1 of \cite{RSEremenkoPoints}.
\end{proof}

\begin{cor}
	\label{Singleton_C1_Decomp_Gen}
	Let $\Omega$ be a Fatou component of $f_N$. Then $\partial \Omega$ consists of uncountably many singleton components and countably many $C^1$ smooth Jordan curves. 
\end{cor}
\begin{proof}
	Note that every Fatou component of $f_N$ is bounded. Let $\Gamma$ denote the connected component of the boundary of $\Omega$ that separates $\Omega$ from $\infty$. Since $\Gamma \subset \Julia(f_N)$, by Lemma \ref{Julia_Place_X} we have $\Gamma \subset E' \cup X = E' \cup Y \cup Z_1 \cup  Z_2$. Since $\Gamma$ is a non-singleton connected set, we must have $\Gamma \subset Z_1$. By Lemma \ref{Z_1_key}, there exists $p,n \geq 1$ and a Jordan domain $B$ containing $\Gamma$ such that $f_N^n|_B$ is conformal and $f^n_N(\Gamma) = \Gamma_p$. Consequently, we have $f^n_N(\Omega) = \Omega_p$, and $f^n_N|_{\Omega}$ is conformal. The result now follows from Corollary \ref{Singleton_C1_Decomp}.
\end{proof}

%\section*{Appendix}
%\label{Appendix}

\appendix
\section{ }
\label{appendix}

In this Appendix we collect several classical Theorems and Definitions used throughout the paper, and we will briefly prove a technical result (needed in Section \ref{Expanding Zeros}) on the behavior of the interpolating map of \cite{BurLaz} near its zeros. We begin with the statements of some classical distortion theorems for conformal mappings.

%\begin{definition}
%\label{general_powers_def} 
%Let $c\in\mathbb{C}^\star$, $(M_j)_{j=1}^\infty \in \mathbb{N}$ be increasing, and $(r_j)_{j=1}^\infty \in \mathbb{R}^+$. We set  $r_0:=0$,  $M_0:=1$. Suppose that 
%\begin{equation}
%\label{growth_condition}
%r_{j+1} \geq  \exp\left(\pi\big/M_j\right) \cdot r_j \textrm{ for all } j\in\mathbb{N}, \textrm{ and } r_j\xrightarrow{j\rightarrow\infty}\infty. 
%\end{equation} 
%Set 
%\begin{equation}\label{c_defn} c_1:=c\textrm{, and } c_j:=c_{j-1}\cdot r_{j-1}^{M_{j-1}-M_{j}}=c\cdot\prod_{k=2}^{j}r_{k-1}^{M_{k-1}-M_{k}} \textrm{ for } j\geq2. \end{equation}  
%We then define: 
%\begin{equation}
%\label{h_formula}
% h(z):= \begin{cases} 
%	c_j\cdot z^{M_j} & \textrm{ if } \hspace{2mm} r_{j-1}\cdot\exp(\pi/M_{j-1}) \leq |z| \leq r_j \\
%	g_{M_{j}, M_{j+1}, r_j, c_j}(z) & \textrm{ if } \hspace{2mm}  r_j \leq |z| \leq r_j\cdot \exp(\pi/M_j). 
%	\end{cases} 
%\end{equation} 
%over all $j\in\mathbb{N}$. 
%\end{definition}

%For the definition of the quasiregular mapping $g$ in (\ref{h_formula}), we refer the reader to Section $3$ of \cite{BurLaz}. 

\begin{thm}[Koebe $1/4$ Theorem]
	\label{Kobe_Quarter}
	Let $D \subset \C$ be a domain and let $z_0 \in D$, and suppose that $f: D \rightarrow f(D)$ is conformal. Then 
	\begin{equation}
	\label{Kobe_Quarter_Ineq}
	\frac{1}{4} |f'(z_0)| \leq \frac{\dist(f(z_0), \partial(f(D)))}{\dist(z_0,\partial D)} \leq 4 |f'(z_0)|. 
	\end{equation}
\end{thm}

%If $f: D \rightarrow f(D)$ is conformal, and $K$ is relatively compact in $D$, then the \textit{conformal distortion} of $f$ on $K$ is
%\begin{equation}
%\mathcal{D}|_{K} := \sup_{z,w \in K} \frac{|f'(z)|}{|f'(w)|}.
%\end{equation}

\begin{thm}
	\label{Derivative_Distortion_Estimate}
	Let $D$ be a simply connected domain and let $f:D \rightarrow f(D)$ be a conformal mapping. Let $U$ be a relatively compact subset of $D$. Then there is a constant $C$ which depends only on the modulus of $D \setminus \overline{U}$ such that 
	\begin{equation}
	\label{Derivative_Distortion_Estimate_Eqn}
	\frac{1}{C} \leq \sup_{z,w \in U} \frac{|f'(z)|}{|f'(w)|} \leq C.   
	\end{equation}
\end{thm}

\noindent Next, we need the following consequence of the Koebe distortion theorem. The statement below is Theorem 2.9 of \cite{McMullenBook}. 

\begin{thm}
	\label{BilipConformal}
	Let $U$ and $D$ be simply connected domains with $U$ compactly contained in $D$. Let $f:D \rightarrow f(D)$ be conformal. Then there exists a constant $C$ that depends only on the modulus of $D \setminus \overline{U}$ such that for any $x,y,z \in U$, we have
	\begin{equation}
	\label{BilipConformal_eqn}
	\frac{1}{C}|f'(x)| \leq \frac{|f(y)-f(z)|}{|y-z|} \leq C |f'(x)|.
	\end{equation}
\end{thm}

\noindent Using the BiLipschitz estimate (\ref{BilipConformal_eqn}), we obtain the following Corollary. 

\begin{cor}[Koebe Distortion Theorem]
	\label{Kobe}
	Let $D$ be simply connected, let $U$ be open and compactly contained in $D$, and let $K$ be a compact subset of $\overline{U}$. Suppose $f: D \rightarrow f(D)$ is conformal. Then there is a constant $C$ which depends only on the modulus of $D \setminus \overline{U}$ so that 
	\begin{equation}
	\label{Kobe_Ineq}
	\frac{1}{C}\frac{\diam(K)}{\diam(U)} \leq \frac{\diam(f(K))}{\diam(f(U))} \leq C \frac{\diam(K)}{\diam(U)}.
	\end{equation}	
\end{cor}

\noindent We can also deduce the following corollary using (\ref{BilipConformal_eqn}), but we first need the following definitions.

\begin{definition}  Let $f: D \rightarrow f(D)$ be a conformal mapping, and $B = B(z_0,r)$ be compactly contained in $D$. We define the \textit{inner radius} of $f(B)$ by:
	\begin{equation}
	\label{inner_radius}
	r_{f(B),f(z_0)} := \sup \{t: B(f(z_0), t) \subset f(B)\}.
	\end{equation}
	We similarly define the \textit{outer radius} of $f(B)$ by: 
	\begin{equation}
	\label{outer_radius}
	R_{f(B),f(z_0)} := \inf \{t: f(B) \subset B(f(z_0), t)\}.
	\end{equation}
\end{definition}

\begin{cor}
	\label{shape}
	Let $D$ be a simply connected domain and let $f:D \rightarrow f(D)$ be conformal. Let $B=B(z_0,r)$ be a disk compactly contained inside of $D$. Then there exists a constant $C$ which depends only on the modulus of $D \setminus \overline{B}$ so that 
	\begin{equation}
	\label{shape_ineq}
	C^{-1} |f'(z_0)| r \leq r_{f(B),f(z_0)} \leq R_{f(B),f(z_0)} \leq C |f'(z_0)| r.
	\end{equation}
\end{cor}

\begin{rem}
	\label{Bounded_conformal_distortion}
	In this paper, we will frequently encounter the following situation. Let $f: \C \rightarrow \C$ be an entire function, and let $D_n$ be a sequence of simply connected domains in $\C$. Let $U_n$ be open and relatively compact in $D_n$ and let $K_n$ be a compact subset of $U_n$. Suppose that $f$, when restricted to $D_n$, is conformal, and suppose that the modulus of $D_n \setminus \overline{U_n}$ is bounded below by some fixed constant $\delta > 0$ that does not depend on $n$. Then there exists a single constant $C$ so that equation (\ref{Kobe_Ineq}) holds for all pairs of domains $U_n$ and $K_n$. A similar assertion is true for (\ref{shape_ineq}).
\end{rem}

We now recall some basic facts about Hausdorff dimension, following \cite{Mattila}.
\begin{definition}
	\label{Hausdorff}
	Let $A \subset \C$, be a set. We define the \textit{$\alpha$-Hausdorff measure} of $A$ to be the quantity
	\begin{equation}
	\label{H_measure}
	H^{\alpha}(A) := \lim_{\delta \ra 0} H_{\delta}^{\alpha} (A) := \lim_{\delta \ra 0}\left( \inf \left \{\sum_{i=1}^{\infty} \diam(U_i)^{\alpha} \,: \, A \subset \bigcup_{i=1}^{\infty} U_i, \, \diam(U_i) < \delta\right \}\right),
	\end{equation}
	where the infimum is taken over all countable covers of $A$ by sets $\{U_i\}_{i=1}^{\infty}$.
\end{definition}

\noindent One can verify by directly using Definition \ref{Hausdorff} that if $H^{t}(A) < \infty$, then $H^{s}(A) = 0$ for all $s > t$, and similarly, if $H^{t}(A) > 0$, then $H^s(A)  = \infty$ for all $s < t$. Therefore, the following definition is well defined.

\begin{definition}
	\label{Hdim}
	Let $A \subset \C$ be a set. The \textit{Hausdorff dimension} of $A$ is defined to be
	\begin{equation}
	\label{HDim_Def}
	\dim_H(A) := \sup \{t: \, H^t(A) = \infty\} = \inf\{t: H^t(A) = 0\}.
	\end{equation}
\end{definition} 
\noindent We also use the following well-known facts about Hausdorff dimension.
\begin{lem}
	\label{Hdim_Facts}
	Let $A \subset \C$ be a set, and let $s \geq 0$. Then:
	\begin{enumerate}
		\item $H^s(A) = 0$ if and only if for all $\varepsilon > 0$, there exists sets $E_i \subset \C$, $i =1,2,\dots$ such that $A \subset \cup_{i=1}^{\infty} E_i$ and 
		\begin{equation}
		\label{estimate_sum}
		\sum_{i=1}^{\infty} \diam(E_i)^s < \varepsilon.
		\end{equation}
		\item Suppose that $A = \cup_{i=1}^{\infty} A_i$ for some sets $A_i \subset \C$. Then
		\begin{equation}
		\label{countable_stability}
		\dim_H(A) = \dim_H(\cup_{i=1}^{\infty} A_i) = \sup_{i \geq 1} \dim_H(A_i).
		\end{equation}
		\item Let $S \subset \C$ be a set and let $f: \C \rightarrow \C$ be an entire function. Then 
		\begin{equation}
		\label{dimension_entire}
		\dim_H(S) = \dim_H(f(S)) = \dim_H(f^{-1}(S)).
		\end{equation}
	\end{enumerate}	
\end{lem}

We will now record some useful lemmas about branched coverings that are topological in nature. The following are Propositions 3.1 and 3.2 of \cite{RempeSixsmith}.

\begin{lem}
	\label{preimage_helper}
	Let $f:X \rightarrow Y$ be a branched covering map between two non-compact, simply connected Riemann surfaces. Suppose that $U \subset Y$ is a simply connected domain and let $U'$ be a connected component of $f^{-1}(U)$ such that $U'$ contains only finitely many critical points of $f$. Then $f: U' \rightarrow U$ is a proper map, and $U'$ is simply connected. Additionally, if the boundary of $U$ is a Jordan curve in $Y$ that contains no critical values of $f$, then the boundary of $U'$ is a Jordan curve in $X$.
\end{lem}

\begin{lem}
	\label{branched_covering_helper}
	Suppose that $f: \C \rightarrow \C$ is an entire function, and suppose that $U \subset \C$ is simply connected. Suppose that $U$ contains no asymptotic values of $f$ and that the critical values of $f$ inside of $U$ form a discrete set. Then $f$ is a branched covering from each connected component of $f^{-1}(U)$ onto $U$.
\end{lem}

\noindent The following version of the Riemann-Hurwitz formula is from \cite{RiemannHurwitz}.

\begin{thm}
	\label{RH}
	Let $V$ and $W$ be domains in $\C$ and suppose the connectivity (the number of complementary components) of $V$ is $m$ and the connectivity of $W$ is $n$. Let $f:V \rightarrow W$ be a proper, branched covering map of degree $k$ that has $r$ many critical points. Then 
	\begin{equation}
	\label{RH_eqn}
	m-2 = k(n-2) + r
	\end{equation} 
\end{thm}

\noindent We also make use of polynomial-like mappings, see \cite{Hub}.
\begin{definition}
	\label{polynomial-like}
	Let  $\Omega$, $\Omega' \subset \C$ be Jordan domains and suppose that $\Omega$ is compactly contained inside of $\Omega'$. A holomorphic mapping $f: \Omega \rightarrow \Omega'$ is called a \textit{degree $d$ polynomial-like mapping} if it is a proper, degree $d$, branched covering map. Given a polynomial-like mapping, we denote its \textit{filled Julia set} by 
	$$K_f = \bigcap_{n=1}^{\infty} f^{-n}(\Omega).$$
\end{definition}

\noindent We make frequent use out of the following Lemma.

\begin{lem}
	\label{component_lemma}
	Suppose that $f:X \rightarrow Y$ is a degree $d$ branched covering map between two simply connected planar domains with only finitely many critical points. Let $U \subset Y$ be a Jordan domain. Suppose that $\overline{U}$ does not contain any critical values of $f$. Then there are $d$ many connected components of $f^{-1}(U) \subset X$, each of which is a Jordan domain that is mapped conformally onto $U$ by $f$.
\end{lem}
\begin{proof}
	Since $f: X \rightarrow Y$ only has finitely many critical points in $X$, every connected component $U'$ of $f^{-1}(U)$ is a Jordan domain and $f: U' \rightarrow U$ is proper, finitely branched covering map by Lemma \ref{preimage_helper} and Lemma \ref{branched_covering_helper}. Since $\overline{U}$ contains no critical values of $f:X \rightarrow Y$, the mapping $f: U' \rightarrow U$ has no critical points. Since $U'$ and $U$ are each Jordan domains, it follows from Theorem \ref{RH} that $f: U' \rightarrow U$ is conformal. Since $f$ is degree $d$, it follows that we must have $d$ many connected components of $f^{-1}(U)$.
\end{proof}

Next, we state the main result of \cite{BurLaz}, which is central to the proof of Theorem \ref{main_theorem}. We refer the reader to \cite{BurLaz} for a detailed discussion and proof. 

\begin{definition}\label{permissible} Let $(M_j)_{j=1}^\infty \in \mathbb{N}$ be increasing, and $(r_j)_{j=1}^\infty \in \mathbb{R}^+$. We say that $(M_j)_{j=1}^\infty$, $(r_j)_{j=1}^\infty$ are \emph{permissible} if \begin{gather} r_{j+1} \geq  \exp\left(\pi\big/M_j\right) \cdot r_j \textrm{ for all } j\in\mathbb{N}\textrm{, } r_j\xrightarrow{j\rightarrow\infty}\infty\textrm{, and } \sup_j\frac{M_{j+1}}{M_j}<\infty.  \end{gather} 
\end{definition}

\begin{thm}\label{mainthm} Let  $(M_j)_{j=1}^\infty$, $(r_j)_{j=1}^\infty$ be permissible, $r_0:=0$ and $c\in\mathbb{C}^\star:=\mathbb{C}\setminus\{0\}$.  Set
	\begin{equation}\label{mainthm_c}  c_1:=c\textrm{, and } c_j:=c_{j-1}\cdot r_{j-1}^{M_{j-1}-M_{j}} = c \cdot \prod_{k=2}^{j} r^{M_{k-1} - M_{k}}_{k-1} \textrm{ for } j\geq2.   \end{equation} 
	Then there exists an entire function $f: \mathbb{C} \rightarrow\mathbb{C}$ and a quasiconformal homeomorphism $\phi: \mathbb{C} \rightarrow \mathbb{C}$ such that \begin{gather} f\circ\phi(z)=c_jz^{M_j} \textrm{ for } r_{j-1}\cdot\exp(\pi/M_{j-1}) \leq |z| \leq r_j\textrm{, } j\in\mathbb{N}.  \end{gather} Moreover, if $\sum_{j=1}^\infty M_j^{-1}<\infty$, then $|\phi(z)/z - 1|\rightarrow 0$ as $z\rightarrow\infty$. The only singular values of $f$ are the critical values $(\pm c_jr_j^{M_j})_{j=1}^\infty$.
\end{thm}

Finally, we record the following important Lemma which is used in Section \ref{Expanding Zeros}. 
\begin{lem}
	\label{small_ball}
	Let $g_{n,2n}$ be the function in Proposition 3.13 of \cite{BurLaz}, and let $w$ be a zero of $g$ contained inside $A(1,\exp(\pi/n)).$ There exists constants $0 < \lambda < 1/8$ and $\delta > 0$, which do not depend on $n$, so that 
	\begin{equation}\label{both_inclusions} B(0,\delta) \subset g_{n,2n}(B(w,\lambda(\exp(\pi/n)-1))) \subset B(0,1/2). \end{equation}
	Moreover, $g_{n,2n}$ is injective on $B(w,\lambda(\exp(\pi/n)-1))$.
\end{lem}
\begin{proof} It is possible to prove this result directly from the definition of $g_{n,2n}$ (Definition 3.11 of \cite{BurLaz}), but it will be more straightforward if we use some general results about quasiconformal mappings. We will let $B_\lambda:=B(w,\lambda(\exp(\pi/n)-1))$, and denote the Jacobian of $g$ by $J_g$.
	
	Let $w$ be a zero of $g$ contained inside $A(1,\exp(\pi/n))$. It follows from Definition 3.11 of \cite{BurLaz} that $g$ is injective (and hence quasiconformal) in \[ B\left(w,\frac{\exp(\pi/2n)-1}{2}\right). \] We first show that the first inclusion in (\ref{both_inclusions}) holds for small $\lambda$. To this end, we appeal to Theorem 1.8 of \cite{MR777305} which implies that there is a constant $c$ depending only on $K(g_{n,2n})$ (in particular $c$ does not depend on $n$ or $\lambda$) so that: \begin{equation}\label{both_inclusions_proof} d(g(w), \partial g(B_\lambda)) \geq c \cdot \lambda(\exp(\pi/2n)-1) \cdot \exp\left( \frac{1}{2m(B_\lambda)}\int_{B_\lambda}\log(J_g) \right). \end{equation} It is readily calculated that there is a constant $C$ independent of $n$ and $\lambda$ such that $J_g(z) \geq C\cdot n^2$ for $z\in B_\lambda$. Thus from (\ref{both_inclusions_proof}) we conclude that \[ d(g(w), \partial g(B_\lambda)) \geq c\lambda(\exp(\pi/2n)-1)Cn \geq c\lambda\cdot\pi/2n\cdot Cn = c\lambda\pi C/2 \] We conclude that the first inclusion in (\ref{both_inclusions}) holds for $\delta$ which depends only on $\lambda$. 
	
	Next we show that for sufficiently small $\lambda$ the second inclusion in (\ref{both_inclusions}) also holds. Indeed, since quasiconformal mappings are quasisymmetric (see Theorem 3.6.2 of \cite{AstalaIwaniecMartin}), there exists a constant $\eta$ depending only on $K(g_{n,2n})$ (and in particular $\eta$ does not depend on $n$) such that for all $\lambda\leq(\exp(\pi/2n)-1)/4$ we have: \begin{equation}\label{quasisymmetry} \sup_{\theta\in[0,2\pi]} \frac{\left|g(w)-g(w+e^{i\theta}\lambda(\exp(\pi/n)-1))\right|}{\left| g(w) - g(w+\lambda(\exp(\pi/n)-1)))\right|} \leq \eta. \end{equation} It is readily seen from the definition of $g_{n, 2n}$ that \[ \sup_n\left| g(w) - g(w+\lambda(\exp(\pi/n)-1)))\right| \xrightarrow{\lambda\rightarrow0} 0, \] so that by (\ref{quasisymmetry}) we conclude that the second inclusion in (\ref{both_inclusions}) holds for all sufficiently small $\lambda$.
\end{proof}

\begin{lem}
	\label{small_ball_rescaled}
	Let $g = g_{n,2n,x,c}$ be the function in Proposition 3.19 of \cite{BurLaz}, and let $w$ be a zero of $g$ contained inside $A(x,\exp(\pi/n)\cdot x).$ There exists constants $0 < \lambda < 1/8$ and $\delta > 0$, which do not depend on $n$ or $x$, so that 
	\begin{equation}
	\label{both_inclusions2}
	B(0,\delta \cdot c x^n) \subset g(B(w,\lambda(\exp(\pi/n)-1)x)) \subset B(0,1/2 \cdot c x^n).
	\end{equation}
	Moreover, $g$ is injective on $B(w,\lambda(\exp(\pi/n)-1)x)$.
\end{lem}
\begin{proof}
	This follows immediately from the definition of $g_{n,2n,x}$ and Lemma \ref{small_ball}. Indeed, we have 
	$$g_{n,2n,x,c} = (z \mapsto c x^n z) \circ g_{n,2n} \circ (z\mapsto \frac{z}{x}).$$
	The inclusions (\ref{both_inclusions2}) now follow from (\ref{both_inclusions})
\end{proof}

\bibliographystyle{alpha}
\bibliography{bibfile_dimone}

\end{document}